\documentclass[11pt]{amsart}
\usepackage{amssymb, amsmath}
\usepackage{amsthm}
\usepackage{amscd}
\usepackage{graphicx}
\usepackage[expansion=false]{microtype}
\newtheorem{theorem}{Theorem}
\newtheorem{lemma}[theorem]{Lemma}

\newtheorem{corollary}[theorem]{Corollary}
\newtheorem{proposition}[theorem]{Proposition}

\theoremstyle{definition}

\newtheorem*{remark}{Remark}
\newtheorem*{acknowledgement}{Acknowledgement}

\title[Collapse of ends of manifolds of nonpositive curvature]{Half dimensional collapse of ends of manifolds of nonpositive curvature}
\author{Grigori Avramidi, T. T$\hat{\mathrm{a}}$m Nguy$\tilde{\hat{\mathrm{e}}}$n-Phan}
\address{Mathematische Institut\\Universit\"at M\"unster\\ Germany}
\email{avramidi@uni-muenster.de}
\address{Max Planck Institut f\"ur Mathematik\\ Bonn\\ Germany}
\email{tam@mpim-bonn.mpg.de}

\def\ra{\rightarrow}

\def\beqa{\begin{eqnarray}}
\def\eeqa{\end{eqnarray}}
\def\beqa{\begin{eqnarray}}
\def\eeqa{\end{eqnarray}}

\DeclareMathOperator{\im}{Im}

\DeclareMathOperator{\SL}{SL}
\DeclareMathOperator{\SO}{SO}

\DeclareMathOperator{\Fix}{Fix}

\DeclareMathOperator{\Td}{Td}

\def\R{\mathbb{R}}
\def\Z{\mathbb{Z}}

\def\rank{\mbox{rank}}

\usepackage{pict2e}

\DeclareRobustCommand{\lowerlefttriangle}{%
  \begingroup
  \setlength{\unitlength}{1ex}%
  \begin{picture}(2,2)
  \polyline(2,0)(0,0)(0,2)(2,0)(0.5,0)
  \end{picture}%
  \endgroup
}

\input xy
\xyoption{all}
\begin{document}

\begin{abstract}
This paper accomplishes two things. First, we construct a geometric analog of the rational Tits building for general noncompact, complete, finite volume $n$-manifolds $M$ of bounded nonpositive curvature. Second, we prove that this analog has dimension less than $\lfloor n/2\rfloor $. 
\end{abstract}
\maketitle

\section{Introduction}
Let $M$ be a noncompact, complete Riemannian $n$-manifold with bounded nonpositive sectional curvature $-1 < K\leq 0$ and finite volume\footnote{In fact, all our results hold with ``finite volume'' replaced by ``injectivity radius $\ra 0$''.}. We also assume that $M$ does not have arbitrarily small geodesic loops. Good examples to think about are locally symmetric spaces of noncompact type, such as hyperbolic manifolds, products of surfaces, and the usual beloved $K\backslash G/\Gamma$. Sometimes taking $G = \SL_m\R$ and $K = \SO_m$ can be as satisfactory as any other semisimple Lie groups. This sentiment holds true in terms of examples to keep in mind as one reads  since our approach throughout this paper is purely geometric/topological but can be demonstrated by thinking about these concrete examples the right way. 
\newline

The condition that $M$ has no arbitrarily small geodesic loops holds when $M$ is negatively curved, i.e $-1<K<0$, or when $M$ is locally symmetric. We need this condition in the general setting of bounded nonpositive curvature in order to insure, by a theorem of Gromov-Schroeder, that $M$ is \emph{tame} in the sense that the thin part $M_{< \epsilon}$ has finitely many components and each component is topologically a product of a closed $(n-1)$-manifold with a ray. The mechanism for tameness is that the injectivity radius function on $M$ does not have any critical point outside a compact set which can be taken to be the thick part $M_{> \epsilon}$ for some \emph{small} $\epsilon >0$ (see also Appendix 2 of \cite{ballmangromovschroeder} for a generalization). Let $\widetilde{M}_{<\epsilon}$ be a lift of the thin part $M_{<\epsilon}$ in the universal cover $\widetilde{M}$. We will call $\widetilde{M}_{<\epsilon}$ the thin part of $\widetilde{M}$. It is the topology of $\widetilde{M}_{<\epsilon}$ that we would like to describe.
\newline

When $M$ is locally symmetric (and arithmetic), $\widetilde{M}_{<\epsilon}$ is homotopy equivalent to the rational Tits building, which is a $(k-1)$-dimensional complex, where $k$ is the rational rank of $M$. The rational Tits building of $M$ can be realized as a subset of the visual boundary $\partial_\infty$ of $\widetilde{M}$ and can be thought of as the set of of points at infinity that can be reached if one moves only within $\widetilde{M}_{<\epsilon}$. The rank $k$ is at most $n/2$ with $k=n/2$ when $M$ is a product of non-compact surfaces. The main purpose of this paper is to show that the $n/2$ bound on the dimension of the rational Tits building is no arithmetic coincidence but in a slightly weaker sense. 
\newline

In the general nonpositively curved setting, for an $n$-dimensional manifold $M$ satisfying the conditions described above, we define a map
\[ \rho \colon \partial\widetilde{M}_{<\epsilon} \rightarrow \partial_\infty\]
that is an analog of a rational Tits building in the sense that $\rho$ encodes all the directions to infinity necessary to push any topological feature (e.g. homology cycles, maps) in $\widetilde{M}_{<\epsilon}$ without it leaving $\widetilde{M}_{<\epsilon}$. We then prove that the image of $\rho$ has dimension at most $(\lfloor n/2 \rfloor -1)$, where $\lfloor n/2 \rfloor$ is the greatest integer less than or equal to $n/2$.

\begin{theorem}\label{analog}


Let $M$ be a noncompact, complete, Riemannian manifold with bounded nonpositive sectional curvature $-1 < K\leq 0$ and finite volume. Assume that $M$ has no arbitrarily small geodesic loops. Let $\epsilon > 0$ be smaller than the Margulis constant and small enough so that $M_{<2\epsilon}$ is topologically a product with a ray. Then there is a $\pi_1(M)$-equivariant, Lipschitz map $\rho \colon \partial\widetilde{M}_{<\epsilon} \rightarrow \partial_\infty$, defined on a triangulation of $\partial\widetilde{M}_{<\epsilon}$, with the following properties.


\begin{itemize}
\item[a)] For each $x\in\partial\widetilde{M}_{<\epsilon}$, the unit speed geodesic ray $[x,\rho(x))$ connecting $x$ to $\rho (x)$ stays in $\widetilde{M}_{<2\epsilon}$. Moreover, the projection of $[x,\rho(x))$ to $M$ leaves all compact sets in $M$.
\item[b)] If $\sigma$ is a simplex in $\partial\widetilde{M}_{<\epsilon}$, then $\rho (\sigma)$ has dimension less than $\lfloor n/2\rfloor$. 
\end{itemize}
\end{theorem}

Consequently, we can use $\rho$ to show that  any polyhedron in  $\widetilde{M}_{<\epsilon}$ can be homotoped within  $\widetilde{M}_{<\epsilon}$ to one with dimension at most $(\lfloor n/2 \rfloor -1)$.

\begin{theorem}
\label{factor}
Assume the hypotheses of Theorem \ref{analog}. Let $P$ be a finite polyhedron and let $\varphi \colon P \rightarrow \widetilde{M}_{<\epsilon}$ be a continuous map. Then $\varphi$ can be homotoped within $\widetilde{M}_{<\epsilon}$ to a map $\widehat{\varphi} \colon  P \rightarrow \widetilde{M}_{<\epsilon}$ whose image has dimension $\leq\lfloor n/2\rfloor-1$. 
\end{theorem}

This is done by pushing $P$ toward $\rho (P)$ until it is deep enough in $\widetilde{M}_{<\epsilon}$ that we can ``collapse" $P$ onto a close-by copy of $\rho (P)$. The following corollary is an almost immediate consequence.
\begin{corollary}\label{homology of the thin part}
The map $\widehat{\varphi}$ in Theorem \ref{factor} can be homotoped in $\widetilde M_{<\epsilon}$ to factor through a polyhedron $Q$ of dimension $\lfloor n/2\rfloor-1$. 
Consequently, the homology of $\widetilde{M}_{<\epsilon}$ vanishes in dimension  $\geq\lfloor n/2\rfloor$, i.e.
\begin{equation}
\label{homology}
H_{\geq\lfloor n/2\rfloor}(\widetilde M_{<\epsilon})=0.
\end{equation}
\end{corollary}


\begin{remark}
The upper bound on the dimension of $Q$ is sharp by the following example. If $M$ is the product of $k$ hyperbolic punctured tori, then $M$ has dimension $n = 2k$, so $\lfloor n/2\rfloor -1 = k-1$. We also know that in this case $\widetilde{M}_{<\epsilon}$ is homotopy equivalent to a wedge of $(k-1)$-spheres. 
\end{remark}
The intuition 
behind all of this is that we push any topological feature, such as a polyhedron $P$, to infinity within $\widetilde{M}_{<\epsilon}$ without being too stupid in the way we push it. Note that we can always push anything in $\widetilde{M}_{<\epsilon}$ to infinity since $M$ has tame ends, but we want to push $P$ in such a way not to stretch it more than we absolutely have to. The number of degrees of freedom in stretching $P$ is the dimension of $\rho (P)$. 
This is why we build $\rho$ and why we need to make it as low dimensional as possible. In a way, the topology of $\widetilde{M}_{<\epsilon}$ is very much determined by $\rho$.

\subsection*{Optimal examples}
In \cite{avramidinguyenphanbuildingends} we build, for each $n$, an $n$-manifold $M$ satisfying the hypotheses of Theorem \ref{analog} which has 
$$
\overline H_k(\widetilde M_{<\epsilon})\not=0\mbox{ for all }k<\lfloor n/2\rfloor.
$$
So, unlike in the case of locally symmetric spaces, in general there are no low dimensional homology vanishing results complementing the high dimensional homology vanishing (\ref{homology}).

\bigskip
\subsection*{Some other forms of collapse}
Of course in some situations we expect to be able to do better than $\lfloor n/2\rfloor$. For example, in the simple case of a finite volume hyperbolic manifold, every component of the thin part $\widetilde M_{<\epsilon}$ collapses to a point. The feature responsible for this collapse is that the Tits boundary is discrete.\footnote{Ends of manifolds with for which $(\partial_{\infty},\Td)$ is discrete were studied by Eberlein in \cite{eberleinvisibility}.} In fact, the topological dimension of the Tits boundary $(\partial_{\infty},\Td)$ of $\widetilde{M}$ is one of the factors that control the topology of $\widetilde{M}_{<\epsilon}$. This is reflected in the fact that the map $\rho$ we construct is continuous (in fact, Lipschitz) in the Tits metric. In addition, $\rho$ is constructed in such a  way that it factors through a complex 
built out of virtual equivalence classes of certain abelian subgroups of $\pi_1M$. So, the topology of this complex is another factor that controls the topology of $\widetilde M_{<\epsilon}$. This leads to two additional forms of collapse via ranks of abelian groups and topological dimension of the Tits boundary. To express it, let
$$
d=\min\{\rank_{Ab}(\pi_1M)-1,\hspace{0.2cm}\dim(\partial_{\infty},\Td),\hspace{0.2cm}
\lfloor n/2\rfloor-1\}
$$
where 
\begin{itemize}
\item
$\rank_{Ab}
(\pi_1M)$ is the maximum rank of an abelian subgroup of $\pi_1M$, 
\item
$\dim(\partial_{\infty},\Td)$ is the topological dimension\footnote{See \cite{kleiner} for a comparison of different notions of dimension for the Tits boundary.} of the Tits boundary, and
\item
$\lfloor n/2\rfloor$ still denotes the greatest integer less than or equal to $n/2$.
\end{itemize} 
\begin{theorem}
\label{other}
Theorem \ref{factor} is true if we replace ``$\lfloor n/2\rfloor-1$" by ``$d$". 
\end{theorem} 
\begin{remark}
If $M$ is a noncompact, finite volume hyperbolic $n$-manifold, with several cusps then the fundamental group $\Gamma$ contains a parabolic abelian subgroup of rank $n-1$. Doubling such a manifold along one of its cusps gives a non-positively curved manifold containing a {\it hyperbolic} abelian subgroup of rank $n-1$, in which case $\dim(\partial_{\infty},\Td)\geq n-2$. So, the bounds via ranks of abelian groups and via the dimension of the Tits boundary are situational. But, the half dimension bound $d\leq\lfloor n/2\rfloor-1$ is something we always have. 
\end{remark}

\subsection*{Low dimensional collapse} Now let us turn to the special situation when $d$ is small. 
Notice that $d\leq 1$ if
\begin{itemize}
\item
$\pi_1M$ does not contain $\mathbb Z^3$, or
\item
$\dim(\partial_{\infty},\Td)\leq 1$, or
\item
$\dim M\leq 5$. 
\end{itemize}
In any of these situations we get the 
following immediate corollary.
\begin{corollary}
\label{asphericalends}
If $d\leq 1$ then each component of the end $M_{<\epsilon}$ is aspherical.
\end{corollary}
Once we know that each component of the end is aspherical, it is natural to ask ``how bad'' this aspherical manifold can be, i.e. how troubling its fundamental group is. We prove the following amplification of Corollary \ref{asphericalends}.
\begin{theorem}
\label{locallyfreeintro}
If $d\leq 1$ then for each component $C$ of $M_{<\epsilon}$ the fundamental group is an extension
$$
1\ra F\ra\pi_1 C\ra\pi_1M
$$
of a subgroup of $\pi_1M$ by a {\it locally free}\footnote{A countable group $F$ is locally free if every finitely generated subgroup is free. In other words, $F=\cup_iF_{n_i}$ is a union of finitely generated free groups that are included in each other via possibly complicated inclusions $F_{n_1}\hookrightarrow F_{n_2}\hookrightarrow\dots$} group $F$.
\end{theorem}
This is enough information about the fundamental group of the end to get some applications. For instance (via the method of \cite{blockweinberger}) one gets
\begin{corollary}\label{nopsccor}
If $d\leq 1$ then $M$ does not have a complete Riemannian metric which has uniformly positive scalar curvature.
\end{corollary}

In proving the main theorems we obtain the following. 
\newline

\noindent
\textbf{Technical byproduct of independent interest.} Let $G$ be a group that acts on a Hadamard\footnote{A Hadamard manifold is a simply connected, complete manifold of nonpositive (not necessarily bounded) sectional curvature.} $n$-manifold $\widetilde{M}$ via covering space transformations. Suppose that $G$ preserves horospheres centered at $\{z_i\}_{i = 0,..., k} \subset \partial_\infty$. We ask the following questions. 
\begin{itemize}
\item[1.] Can one connect $z_i$'s in $\partial_\infty$ through points in $\partial_\infty$ whose horospheres are preserved by $G$?
\item[2.] If so, then one gets a map
\[ \sigma \colon \Delta^k \rightarrow \partial_\infty,\]
where $\Delta^k$ is the standard $k$-simplex, such that horospheres centered at each point in $\sigma (\Delta^k)$ are preserved. How nice (e.g. continuous, Hoelder, Lipschitz, etc) can this map be? 
\item[3.] Let $\Fix^0(G)$ be the set points in $\partial_\infty$ whose horospheres are preserved by $G$. What is the relation between the dimension of $\Fix^0 (G)$ and the dimension of $G$?
\end{itemize}
We answer these questions in the case when the vertices $z_i$ are mutually a Tits distance $\leq\pi/2$ apart in Sections \ref{gradientsection} and \ref{dimensionboundsection}. In short, the answer to the first question is yes, the answer to the second question is Lipschitz -- we construct such a map $\sigma$, which we call a \emph{Busemann simplex} -- and the answer to the third question is the following. 

\begin{theorem}
\label{gdimensionboundintro}
If the vertices $z_i$ span a non-degenerate Busemann $k$-simplex, then the homological dimension of $G$ is less than $(n-k)$.
\end{theorem}

\noindent
\textbf{On Section \ref{problems and solutions}.} We motivate the proof of Theorem \ref{factor} in Section \ref{problems and solutions}. Theorem \ref{analog} will be attained along the way. We will try to explain why things are done the way they are through an iteration of ``what is the simplest thing to do?" and ``what are the problems to overcome?" until there are no more problems. In times of trouble, good things to think about are locally symmetric spaces, in particular the examples of $\SO_3\backslash\SL_3\R/\SL_3\Z$ and products of surfaces. Some readers might find that Section \ref{problems and solutions} is ``madness", in which case they are encouraged to skip it to the precise formulation given in the rest of the paper.

\begin{acknowledgement}
We would like to thank Igor Belegradek for comments on earlier versions of this paper. The first author would like to thank the University of Muenster and the second author would like to thank the Max Planck Institute for Mathematics for their support and excellent working conditions.
\end{acknowledgement}

\section{Problems and solutions}\label{problems and solutions}

Let $P$ be as in Theorem \ref{factor}. One can naively take a length minimizing geodesic ray $\gamma$ in $M$ starting at a point in $\varphi(P)$, take a lift $\widetilde{\gamma}$ of $\gamma$, and then push $P$ toward $\widetilde{\gamma} (\infty) \in\partial_\infty$ with unit speed. Then the diameter of $P$ will stay bounded, so once it is far enough to infinity it will be contained in a ball that is contained in $\widetilde{M}_{<\epsilon}$. We then can contract $P$ to a point within this ball. However, there is a problem with this approach, which is that as we push $P$ toward $\widetilde{\gamma}(\infty)$ it might slide off $\widetilde{M}_{<\epsilon}$ for some time during this process. This problem does not have a solution for otherwise one could contract any such $P$ to a point within $\widetilde{M}_{<\epsilon}$, which is not true if $M$ is a product of noncompact surfaces.
\newline

So we need to find a way to push $P$ to infinity without it leaving $\widetilde{M}_{<\epsilon}$. The strategy is that we push different points of $P$ to different points  in the visual boundary $\partial_\infty$ along geodesics and we keep track of the amount of directions to infinity we need. The set of points in $\partial_\infty$  to which we push $P$ tells us how much $P$ ``expands" as we push it to infinity. It also gives a complex $c_t(P)$ in $\widetilde{M}_{<\epsilon}$ to which we can ``collapse" $P$ onto. We then bound the dimension of this complex to be less than $\lfloor n/2\rfloor$ by trying to make this process as efficient (in terms of how many degrees of freedom are needed as $P$ expands) as possible. 
\newline

\noindent
\textbf{Keeping the homotopy within the thin part $\widetilde{M}_{<\epsilon}$.}
For each point $x$ of the polyhedron $P$, we find in $\partial_\infty$ a point $\rho (x)$ to which we push $x$ with unit speed along the geodesic connecting $x$ with $\rho(x)$ as illustrated in Figure 1.  We will define $\rho \colon P \rightarrow \partial_\infty$ systematically, skeleton by skeleton. We call this homotopy 
\[\varphi_t \colon P \rightarrow \widetilde{M}_{<\epsilon}, \quad \text{for} \quad t\in [0,\infty ), \quad \text{with} \quad \varphi_0 = \varphi.\]

Start with the vertices of $P$ and suppose that $x$ is a vertex of $P$. Since $\varphi(x)$ is in  $\widetilde{M}_{<\epsilon}$, there is a parabolic isometry $\gamma_x$ that moves $\varphi(x)$ by a \emph{small} amount, where \emph{small} means less than $\epsilon$, so that the group $\Gamma_x$ generated by such $\gamma_x$ is virtually nilpotent by the Margulis lemma. Therefore, a natural choice\footnote{Note that the choice of $\rho (x)$ may not be unique. Think about products of surfaces.} for $\rho (x)$ is the center of a horosphere preserved by $\Gamma_x$ because if we push $x$ to $\rho (x)$ along a geodesic $\varphi_t(x)$,  then the \emph{small} elements in $\Gamma_x$ will remain \emph{small} along $\varphi_t(x)$, so $\varphi_t(x)$ will stay in  $\widetilde{M}_{<\epsilon}$. 
\newline

\begin{figure}
\label{pushfigure}
\centering
\includegraphics[scale=0.1]{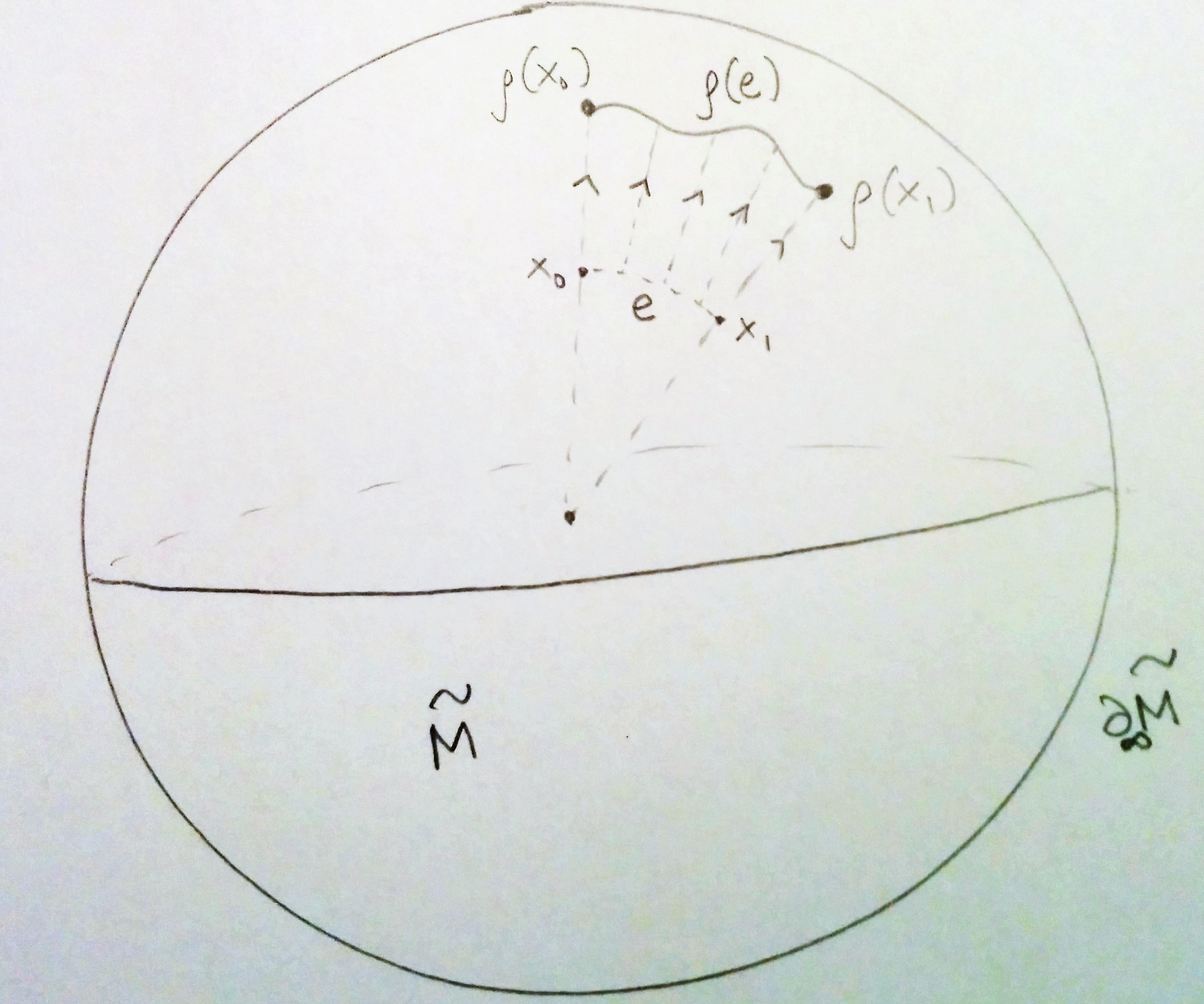}
\caption{}
\end{figure}

Next, we extend $\rho$ to the edges and higher dimensional simplices of $P$. Just like for the vertices, the way to go to infinity is to ``follow the shrinking small loops". Let $e$ be the edge connecting vertices $x_0$ and $x_1$ of $P$. Clearly  a p\underline{roblem} is that there is no clear ``transition" in terms of \emph{small} loops at $\varphi(x_0)$ to \emph{small} loops at $\varphi(x_1)$. But there is a \underline{solution}, which is to take a fine enough subdivision of $P$ at the beginning. That is, we take a subdivision of $P$ in which the diameter of each simplex is tiny enough so that if  $\sigma$ is a $k$-simplex of $P$ with vertices $x_0, x_1, ..., x_k$, then for each point $y \in \sigma$, some parabolic isometry that is \emph{small} at $\varphi(x_i)$ is still \emph{small} at $\varphi(y)$. 
We can always take such a subdivision of $P$, so we can harmlessly assume that $P$ is triangulated in such a way at the beginning. 
\newline

This gives us a way to assign to each simplex a nontrivial nilpotent group as follows. Let $\Gamma = \pi_1(M)$. For a vertex $x$ of $P$, let 
\[S_x = \{ \gamma \in \Gamma \; |\; d(x,\gamma(x)) <\epsilon\}\] be the set of $\epsilon$-small parabolic isometries at $\varphi(x)$. Then by the Margulis lemma, the group
\[ \Gamma_x = \langle \gamma \; |\; \gamma \in S_x \rangle,\]
has a nilpotent subgroup of index less than a constant $I_n$ that depends only on $n$. We assign to $x$ the following nontrivial nilpotent subgroup
\[ N_x = \langle \gamma^{I_n!} \; |\; \gamma \in \Gamma_x \rangle\]
of $\Gamma_x$. Note that $N_x$ is characteristic and thus normal in $\Gamma_x$. One might worry  with this choice of nilpotent group at $x$ the group $N_x$ may not contain any \emph{small} parabolic isometry. This can easily be fixed by making $\epsilon$ small enough at the beginning\footnote{It turns out, however, that we do not have to do so.}. For a $k$-simplex $\sigma = x_0*x_1*...*x_k$, let 
\[ \Gamma_\sigma = \langle \Gamma_{x_0}, ..., \Gamma_{x_k} \rangle.\]
We assign to $\sigma$ the nilpotent group 
\[ N_\sigma = \langle \gamma^{I_n!} \; |\; \gamma \in \Gamma_\sigma \rangle.\]
Let $Z_\sigma = Z_{N_\sigma}$ be the center of $N_\sigma$. Note that $Z_\sigma$ is normal in $\Gamma_\sigma$.
\newline

Since nilpotent groups are at times harder to deal with than abelian groups -- life is hard enough already -- we try to make things as easy as we can by assigning to the simplex $\sigma$ the abelian group $A_\sigma = Z_\sigma$. Again, if one is worried, one can pick $\epsilon$ to be small enough at the beginning so that $A_\sigma$ has \emph{small} elements. 
\newline 

Now the p\underline{roblem} with this is that this does not quite give us a way to define $\rho$ on an edge $e$ connecting $x_0$ and $x_1$ because what we get from $A_e$ is the center of a horosphere preserved by $A_e$ to which we can push $e$ without it leaving $\widetilde{M}_{<\epsilon}$. However, this center is only one point while what we are looking for is a path connecting $\rho(x_0)$ and $\rho(x_1)$. Nevertheless, what we have obtained is a way to define $\rho$ on the vertices of the first barycentric subdivision $P_1$ of $P$. Note that in this assignment of abelian groups adjacent vertices in $P_1$ have commuting abelian groups.
\newline

There will be a \underline{solution} to the above problem if the distance between adjacent vertices of $P_1$ in the  Tits metric $\Td$ on $\partial_\infty$ is less than $\pi$ because then there will be a unique geodesic in $(\partial_\infty, \Td)$ connecting them, so we can use this to define $\rho$ on the edges of $P_1$. This turns out to be true if we are picky enough when we pick where $\rho$ sends vertices of $P_1$. In fact, we can, for each simplex, make $\rho$ send all of its vertices to a set of Tits-diameter $\leq \pi/2$ in $\partial_\infty$. 
\newline

For each abelian group $A_x$, where $x$ is a vertex of $P_1$, let $\Fix(A_x)$ be the set of fixed points of $A_x$ in $\partial_\infty$. There is a canonical way (see Section \ref{goodpoints}) to define a unique ``\emph{Center of Mass}" $\xi_{A_x} \in \Fix (A_x)$ such that 
\begin{itemize}
\item any isometry $\gamma$ that normalizes $A_x$ fixes $\xi_{A_x}$, and
\item all points in $\Fix(A_x)$ are within a Tits distance of $\pi/2$ from $\xi_{A_x}$.
\end{itemize}
Define $\rho (x) = \xi_{A_x}$. Then by the first property above, adjacent abelian groups in $P_1$ fix each other's Centers of Mass. By the second property, they are within a Tits distance of $\pi/2$ from each other. It follows that any two adjacent vertices $x_0$ and $x_1$ in $P_1$ can be connected by a \emph{unique} geodesic $\rho (e)$ in $\partial_\infty$ connecting $\rho (x_0)$ and $\rho (x_1)$. Parametrize both $e$ and  $\rho (e)$ by constant speed and use this to define $\rho$ on $e$ the obvious way. We can extend $\rho$ to higher dimensional skeleta via geodesic triangles in the obvious way.  
\newline

Now that we have found a way to define $\rho$ on $P_1$, we need to check that the homotopy $\varphi_t$ does not push $P$ off the thin part $\widetilde{M}_{<\epsilon}$. For each $k$-simplex $\sigma$ in $P_1$, we have a chain $\Gamma_{x_0} \leq ... \leq \Gamma_{x_k}$. The ``bottom" group $\Gamma_{x_0}$ normalizes $A_{x_0},..., A_{x_k}$ and therefore fixes $\rho(x_i) = \xi_{A_{x_i}}$ for all $i =0, 1,...,k$. It follows that $\Gamma_{x_0}$ fixes $\rho(\sigma)$ \emph{pointwise}\footnote{by the uniqueness of the geodesics we use to connect vertices of $\rho(\sigma)$.}. Remember that there is an element $\gamma$ of $\Gamma_{x_0}$ that is \emph{small} at $x_0$ and that is still \emph{small} at all points $y \in \sigma$.  Since $\gamma$ fixes $\rho(y)$, it will stay \emph{small} on $\varphi_t(y)$ for all $t >0$, and therefore, the homotopy $\varphi_t$ does not move $P$ off $\widetilde{M}_{<\epsilon}$. Note that this shows we did not have to go back and manually make $\epsilon$ smaller at the beginning as one might have worried before.  Now, we can move on to the next task. 
\newline

\noindent
\textbf{``Collapsing" $P$ within $\widetilde{M}_{<\epsilon}$.} Now that we have defined $\rho$ of $P$ and made sure that pushing $P$ to $\rho (P)$ does not leave $\widetilde{M}_{<\epsilon}$, we want to find a copy of $\rho(P)$ in $\widetilde{M}$ to which we can ``collapse" $P$ within $\widetilde{M}_{<\epsilon}$. Take a point $c_0\in\widetilde M$ and take the geodesic cone on $\rho(P)$ with cone point $c_0$. For $t \geq 0$ and $x\in P$,  let $c_t(x)$ be the point obtained by flowing for time $t$ along the geodesic ray from $c_0$ to $\rho(x)$. Then $c_t(P)$ homeomorphic to $\rho (P)$ because geodesic retractions are homeomorphisms. Also, it is not hard to see that  the distance between $\varphi_t(P)$ and $c_t(P)$ is bounded by some number $R$ that does not depend on $t\geq 0$ (but depends on $P$ and $c_0$). Thus, we can ``collapse" $\varphi_t(P)$ onto $c_t(P)$ in an $R$-neighborhood of $\varphi_t(P)$.
\newline

There is a p\underline{roblem}, which is that the ``collapse" might leave $\widetilde{M}_{<\epsilon}$, which could happen if the $R$-neighborhood of $\varphi_t(P)$ is large enough that it contains points outside $\widetilde{M}_{<\epsilon}$. But there is a \underline{solution} if we can show that for $t$ large enough, $\varphi_{t_{\text{large}}}(P)$ is deep enough in  $\widetilde{M}_{<\epsilon}$ that an $R$-neighborhood of $\varphi_{t_{\text{large}}}(P)$ is contained in $\widetilde{M}_{<\epsilon}$, so when we collapse $\varphi_{t_{\text{large}}}(P)$ to $c_{t_{\text{large}}}(P)$ it will not leave $\widetilde{M}_{<\epsilon}$ during this process.  Therefore, in addition to making sure that $\varphi_t(\sigma)$ stays in $\widetilde{M}_{<\epsilon}$ for all $t >0$, we also need its projection under the covering space projection $p \colon \widetilde{M} \rightarrow M$ to be \emph{divergent} in $M$. That is, $p(\varphi_t(P))$ leaves all compact sets in $M$ as $t\rightarrow\infty$. This is true by the following key lemma.

\begin{figure}
\centering
\includegraphics[scale=0.07]{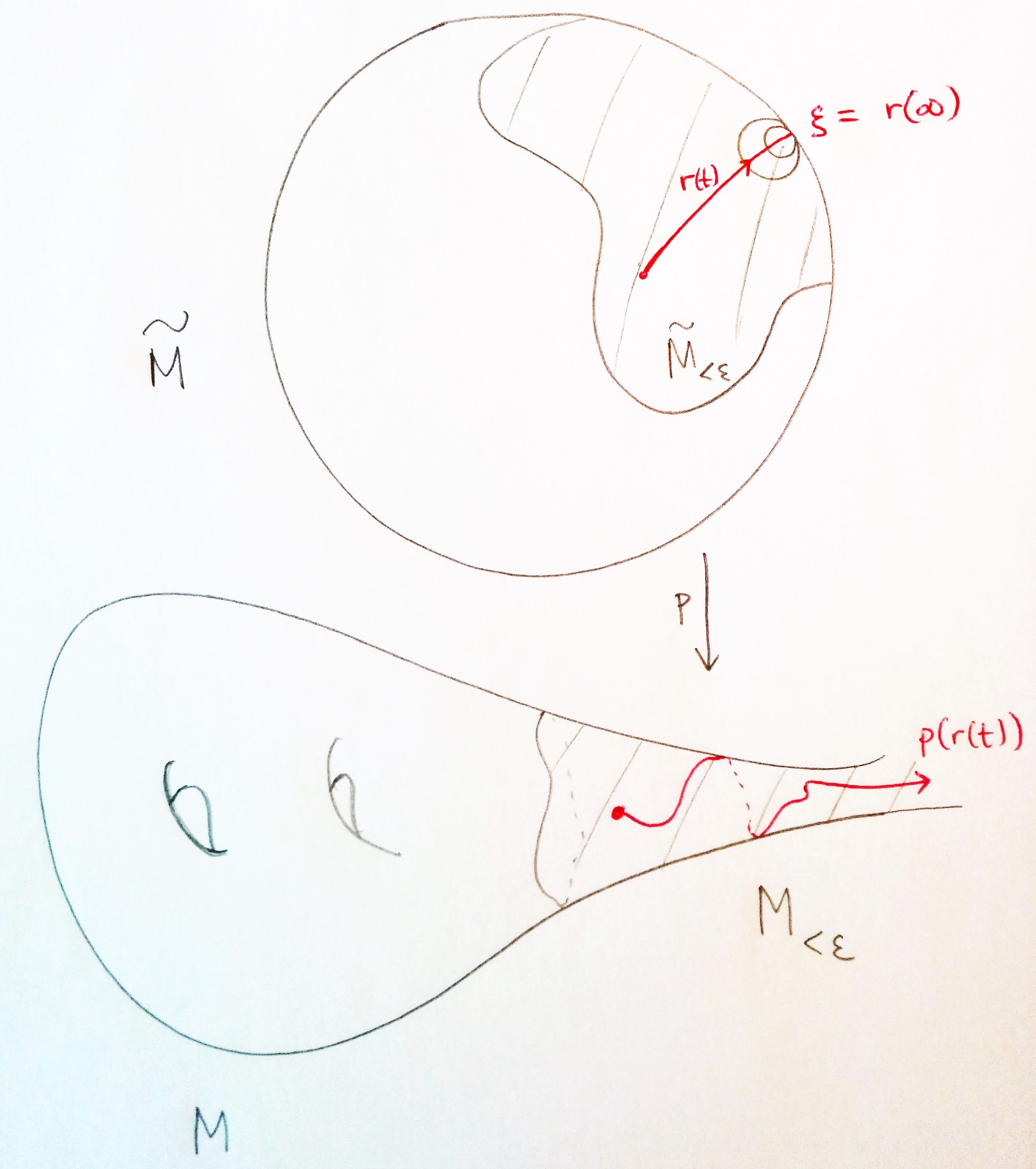}
\caption{}
\end{figure}

\begin{lemma}[Divergent Geodesic Ray]\label{div ray}
Let $A$ be a free abelian group of isometries. Suppose the centralizer $C_A$ preserves each horosphere centered at a point $\xi$ in $\partial_\infty\widetilde{M}$. Then for any geodesic ray $r \colon [0,\infty)\rightarrow \widetilde{M}$ with end point $r(\infty) = \xi$ the projection $p(r(t))$ is divergent. 

More generally, if $r(\infty)=\eta$ for some $\eta$ fixed by $C_A$ and $\Td(\eta,\xi)<\pi/2$ then $p(r(t))$ is also divergent. 
\end{lemma}

Lemma \ref{div ray} takes care of the above problem if for each abelian group $A_x$ above \emph{the horospheres centered at the Center of Mass $\xi_{A_x}$ are preserved by the centralizer $C_{A_x}$}. We prove that this is true in Section \ref{goodpoints}. If one is concerned that Lemma \ref{div ray} might apply to only one single ray $\varphi_t(x)$ at a time but not uniformly to a family $\varphi_t(P)$ of rays, then one is absolutely right, but we take care of this in Proposition \ref{divergentsectorsearly}, which says that one needs not worry if $P$ is bounded (and $P$ is indeed bounded). 
\newline


\noindent
\textbf{Bounding the dimension of $\rho(P)$.} That the dimension of $\rho (P)$ is at most $\lfloor n/2 \rfloor -1$ is due to two factors. 
\begin{itemize}
\item  First, for each simplex $\sigma =x_0*x_1*...*x_k$ in $P_1$, we get a ``biggest" abelian\footnote{The group $A^\sigma$ is abelian because the groups $A_{x_i}$'s commute with each other.} group $A^\sigma = \langle A_{x_0}, ..., A_{x_k}\rangle$. Also, $A^\sigma$ preserves horospheres centered at $\rho(x_i) = \xi_{A_{x_i}}$ for $i = 0,..., k$ and therefore preserves their intersection. If $\rho(x_i)$'s span an $l$-dimensional simplex at infinity (for $\l \leq k$), then the dimension of the intersection of the horospheres \emph{should} be $n-(l+1)$. This \emph{should} mean that the rank of $A^\sigma$ is less than or equal to $n-(l+1)$.
\newline

\item Second, if $\sigma$ is a simplex in $P_1$, we expect the dimension of $\rho(\sigma)$ to be less than the rank of $A^\sigma$. One reason is because virtually equivalent abelian groups are too similar to demand different treatments, in particular, they should be assigned the same point at infinity. 
\newline 
\end{itemize}
Putting these two factors together we get that if $A^\sigma$ has rank $r$, then 
\[   r  \leq n - (l+1) \quad \text{and} \quad r \geq l +1.  \]
Therefore, $l+1 \leq \lfloor n/2 \rfloor$. So the dimension of $\rho (P)$ is at most  $\lfloor n/2 \rfloor -1$.
\newline

There are, of course, problems to overcome in both claims. There are two p\underline{roblems} in the second claim. One is that virtually equivalent abelian groups $A_x$ and $A_y$ need not have the same Centers of Mass, in which case the complex $\rho (P)$ might be higher dimensional than it should be. Even with little optimism one expects that if $A_x$ and $A_y$ share a finite index subgroup, there should be a point at infinity whose horospheres are preserved by both $A_x$ and $A_y$. The \underline{solution} is to construct, for each such abelian group, a Center of Mass  that is invariant under virtual equivalence and that has all the metric and invariance properties we mentioned above. This can be done and is done in Section \ref{goodpoints}. 
\newline

The other p\underline{roblem} in the second claim is that the rank of $A^\sigma$ could be \emph{strictly} less than the number of virtual equivalence classes of abelian groups  at the vertices. For example, $\sigma$ is a triangle and the group at each vertex of $\sigma$ is isomorphic to $\mathbb{Z}$ and no two of them share a finite index subgroup, yet $A^\sigma$ could be $\mathbb{Z}^2$. A \underline{solution} is to work with the second barycentric subdivision $P_2$ of $P$, instead of $P_1$, right from the beginning. So we need to assign abelian groups to vertices of $P_2$. Each vertex $x$ in $P_2$ that is not in $P_1$ is a point in the interior of a simplex $\tau$ of $P_1$. Assign to $x$ the abelian group $A^\tau$ generated by the abelian groups at the vertices of $\tau$ (i.e. let $A_x = A^\tau$). The pay off for working with $P_2$ is that for each $k$-simplex $\sigma $ in $P_2$, the abelian groups at the vertices of $\sigma$ form a chain $A_0\leq ... \leq A_k$. Another nice consequence is that the group generated by vertex groups is the biggest group $A_k$ in the chain (so one can forget the upper index). It follows that the rank of $A_k$ is greater than or equal to number of the number of virtual equivalence classes of abelian groups at the vertices, which takes care of the problem.    
\newline

The p\underline{roblem} with the first claim is that sometimes things might not be the way they \emph{should}. In this case, it is not clear if the intersection of horospheres described above cannot have dimension larger than $n-(l+1)$. This is ridiculous but it is unclear (to us) how to rule out the following situation.
\newline

Suppose that $h_0$, $h_1$ and $h_2$ are Busemann functions on $\widetilde{M}$. Let $z_i\in \partial_\infty $, for $i = 0,1,2$, be the center of the horosphere $S_i$ defined as $h_i = 0$.  Now, the intersection $ S = S_1 \cap S_2 \cap S_3$ of three horospheres is an $(n-3)$-dimensional manifold if $S_1$, $S_2$, and $S_3$ intersect transversely, 
i.e. the gradient vectors $\nabla h_0$, $\nabla h_1$ and $\nabla h_2$ at each point in  $S$ are linearly independent. Suppose that $z_0$, $z_1$ and $z_2$ are not co-linear in $\partial_\infty $, i.e. none of the three points is on the geodesic connecting the other two, so they span a triangle in $\partial_\infty $. Since $\nabla h_0$, $\nabla h_1$ and $\nabla h_2$ ``point toward" $z_0$, $z_1$ and $z_2$ respectively, this strongly suggests that they \emph{should} be linearly independent. However, being linearly independent at a point is too delicate\footnote{This is not sarcastic.} a condition and there is no reason to relate the linear structure at a point to what happens at infinity, which is something obtained via a limiting process in terms of the metric. We are not sure if this is a real problem or the problem lies in our inability\footnote{It turns out that the case of the intersection of three horospheres is not a problem because linear independence of three vectors is equivalent to a nondegenerate triangle on the unit tangent sphere and being a nondegenerate triangle can be captured by knowing the length of the sides starting at a fixed vertex. A triangle inequality at infinity translates to a triangle inequality in the unit tangent sphere at a point (if we pick the point far enough). However, we cannot resolve the case of the intersection of four or more horospheres because being a nondegenerate tetrahedron cannot be captured by knowing the length of the sides starting at a fixed vertex.}. 
\newline
\begin{figure}
\centering
\includegraphics[scale=0.085]{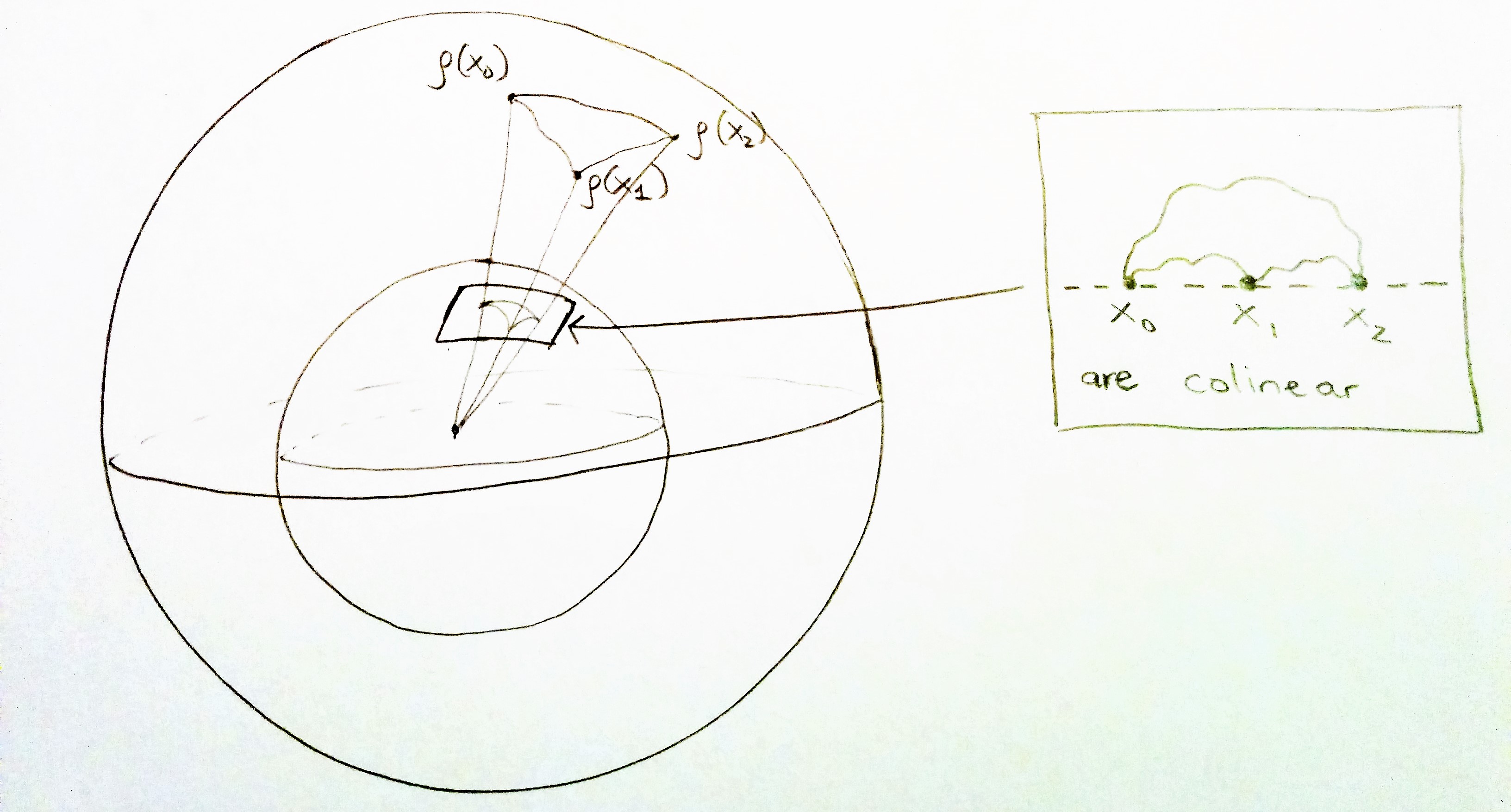}
\caption{}
\end{figure}

But we find a way around this problem and this is a \underline{solution}, which requires a modification to how $\rho$ is defined. This is the \emph{last} modification we will make to $\rho$. We define $\rho$ on the vertices of $P_2$ exactly as we did, but we will not use geodesics in $(\partial_\infty, \Td)$ to extend $\rho$ to edges and higher dimensional simplices. Instead, we construct what we call \emph{Busemann paths} and \emph{Busemann simplices} and we use them in place of geodesics in the above process of defining $\rho$. A Busemann $k$-simplex 
\[\sigma \colon \Delta^k \rightarrow \partial_\infty\] 
is a singular $k$-simplex in $\partial_\infty$ with vertices $z_0, z_1, ..., z_k$ and has the property\footnote{It is unclear if this property holds for geodesic simplices.} (by construction) that if a parabolic isometry preserves the horospheres centered at $z_i$ for $i =0,1,..., k$, then it will preserve horospheres centered at $\sigma (x)$ for \emph{all} $x\in \Delta^k$. The construction of Busemann simplices uses convex combinations of Busemann functions, which explains the name, and is given in Section \ref{gradientsection}. The incentive for constructing Busemann simplices is to create more points at infinity whose horospheres are preserved so that we can use them in the case when the horospheres centered at the vertices do not intersect transversely. 
\newline

It turns out, however, that even with a whole nondegenerate $k$-simplex of points at infinity whose horospheres are preserved by an abelian group $A_x$ we are unable to even prove existence of $(k+1)$ points whose Busemann functions have linearly independent gradient vectors everywhere. Nevertheless, Busemann simplices are too good to waste and we manage to use them to show that the first claim is true if geodesic simplices are replaced by these. Busemann simplices are constructed as pointwise limits of singular simplices 
\[\sigma_{R_i} \colon \Delta^k \rightarrow S(x_0, R_i) \subset \widetilde{M}\]  
on larger and larger spheres centered at some fixed point $x_0\in\widetilde M$. So for $R$ large enough, $\sigma_R$ ``approximates" the Busemann simplex $\sigma$ at infinity so well that nondegeneracy at infinity implies nondegenacy of $\sigma_R(\Delta^k)$. The union of all $\sigma_R(\Delta^k)$ is called the \emph{``Busemann cone"}, which is itself not a cone\footnote{It is more like a wizard's hat.}, serves as a parametrization space of intersections of  horospheres centered at the vertices $z_i$, for $i =0,1,..., k$. We show that if an abelian group $A$ of rank $r$ preserves horospheres centered at $z_i$'s, and if the Busemann simplex with vertices $z_i$'s is non-degenerate, then when we line up these intersections of horospheres over $C$ the union has dimension at least $r+k+1$, which gives the first claim. This is discussed in Section \ref{dimensionboundsection} and is hard enough to have its own ``problems and solutions".
\newline

Last but not least, all of the above effort will go to naught if Busemann simplices are space filling, in which case the Hausdorff dimension of a Busemann $l$-simplex could be greater than $l$. However, we show that Busemann simplices are Lipschitz and thus do not increase Hausdorff dimension. It follows that $\rho (P)$ has  dimension at most  $\lfloor n/2 \rfloor -1$, and since $c_t(P)$ is homeomorphic to $\rho (P)$, the  dimension of $c_t(P)$ is also at most  $\lfloor n/2 \rfloor -1$. This explains Theorems \ref{analog} and \ref{factor}. 
\newline

Finally, we can approximate $c_t(P)$ by a polyhedron $Q$ of dimension at most $\lfloor n/2 \rfloor -1$ in some small neighborhood of $c_t(P)$. So when we ``collapse" $\varphi_t(P)$ to $c_t(P)$, we can ``collapse" it to $Q$ instead and not have to worry that the collapse will not leave $\widetilde{M}_{<\epsilon}$ because $Q$ is pointwise close to $c_t(P)$. This gives a map $\widehat{\varphi}$ that factors through $Q$ as in the statement of Corollary \ref{homology of the thin part}.
\newline

There are no more problems.

\newpage

\section{ Setup and notation}\label{constants}
\subsection{Setup} In the rest of the paper, $M$ is a complete, finite volume $n$-dimensional manifold of bounded non-positive curvature $(-1\leq K\leq 0)$ with fundamental group $\Gamma:=\pi_1M$ and universal cover $\widetilde M\ra M$. Moreover, we assume that there are no arbitrarily small closed geodesics.

\subsection{Margulis lemma}There are constants $\mu_n$ and $I_n$, depending only on the dimension $n$, for which the group $\left<\gamma\in\Gamma\mid d(x,\gamma x)<\mu_n\right>$ generated by elements that move $x$ less than $\mu_n$ is virtually nilpotent and contains a nilpotent subgroup of index $\leq I_n$. The constant $\mu_n$ is called the Margulis constant (\cite{ballmangromovschroeder}).  

\subsection{Small $\epsilon$} We fix a constant $\epsilon>0$ to be less than the Margulis constant and the length of the smallest closed geodesic in $M$. Then elements $\gamma\in\Gamma$ which have displacement $<\epsilon$ at some point are parabolic. The ``$\epsilon$-thin part" 
\begin{equation*}
\widetilde M_{<\epsilon}:=\{x\in\widetilde M\mid d(x,\gamma x)<\epsilon\mbox{ for some } \gamma\in\Gamma\setminus\{1\}\}
\end{equation*} 
is topologically (see \cite{gromovnegativelycurved,ballmangromovschroeder}) a product $\partial\widetilde M_{<\epsilon}\times[0,\infty)$. For each $x\in\widetilde M_{<\epsilon}$, let 
\begin{eqnarray*}
S_x&:=&\{\gamma\in\Gamma\mid d_{\gamma}(x)<\epsilon\},\\
\label{nilpotent}
\Gamma_x&:=&\left<S_x\right>,\\
N_x&:=&\left<\gamma^{I_n!}\mid\gamma\in \Gamma_x\right>.
\end{eqnarray*}
By the Margulis lemma, the group $\Gamma_x$ is virtually nilpotent and $N_x$ is a nilpotent subgroup of $\Gamma_x$. Moreover, $N_x$ is {\it normal} in $\Gamma_x$ and since $\Gamma_x$ contains parabolic elements, so does $N_x$ does (Lemma 6.6 of \cite{ballmangromovschroeder}). 

\subsection{Tiny $\delta$\label{tinydelta}}
Fix another constant $\delta>0$ so that $\epsilon+2\delta$ is still less than the Margulis constant and the length of the shortest closed geodesic in $M$. If $\sigma=x_0*\dots*x_k$ is a $k$-simplex in $\widetilde M$ of diameter $<\delta$, then at any point $x\in\sigma$ all the elements in the set
$$
S_{\sigma}:=S_{x_0}\cup\dots\cup S_{x_k}
$$ 
have displacement $<\epsilon+2\delta$. This is less than the Margulis constant, so 
$$
\Gamma_{\sigma}:=\left<S_{\sigma}\right>=\left<\Gamma_{x_0},\dots,\Gamma_{x_k}\right>
$$ 
is a virtually nilpotent group and
$$
N_{\sigma}:=\left<\gamma^{I_n!}\mid\gamma\in\Gamma_{\sigma}\right>.
$$
is a normal nilpotent subgroup of $\Gamma_{\sigma}$ containing all the groups $N_{\tau}$ for $\tau\subset\sigma$. Since $N_{\sigma}$ contains parabolic elements, its center $Z_{\sigma}$ does as well. 
\newpage

\section{The Abelianization map or ``much ado about nothing" \label{abelmap}}
The goal of this section is to define a map $\mu$ from a triangulation of $\partial\widetilde M_{\leq\epsilon}$ to an abstract complex $\Delta_{\lfloor pAb\rfloor}$ of virtual equivalence classes of  abelian subgroups of $\Gamma$. This is the zeroth step in defining a map $\rho \colon \partial\widetilde M_{\leq\epsilon} \ra \partial_\infty$. Then in the next section, we will construct for each virtual equivalence class of such abelian subgroups a canonical center of mass at infinity in $\partial_\infty$ and use it to define $\rho$ on the vertices of $\partial\widetilde M_{\leq\epsilon}$. 
\subsection{Complexes of abelian and nilpotent groups}
Let 
$$
pAb:=\{A<\pi_1M\mid A \mbox{ is an abelian group containing a parabolic element}\}
$$
be the set of abelian groups in $\Gamma$ containing parabolics. Also, let $\lfloor pAb\rfloor$ be the set of virtual equivalence classes of such things. Denote by $\Delta_{pAb}$ the complex whose vertices are elements of $pAb$ and whose simplices are chains of such subgroups and define $\Delta_{\lfloor pAb\rfloor}$ similarly. In the same way we define $pNil$ and $\Delta_{pNil}$. The group $\Gamma$ acts on all these complexes by conjugation.
\subsection{Labeling the thin part with abelian groups}
We assemble consequences of the Margulis lemma at points in the thin part in three steps. Let $P$ be a $\delta$-fine\footnote{This means every simplex in the triangulation has diameter $<\delta$}, $\Gamma$-equivariant\footnote{Lift a triangulation of $M$ to $\widetilde M$. The result is a $\Gamma$-equivariant triangulation of $\widetilde M$.} triangulation of $\partial\widetilde M_{\leq\epsilon}$ and $P_1$ its barycentric subdivision. 
\begin{enumerate}
\item
Assign to each {\it vertex} $\tau$ of $P_1$ the nilpotent group 
$$
\mu'(\tau):=N_{\tau}.
$$ 
This extends to a map
$$
\mu':P\ra\Delta_{pNil}
$$ 
because adjacent vertices in $P_1$ give inclusions of nilpotent groups. The nilpotent groups $N_{\tau}$ in the image contain parabolics because $\epsilon$ is small. Note that $\mu'$ is determined by the isometric $\Gamma$-action once we fix the $\Gamma$-equivariant triangluation $P$, so $\mu'$ is $\Gamma$-equivariant and 
$$
\Gamma_{\tau}\mbox{ fixes }\mu'(\tau).
$$  
\item
Assign to each {\it vertex} $(N_0<\dots<N_k)$ of the barycentric subdivision of $\Delta_{pNil}$ the {\it group generated by the centers $Z_i$ of $N_i$},
$$
\zeta(N_0<\dots<N_k):=\left<Z_0,\dots,Z_k\right>.
$$ 
These groups are {\it abelian}\footnote{Because $Z_i$ centralizes the subgroup $\left<Z_0,\dots,Z_{i-1}\right>$ of $N_i$.} and contain parabolic elements. The assignment extends to a continuous map 
$$
\zeta:\Delta_{pNil}\ra\Delta_{pAb},
$$
because adding more groups to $N_0<\dots<N_k$ makes $\left<Z_0,\dots,Z_k\right>$ bigger. 
\item
Pass to virtual equivalence classes via 
$$
\nu:\Delta_{pAb}\ra\Delta_{\lfloor pAb\rfloor}.
$$
\end{enumerate}
\textbf{Abelianization map.} We call the composition $\mu:=\nu\circ\zeta\circ\mu'$,
\begin{equation*}
\mu:\partial\widetilde M_{\leq\epsilon}\ra\Delta_{\lfloor pAb\rfloor}.
\end{equation*}
the {\it Abelianization map}.
\begin{figure}
\label{barycentrictriangle}
\centering
\includegraphics[scale=0.08]{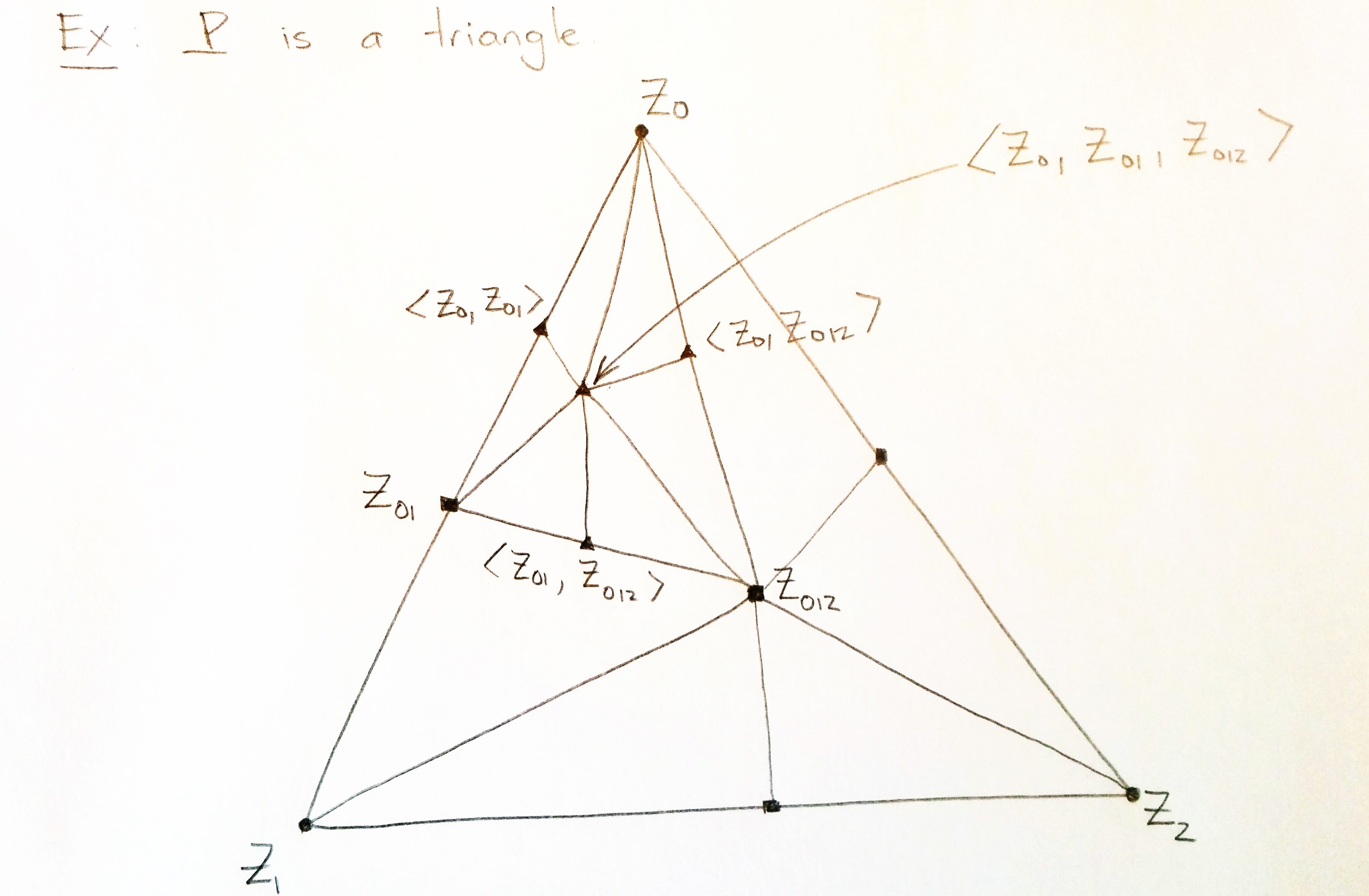}
\caption{}
\end{figure}
\begin{remark}
In summary, the map $\mu$ was described on each simplex in terms of the {\it second barycentric subdivision} $P_2$ of $P$. The reason is that we can connect nilpotent groups $N_0$ and $N_1$ corresponding to vertices of a simplex in $P$ by inclusions of nilpotent groups via
\begin{equation}
N_0\leq N_{01}\geq N_1,
\end{equation}
which is a path in the barycentric subdivision $P_1$, and we can connect the centers $Z_0$ and $Z_1$ by inclusions of {\it abelian} groups 
\begin{equation}
Z_0\leq \left<Z_0,Z_{01}\right>\geq Z_{01}\leq\left<Z_{01},Z_1\right>\geq Z_1,
\end{equation}
which is a path in the second barycentric subdivision $P_2$. It is necessary to pass from $P_1$ to $P_2$ when we go abelian because unlike the situation with nilpotent subgroups, it is in general not true that $ Z_0 \leq Z_{01} \geq Z_1$.
\end{remark}

\subsection{For later use \label{lateruse} (to stay thin)}
Let $\sigma=(\tau_0\subset\dots\subset\tau_k)$ be a $k$-simplex in the barycentric subdivision $P_1$. We noted above that $\Gamma_{\tau_i}$ fixes the vertex $\mu'(\tau_i)$.  Since $\mu=\nu\circ\zeta\circ\mu'$ and the maps $\nu$ and $\zeta$ are obviously $\Gamma$-equivariant, we see that 
the ``bottom'' group $\Gamma_{\tau_0}$ in the chain $\Gamma_{\tau_0}<\dots<\Gamma_{\tau_k}$ fixes the entire simplex $\mu(\sigma)=\nu\circ\zeta\circ\mu'(\tau_0\subset\dots\subset\tau_k)$.  Moreover, there is an element in $\Gamma_{\tau_0}$ that is $(\epsilon+2\delta)$-small everywhere on $\sigma$.\footnote{This is because $\sigma$ is in the subdivision of a simplex of $P$ whose diameter is $\leq\delta$ and at least one of whose vertices $x$ has an non-trivial element $\gamma\in S_x<\Gamma_{\tau_0}$ with $d_{\gamma}(x)<\epsilon$. Thus $d_{\gamma}\leq\epsilon+2\delta$ everywhere on $\sigma$. Also note that $\gamma^{I_n!}\in N_{\tau_0}$ is $(I_n!\epsilon+2\delta)$-small on $\sigma$.} In summary, 
$$
\mbox{for every }x\in\partial\widetilde M_{\leq\epsilon}\mbox{ there is }1\not=\gamma\in\Gamma\mbox{ such that} 
$$
$$
d_{\gamma}(x)\leq\epsilon+2\delta \mbox{ and } \gamma(\mu(x))=\mu(x).
$$ 

\newpage

\section{Center of mass for an abelian group with parabolics\label{goodpoints}}
In this section we will build a map $\beta$ on the {\it vertices} of the complex $\Delta_{\lfloor pAb\rfloor}$ 
\begin{equation}
\beta:\Delta_{\lfloor pAb\rfloor}^{(0)}\ra(\partial_{\infty},\Td)
\end{equation} 
such that 
\begin{itemize}
\item
the map $\beta$ is $\Gamma$-equivariant,
\item
horospheres centered at $\beta([A])$ are preserved by the centralizer $C_{A'}$ for any abelian group $A'$ virtually equivalent to $A$, and
\item  
for any simplex $\sigma$ in $\Delta_{\lfloor pAb\rfloor}$ we have in the Tits metric
\begin{equation}
\label{strictineq}
\mbox{diam }(\beta(\sigma^{(0)}))<\pi/2.
\end{equation} 
\end{itemize}
Then in later sections, we will discuss how to correctly fill in $\beta$ with simplices at infinity in order to finally obtain a map $\rho \colon \partial\widetilde{M}_{\leq \epsilon} \rightarrow \partial_\infty$.

Every abelian group $A$ containing a parabolic isometry has a nonempty fixed set $\Fix(A)$ at infinity with a canonical \emph{center of mass}\footnote{This construction can be found in Appendix 3.B of \cite{ballmangromovschroeder} and a variation is in 3.5 of \cite{eberlein}.} $\xi_A$. 
A review of this construction can be found in Appendix \"A. It is, however, not invariant under virtual equivalences, so we cannot use it to define $\beta$. The plan for this section is to first recall properties of the classical center of mass, and then, inspired by this, we construct 
 a canonical center of mass that depends only on the virtual equivalence class of the parabolic abelian group. We will prove similar properties for this new center of mass and use it to define $\beta$.
 
\subsection{\label{preservedhorospheres}The classical center of mass for a parabolic abelian group $A$}
The center of of mass $\xi_A$ has the crucial property that for every $y\in\Fix(A)$, 
$$
\Td(\xi_A,y)\leq\pi/2.
$$ Also, since the construction of $\xi_A$ is canonical,  $\xi_A$ is fixed by the normalizer of $A$. This implies that for a simplex $\sigma=(A_0<\dots<A_k)$ in $\Delta_{pAb}$ each group $A_i$ fixes all the points $\xi_{A_j}$, and from this the first property gives\footnote{We will see in subsection \ref{defofbeta} that by looking at the construction carefully we actually get that these distances are strictly less than $\pi/2$, but this is not important yet.} 
$$
\Td(\xi_{A_i},\xi_{A_j})\leq\pi/2.
$$ 
In addition, the following feature of $\xi_A$ is fundamental and is crucial in the next section.   
\begin{proposition}[Preserved horospheres]
\label{preservedcentralizer}
The centralizer $C_A$ preserves horospheres centered at $\xi_A$.\end{proposition}
\begin{proof}
Let $h$ be a Busemann function centered at $\xi_A$. We need to show that $h$ is $\gamma$-invariant under isometries $\gamma\in C_A$. Since $|\nabla h|=1$ we have
$$
{|h(\gamma^nx)-h(x)|\over n}\leq{d(\gamma^nx,x)\over n}.
$$
Since $\gamma\in C_A$ we already know it fixes $\xi_A$, so the quantity on the left is independent of $n$ and $x$ and is equal to $|h(\gamma x)-h(x)|$.
Letting $n\ra\infty$ and using the well known formula for the infimum displacement of an isometry
$$
|\gamma|:=\inf_{x\in\widetilde M}d(\gamma x,x)=
\lim_{n\ra\infty}{d(\gamma^nx,x)\over n}
$$
we get
$$
|h(\gamma x)-h(x)|\leq|\gamma|.
$$ 
So we see that $h$ is $\gamma$-invariant whenever $|\gamma|=0$.

Now, suppose that $|\gamma|>0$. Then, according to Karlsson-Margulis (\cite{karlssonmargulis}, see also Appendix B), there are geodesic rays $r_{\pm}=[x,\eta^{\pm})$ sublinearly tracking the positive and negative $\gamma$-orbits, i.e. 
$$
\lim_{n\ra\infty} {d(\gamma^nx,r_+(n|\gamma|))\over n}=0,$$
and 
$$
\lim_{n\ra\infty} {d(\gamma^{-n}x,r_-(n|\gamma|))\over n}=0.
$$
It follows from this and the formula for $|\gamma|$ that 
$$
2|\gamma|=\lim_{t\ra\infty}{d(r_+(t|\gamma|),r_-(t|\gamma|))\over t}.
$$
Reparametrizing and using the $\angle$-metric description on p. $36$ of \cite{ballmangromovschroeder} we get
$$
1=\lim_{t\ra\infty}{d(r_+(t),r_-(t))\over 2t}=\sin\left({\angle(\eta^+,\eta^-)\over 2}\right).
$$
Therefore $\angle(\eta^+,\eta^-)=\pi$ which implies that
$$
\Td(\eta^+,\eta^-)\geq\pi.
$$ 
On the other hand, since $A$ commutes with $\gamma$ it fixes the limit points $\lim_{n\ra\infty}\gamma^nx=\eta^+$ and $\lim_{n\ra\infty}\gamma^{-n}x=\eta^-$. Therefore 
$$
\Td(\eta^{\pm},\xi_A)\leq\pi/2.
$$
Putting these two inequalities together we get
$$
\pi\leq\Td(\eta^+,\eta^-)\leq\Td(\eta^+,\xi_A)+\Td(\xi_A,\eta^-)\leq\pi
$$
and consequently 
$$
\Td(\eta^{\pm},\xi_A)=\pi/2.
$$ 
This implies that $\nabla h\cdot\nabla {r_+}\ra 0$ along the geodesic ray $r_+$ (see 4.2 in \cite{ballmangromovschroeder}) and consequently
$$
{h(r_+(n|\gamma|))-h(x)\over n}={1\over n}\int_0^{n|\gamma|}|\nabla h\cdot\nabla r_+|dt\ra 0
$$
as $n\ra\infty$. Putting this together with  
$$
{|h(\gamma^nx)-h(r_{+}(n|\gamma|))|\over n}\leq{d(\gamma^nx,r_+(n|\gamma|))\over n}\ra 0
$$ 
we get that ${h(\gamma^n x)-h(x)\over n}\ra 0$. Since this quantity is actually constant and equal to $h(\gamma x)-h(x)$, we conclude that $h$ is $\gamma$-invariant.  
\end{proof}

\subsection{Dealing with finite index issues\label{findex}\label{twostep}}
Centers of mass of virtually equivalent abelian groups might be different. Our goal in this subsection is to pick a single point at infinity that will play the role of the center of mass for the whole virtual equivalence class $[A]$ of $A$. We will do this by constructing a center of mass for the union of fixed point sets of all the groups virtually equivalent to $A$, which is also equal to 
$$
F_A:=\bigcup_{n\in\mathbb N}\Fix(n!A),
$$ 
since every group virtually equivalent to $A$ contains some $n!A$ as a subgroup. It is easy to see and occasionally useful to remember that
\begin{eqnarray}
\label{independent}
[A]=[A']&\implies& F_A=F_{A'}, \mbox{ and}\\
B\leq A &\implies& F_B\supseteq F_A.
\end{eqnarray} 
\noindent
\textbf{A two-step center of mass construction.} Now we construct a center of mass for $F_A$ that depends only on the virtual equivalence class of $A$. We do this in two steps.

(\textit{Step 1.}) First, we will show that there is a point $\xi$ in $F_A$ so that any other point of $F_A$ is within $\pi/2$ of $\xi$. To show this, let   
$$
B_{n,A}:=\{x\in \Fix(A)\mid \Td(x,y)\leq\pi/2\mbox{ for all }y\in\Fix(n!A)\}.
$$ 
The sets $B_{n,A}$ are closed in the sphere topology and nested:
$$
B_{0,A}\supseteq B_{1,A}\supseteq B_{2,A}\supseteq\dots
$$
They are also non-empty because the center of mass $\xi_{n!A}$ of the fix set of $n!A$ is fixed by $A$ and therefore $\xi_{n!A}\in B_{n,A}$. So there is a point $\xi$ contained in the intersection $\cap_{n}B_{n,A}$. For this $\xi\in\Fix(A)$ we have $\Td(\xi,y)\leq \pi/2$ for all $y\in F_A$. This finishes the first step.

(\textit{Step 2.})
The point $\xi$ constructed in Step 1 may not be unique, and we denote the set of all such points by 
$$
B_{[A]}:=\{x\in F_A\mid \Td(x,y)\leq\pi/2 \mbox{ for all } y\in F_A\}.
$$ 
This set is our collection of {\it potential} centers of mass. The second step is to pick in a canonical way a single point from this set. 

It is clear that $B_{[A]}$ has Tits diameter $\leq \pi/2$, so {\it if it was closed in the sphere topology} then one way to pick a unique point would be to take the Center of $B_{[A]}$ (in the sense of Appendix \"A). Unfortunately, the set $B_{[A]}$ may not be closed in the sphere topology.\footnote{The reason is that $F_A$ is not a closed set but only a countable union of closed sets.}, so we first replace it by its closure in the sphere topology $\overline B_{[A]}$. It is easy to see\footnote{Using lower semicontinuity of $\Td$, see Appendix \"A.} that this closure is equal to
$$
\overline{B}_{[A]}=\{x\in\overline F_A\mid\Td(x,y)\leq\pi/2\mbox{ for all } y\in\overline F_A\}.
$$
In particular, $\overline B_{[A]}$ still has Tits diameter $\leq\pi/2$. Therefore, the function
$$
\rho(\cdot)=\sup_{x\in\overline B_{[A]}}\Td(\cdot,x)
$$
has infimum 
$$
\inf\rho<\pi/2-\alpha
$$ 
for a positive constant $\alpha=\alpha_n>0$ that only depends on the dimension $n$. This infimum is attained at a unique point\footnote{In other words, there is a unique ball of smallest radius containing $\overline B_{[A]}$. This ball is centered at $\xi_{[A]}$ and has radius $\rho(\xi_{[A]})$.} in $\partial_{\infty}$ which we will denote by $\xi_{[A]}$. (See Appendix \"A for everything in this paragraph.) We call this the {\it center of mass of $[A]$}.

What remains to be shown is that $\xi_{[A]}$ is actually contained in $B_{[A]}$. We prove this in the remainder of the subsection. We begin with
\begin{lemma}
The set $B_{[A]}$ is convex. 
\end{lemma}
\begin{proof}
Let $x_0,x_1\in B_{[A]}$ and let $x_t$ be a point on the geodesic segment in $\partial_{\infty}$ connecting them. There are virtually equivalent groups $A_0,A_1\in[A]$ fixing $x_0$ and $x_1$, so the entire geodesic segment is fixed by the group $(A_0\cap A_1)\in[A]$. Moreover, since $\partial_{\infty}$ is CAT($1$) we see for any $y\in F_A$ that $\Td(x_t,y)\leq\pi/2$ by comparison with the round sphere. Thus $x_t\in B_{[A]}$.
\end{proof}
We will use this to show that $\xi_{[A]}$ is contained in $\overline B_{[A]}$. If the closure $\overline B_{[A]}$ was convex, then this would follow easily (see Appendix \"A). But, we only know that $\overline B_{[A]}$ is {\it the closure in the sphere topology of a $\Td$-convex set}. So, we need the following lemma.  
\begin{lemma}
Fix $\alpha>0$.
Suppose $C$ is a convex set of diameter $\leq\pi/2$. Let $\overline C$ be its closure in the sphere topology. Let $\xi\in\partial_{\infty}$ be a point for which 
$$
\Td(\xi,y)\leq\pi/2-\alpha\mbox{ for every }y\in C.
$$ 
Then there is a point $x\in\overline C$ so that 
$$
\Td(x,y)\leq\Td(\xi,y) \mbox{ for all } y\in C.  
$$
\end{lemma}
\begin{proof}
First, let $r=\inf_{y\in C}\Td(\xi,y)$. Then there is a sequence of points $x_i\in C$ so that $\Td(\xi,x_i)\ra r$. After passing to a subsequence, we may assume that $x_i\ra x\in\overline C$. 
\begin{figure}
\centering
\includegraphics[scale=0.5]{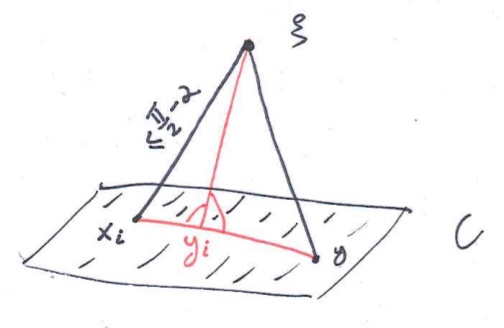}
\end{figure}
Now, pick a point $y\in C$. Since $C$ is convex, the geodesic segment $[x_i,y]$ is contained in $C$. On this segment there is a unique closest point $y_i$ to $\xi$. Since it is the closest point to $y$ on the segment $[x_i,y]$, at this point the angles $\angle_{y_i}(x_i,\xi)$ and $\angle_{y_i}(\xi,y)$ are both obtuse. Therefore, triangle comparison with obtuse triangles on the round sphere gives
$$
\Td(y_i,y)\leq\Td(\xi,y),
$$
and also
$$
\Td(x_i,y_i)\leq\cos^{-1}\left({\cos(\Td(x_i,\xi))\over\cos(\Td(y_i,\xi))}\right).
$$ 
As $i\ra\infty$ the right hand side of this tends to zero because the denominator 
$$
\cos(\Td(y_i,\xi))\geq\cos(\pi/2-\alpha)>0
$$ 
doesn't approach zero and
$$
r\leq\Td(y_i,\xi)\leq\Td(x_i,\xi)\ra r.
$$
Therefore, using lower semicontinuity of $\Td(\cdot,y)$ we get 
\begin{eqnarray*}
\Td(x,y)&\leq&\liminf_{i}\Td(x_i,y)\\
&\leq&\liminf_i(\Td(x_i,y_i)+\Td(y_i,y))\\
&\leq&\Td(\xi,y).
\end{eqnarray*}\end{proof}
So, since $B_A$ is a convex set of diameter $\leq\pi/2$ and $\xi_{[A]}$ is the unique point at which $\rho$ attains its infimum, we conclude that
$$
\xi_{[A]}\in\overline B_A.
$$ 
Therefore 
\begin{eqnarray}
\Td(\xi_{[A]},y)&\leq&\pi/2\hspace{1cm} \mbox{ for all } y\in F_A,\mbox{ and}\\
\label{tdbound}
\Td(\xi_{[A]},y)&\leq&\pi/2-\alpha\hspace{0.3cm} \mbox{ for all } y\in B_{[A]}.
\end{eqnarray}
The set $F_A$ is preserved by $A$ so its center of mass $\xi_{[A]}$ is fixed by $A$, and therefore 
$$
\xi_{[A]}\in B_{[A]}, 
$$
which is what we needed to show. This finishes the second step. 
\begin{remark}
At this point the reader can safely forget about the closures $\overline B_{[A]}$. While they appeared in the construction of $\xi_{[A]}$ they will never appear again.
\end{remark}  

\subsection{\label{defofbeta}The map $\beta:\Delta^{(0)}_{\lfloor pAb\rfloor}\ra(\partial_\infty,\Td)$} 
We set
$$
\beta([A]):=\xi_{[A]}.
$$ 
Let us verify that it has all the properties we promised.
First, it follows from the construction that $\xi_{[\gamma A\gamma^{-1}]}=\gamma\xi_{[A]}$ so the map $\beta$ is $\Gamma$-equivariant. Second, for any $A'$ virtually equivalent to $A$, the group $C_{A'}$ fixes $\xi_{[A']}=\xi_{[A]}$, so the proof of Proposition \ref{preservedcentralizer} applies word-for-word and shows that $C_{A'}$ preserves horospheres centered at $\xi_{[A]}$. To prove the third bullet we proceed as follows. 
For any simplex $\sigma=(A_0<\dots<A_k)$ in $\Delta_{\lfloor pAb\rfloor}$ 
$$
F_{A_0}\supseteq\dots\supseteq F_{A_k}, 
$$
so if $i\leq j$ we get $\Td(\xi_{[A_i]},y)\leq\pi/2$ for all $y\in F_{A_j}$.
Since each group $A_i$ fixes all the points $\xi_{[A_j]}$ we also have $\xi_{[A_i]}\in F_{A_j}$. In summary 
$$
\xi_{[A_i]}\in B_{[A_j]}\mbox{ for all }i\leq j.
$$ 
The upshot of these gymnastics is that (\ref{tdbound}) implies that for all $i,j$ 
$$
\Td(\xi_{[A_i]},\xi_{[A_j]})\leq\pi/2-\alpha.
$$
In other words, the diameter of $\beta(\sigma^{(0)})$ is strictly less than $\pi/2$, which is what we wanted to show. 





\section{A criterion for and the necessity of being divergent}\label{divraysection}
Now that we have defined $\beta$ (and thus $\rho$) on vertices, we need to extend it to each simplex. The extension must be canonical and satisfy a divergence property. The goal of this section is to carefully discuss this notion of a divergent simplex at infinity, give a criterion for when a simplex is divergent, and illustrate how it is useful in the context of the main theorem.    

\subsection{Divergent rays, divergent simplices and divergent maps}
\subsubsection{\textbf{Divergent rays}}
Recall that a geodesic ray in $\widetilde{M}$ is \emph{divergent} if its projection under the covering map $p$ leaves all compact sets. It turns out, as we will see, that parabolic abelian subgroups $A<\Gamma$ give geodesic rays $[x,\xi_A)$ which project to {\it divergent} rays in $M$, and thus determine distinguished directions to go to infinity in $M$. 

\begin{proposition}\label{distinguished directions}
Any ray $[x,\xi_{[A]})$ in $\widetilde M$ projects to a divergent ray in $M$.
\end{proposition}

The key to obtaining results of this sort is the strong invariance of $\xi_A$ established in Proposition \ref{preservedcentralizer}. The centralizer $C_A$ preserves horospheres centered at $\xi_{[A]}$, so Proposition \ref{distinguished directions} follows from Lemma \ref{divergentrays} below. This is a way to produce divergent rays in $M$.

\begin{lemma}
\label{divergentrays}
Suppose $A$ is a subgroup\footnote{We only care about abelian $A$ in this paper, but the proposition works for general $A$.} of $\Gamma$. If $C_A$ preserves horospheres centered at $\xi$, then any geodesic ray $r:[0,\infty)\ra\widetilde M$ with endpoint $r(\infty)=\xi$ projects to a divergent ray in $M$. 
\end{lemma}

\begin{proof}
We prove the contrapositive. If the projection of the geodesic ray $r$ to $M$ does not diverge then there is a sequence of times $t_i\ra\infty$ and elements $g_i\in\Gamma$ so that $\{g_ir(t_i)\}$ converges to a point $x_0$ in $\widetilde M$. We will construct out of this an element in $C_A$ that does not preserve a Busemann function $h$ centered at $\xi$. This will prove the Proposition. 
Let $D:=\sup_{i}d(x_0,g_ir(t_i))$.

\subsection*{Claim} After passing to a subsequence of $\{g_i\}$, we have $g_j^{-1}g_i\in C_A$.
 
For any element $\gamma$, the triangle inequality\footnote{For any isometry $\rho$, it follows from triangle inequality that $|d_{\rho}(x)-d_{\rho}(y)|\leq 2d(x,y)$.} implies
\begin{eqnarray}
|d_{g_i\gamma g_i^{-1}}(x_0)-d_{g_i\gamma g_i^{-1}}(g_ir(t_i))|&\leq&2d(x_0,g_ir(t_i))\\
&\leq&2D.
\end{eqnarray}
If $\gamma$ fixes $r(\infty)=\xi$, we also get 
\begin{eqnarray}
d_{g_i\gamma g_i^{-1}}(g_ir(t_i))&=&d_{\gamma}(r(t_i)),\\
&\leq&d_{\gamma}(r(0)),
\end{eqnarray}
so that $\{d_{g_i\gamma g^{-1}_{i}}(x_0)\}_{i=1}^{\infty}$ is bounded. Thus, there are only finitely many different conjugates in the sequence $\{g_i\gamma g_i^{-1}\}_{i=1}^{\infty}$. After passing to a subsequence, we may assume that all the conjugates are the same, i.e. that
\begin{equation}
g_1\gamma g_1^{-1}=g_2\gamma g_2^{-1}=\dots,
\end{equation}  
and consequently that $g^{-1}_jg_i$ commutes with $\gamma$. 

In the special case when $A=\left<\gamma_1,\dots,\gamma_r\right>$ is a finitely generated group fixing $\xi$ we can do the above argument for each one of the generators. So, after passing to subsequences finitely many times, we get a sequence $\{g_i\}$ for which $g_j^{-1}g_i$ commutes with the entire group $A=\left<\gamma_1,\dots,\gamma_r\right>$.

In general, $A$ is countable so we get the same result via diagonal argument. 

\subsection*{Claim} For large enough $j$, the element $g_j^{-1}g_i$ does not preserve $h$. 

Note that $d(g_j^{-1}g_ir(t_i),r(t_j))=d(g_ir(t_i),g_jr(t_j))$ is bounded by $2D$, so 
$$
|h(g_j^{-1}g_ir(t_i))-h(r(t_j))|\leq 2D.
$$
On the other hand, as $j\ra\infty$ we have
\begin{eqnarray}
h(r(t_j))&=&h(r(0))-t_j\\
&\ra&-\infty.
\end{eqnarray}
Therefore $\lim_{j\ra\infty}h(g_j^{-1}g_ir(t_i))=-\infty$. This implies that $h$ is not $g_j^{-1}g_i$-invariant for a fixed $i$ and large enough $j$.

So we've found an element $g_j^{-1}g_i\in C_A$ that does not preserve horospheres centered at $r(\infty)$. This proves the proposition. 
\end{proof}

\begin{figure}
\centering
\includegraphics[scale=0.8]{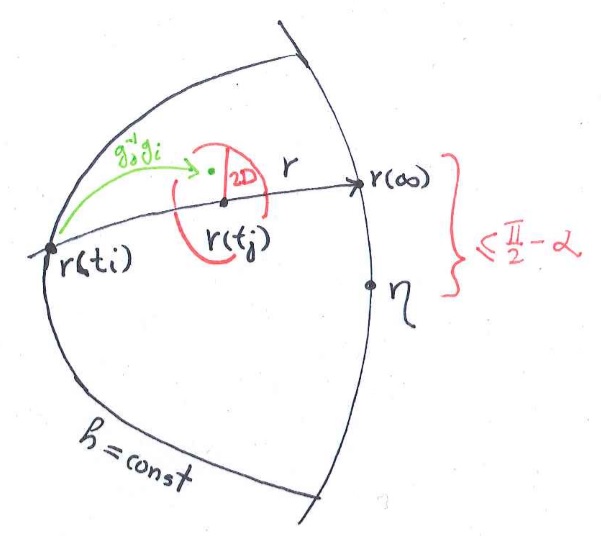}
\end{figure}
\subsection*{Wiggle room in the divergent ray argument} It is good to notice that the divergent ray argument is not delicate. 
There is quite a bit of ``wiggle room'' in the argument. If one looks through the proof, one sees that
the assumptions can be weakened. We only need to know that
\begin{enumerate}
\item
\label{wiggle1} 
the group $A$ fixes the point at infinity $r(\infty)$, \item
\label{wiggle}
there is {\it some point} $\eta$ such that $C_A$ preserves horospheres at $\eta$, and
\item
\label{wiggle3}
there is a positive constant $\alpha>0$ such that 
\begin{equation*}
\Td(\eta,r(\infty))\leq\pi/2-\alpha.
\end{equation*}
\end{enumerate}
In other words, we can separate the horosphere point $\eta$ from the endpoint of the geodesic ray $r(\infty)$ in the argument as illustrated in the Figure above. The point $r(\infty)$ just needs to be fixed by $A$ as long as the (much stronger) condition that horospheres are preserved by the entire centralizer is satisfied by some nearby point $\eta$. Having phrased things in this way, we note that {\it we can vary the endpoint $r(\infty)$ of the geodesic ray, as long as all the rays we use satisfy (\ref{wiggle1}) and (\ref{wiggle3}) for a \underline{single point $\eta$ and a single constant $\alpha>0$}.} Finally, note that we can vary the startpoint of the geodesic ray $r$ in a bounded set. So, we arrive at the following Proposition, which produces divergent sectors.

\begin{proposition}
\label{divergentsectorsearly}
Suppose $A$ is a subgroup of $\Gamma$, $B$ is a bounded subset of $\widetilde M$, and $h$ is a $C_A$-invariant Busemann function centered at a point $\eta\in\partial_{\infty}$. Then for every $\epsilon>0$ there is a constant $T:=T_{A,\epsilon,\eta,\alpha,B}$ so that any geodesic ray $r:[0,\infty)\ra\widetilde M$ with $r(0)\in B$ and $r(\infty)$ satisfying (\ref{wiggle1}) and (\ref{wiggle3})$_{\eta,\alpha}$ has
$$
r(t)\in\widetilde M_{\leq\epsilon}\mbox{ for all } t\geq T.
$$
\end{proposition} 

\begin{proof}
Suppose the conclusion is not true. Then there are times $t_i\ra\infty$, elements $g_i\in\Gamma$, and rays $r_i$ with $r_i(0)\in B$ and $r_i(\infty)$ satisfying (\ref{wiggle1}) and (\ref{wiggle3}) such that $\{g_ir_i(t_i)\}$ converges to a point $x_0\in\widetilde M$. As before, using (\ref{wiggle1}) we show that after passing to a subsequence we can assume $g_j^{-1}g_i\in C_A$ for all $i,j$. As before, $|h(g_j^{-1}g_ir_i(t_i))-h(r_j(t_j))|\leq 2D$ and condition (\ref{wiggle3}) implies 
\begin{eqnarray*}
h(r_j(t_j))&\leq&h(r_j(0))-t_j\cdot\sin\alpha,\\
&\ra&-\infty
\end{eqnarray*}
so we again conclude that $\lim_{j\ra\infty}h(g_j^{-1}g_ir(t_i))\ra-\infty$. This contradicts the assumption that $h$ is $C_A$-invariant, so it proves the proposition. 
\end{proof}

\subsubsection{\textbf{Divergent maps and divergent simplices}}
In order to state our main application of the proposition, we introduce the following terminology. A family of maps $\{\varphi_t:X\ra\widetilde M\}_{t\in\mathbb R^+}$ {\it diverges over $M$} if for any $\epsilon>0$ there is $T$ so that
$$
\im(\varphi_t)\subset\widetilde M_{\leq\epsilon}\mbox{ for all } t\geq T.
$$ 
Two families $\{\varphi_t,\psi_t:X\ra \widetilde M\}_{t\geq 0}$ are {\it asymptotic} if for every compact set $K\subset X$, the distance $\sup_{x\in K,t\geq 0}d(\varphi_t(x),\psi_t(x))$ is finite.
Next, fix a basepoint $z\in\widetilde M$ and let 
$$
c_t=c_t^z:\partial_{\infty}\ra S_z(t)
$$ 
be the geodesic retraction that sends $\xi\in\partial_{\infty}$ to the point $[z,\xi)_t$ obtained by flowing for a time $t$ along the geodesic ray from $z$ to $\xi$. We say that a singular simplex $\lambda:\Delta^k\ra\partial_{\infty}$ {\it diverges over $M$} if the family $c_t\circ\lambda$ diverges over $M$. Whenever it doesn't cause confusion, we will omit ``over $M$'' and just say that the simplex $\lambda$ diverges. A direct corollary of Proposition \ref{divergentsectorsearly} is the following criterion for finding divergent simplices.

\begin{corollary}
\label{divergencecriterion}
Suppose horospheres centered at $\eta\in\partial_{\infty}$ are $C_A$-invariant. For $\alpha>0$, any simplex contained in $\Fix(A)\cap B_{\pi/2-\alpha}(\eta)$ diverges over $M$.
\end{corollary} 

It is easy to see that the defintion of divergence for simplices does not depend on the choice of basepoint $z$. The underlying reason is that for different basepoints $z$ and $z'$ the cone homotopies $c^z_t\circ\lambda$ and $c^{z'}_t\circ\lambda$ are asymptotic. Somewhat more generally we have the following lemma which will be useful in the next subsection. 

\begin{lemma}\label{comparison divergence} 
Suppose that two families $\{\varphi_t,\psi_t:X\ra\widetilde M\}_{t\geq 0}$ are asymptotic. If $\varphi_t$ diverges then for any compact subset $K\subset X$
\begin{itemize}
\item
for sufficiently large $t$, the straight line homotopy between $\varphi_t\mid_K$ and $\psi_t\mid_K$ is {\it inside} $\widetilde M_{\leq\epsilon}$, and in particular
\item
$\psi_t\mid_K$ diverges.
\end{itemize}
\end{lemma}

\begin{proof}
Let $D:=\sup_{x\in K,t\geq 0}d(\varphi_t(x),\psi_t(x))$. This is finite because $\varphi_t$ and $\psi_t$ are asymptotic. Since $\varphi_t$ diverges in $M$ and the injectivity radius function on $M$ is proper, there is a time $T$ so that for $t\geq T$ the closed $D$-neighborhood of $\varphi_t\mid_K$ is contained in $\widetilde M_{\leq\epsilon}$.\footnote{Since the injectivity radius is proper, there is $\epsilon'<\epsilon$ so that $d(M_{\geq\epsilon},M_{\leq_{\epsilon'}})>D$. Then also $d(\widetilde M_{\geq\epsilon},\widetilde M_{\leq\epsilon'})>D$ and therefore once $t$ is large enough so that $\varphi_t$ is in $\widetilde M_{\leq\epsilon'}$ its $D$-neighborhood will be in $\widetilde M_{\leq\epsilon}$.} Since the geodesic homotopy between $\varphi_t\mid_K$ and $\psi_t\mid_K$ is in this neighborhood, we get the first bullet point. The second follows immediately from the first.
\end{proof}

\subsubsection{\textbf{Filling $\beta$ in with divergent simplices}\label{goodbeta}} 
We are now almost ready to establish our basic collapse result, Theorem \ref{collapsetheorem} below. In order to do this, we will need to extend the ``center of mass'' map $\beta$ to the entire complex $\Delta_{\lfloor pAb\rfloor}$ by filling it in with divergent simplices in a natural way. The resulting $\beta:\Delta_{\lfloor pAb\rfloor}\ra(\partial_{\infty},\Td)$ should be continuous, of course, and it should
\begin{enumerate}
\item
\label{equiv}
be $\Gamma$-equivariant, 
\item
\label{com}
send a vertex $[A]$ to its center of mass $\xi_A$, and
\item
\label{divprop}
for each simplex $\sigma$ of $\Delta_{\lfloor pAb\rfloor}$, the simplex $\beta(\sigma)$ should diverge.  
\end{enumerate} 
We will express (\ref{divprop}) by saying ``$\beta$ diverges on simplices''.
Because of condition (\ref{com}) and by Proposition \ref{distinguished directions}, the {\it vertices} of a simplex $\sigma$ are mapped to divergent rays by $\beta$. The extra condition (\ref{divprop}) says that for all {\it points} in a simplex $p\in\sigma$ all the rays $[z,\beta(p))$ diverge, and they do so \underline{uniformly}.
In practice, if the map $\beta$ doesn't distort things too much one can get (\ref{divprop}) from (\ref{equiv}) and (\ref{com}):
\begin{lemma}
Suppose $\beta$ satisfies (\ref{equiv}) and (\ref{com}), and let $\alpha>0$. If $\beta(\sigma)$ is in the $(\pi/2-\alpha)$-neighborhood of the vertex set $N_{\pi/2-\alpha}(\beta(\sigma^{(0)}))$ then $\beta(\sigma)$ diverges.
\end{lemma}
\begin{proof}
Let $[A]$ be a vertex of $\sigma$. 
Then $\sigma$ is fixed pointwise by $A$ so, since $\beta$ is equivariant, $\beta(\sigma)$ is fixed by $A$, as well.
Moreover, the group $C_A$ preserves horospheres centered at $\beta([A])$. Therefore Corollary \ref{divergencecriterion} implies
\begin{equation}
\label{diam1}
\beta(\sigma) \mbox{ diverges if for some vertex }v \mbox{ we have }
\beta(\sigma)\subset B_{\pi/2-\alpha}(v).
\end{equation} 
Even if the entire simplex is not contained in the $(\pi/2-\alpha)$-neighborhood of a single vertex, applying Corollary \ref{divergencecriterion} to a vertex shows that the portion of the simplex that lies in the $(\pi/2-\alpha)$-neighborhood of that vertex diverges. Doing this for each one of the vertices of $\sigma$ proves the lemma.\end{proof}
We will describe different versions of $\beta$ in Section \ref{intermission} and use the lemma to check that for each of these versions all simplices $\beta(\sigma)$ diverge. 

\subsection{The perks of being a divergent simplex}\label{perks}
Suppose we have, one way or another, got our hands on such a ``divergent simplex'' map $\beta$. Let us explain how it can be used, together with the abelianization map $\mu$ constructed in section \ref{abelmap}, to understand the topology of the thin part. The idea is that the composition $\rho:=\beta\circ\mu$
$$
\partial\widetilde M_{\leq\epsilon}\stackrel{\mu}\ra\Delta_{\lfloor pAb\rfloor}\stackrel{\beta}\ra(\partial_{\infty},\Td)
$$ 
tells us how to push topological features to infinity while staying in the thin part. To make this precise, denote by 
\begin{eqnarray*}
(\beta\circ\mu)_t:\partial\widetilde M_{\leq\epsilon}&\ra&\widetilde M\\
x&\mapsto&[x,\beta\circ\mu(x))_t
\end{eqnarray*}
the map which sends a point $x\in\partial\widetilde M_{\leq\epsilon}$ to the point obtained by going for a time $t$ along the geodesic ray $[x,\beta\circ\mu(x))$. Note that it is $\Gamma$-equivariant and that it ``approaches $\beta\circ\mu$'' in the sense that
$$
(\beta\circ\mu)_t\mbox{ is asymptotic to the cone homotopy }c_t\circ\beta\circ\mu.
$$
Two additional key features of this map (proved below) is that it stays in the thin part for all $t$ and pushes further into the thin part for large $t$. This lets us collapse any compact subset in $\widetilde M_{\leq\epsilon}$ to a subset of topological dimension less than or equal to $\dim(\im(\beta\circ\mu),\angle_x).$

\begin{theorem}
\label{collapsetheorem}
Let $\mu:\partial\widetilde M_{\leq\epsilon}\ra\Delta_{\lfloor pAb\rfloor}$ be the abelianization map. Suppose there is a $\Gamma$-map $\beta:\Delta_{\lfloor pAb\rfloor}\ra(\partial_{\infty},\Td)$ which sends vertices to their centers of mass $\beta([A])=\xi_{[A]}$ and which diverges on simplices. Then
\begin{itemize}
\item
$(\beta\circ\mu)_t:\partial\widetilde M_{\leq\epsilon}\ra\widetilde M$ is in $\widetilde M_{\leq\epsilon+2\delta}$ for all $t\geq 0$, and 
\item
$(\beta\circ\mu)_t$ diverges over $M$. 
\end{itemize}
Denote the dimension of the image of $\beta\circ\mu$ in the sphere topology by 
$$
d:=d_{\beta}=\dim(\im(\beta\circ\mu),\angle_x).
$$
Then the inclusion of any compact subset $\varphi:K\hookrightarrow\partial\widetilde M_{\leq\epsilon}$ can be homotoped in $\widetilde M_{\leq\epsilon+2\delta}$ to a map $\hat\varphi$ with image of dimension $\leq d$. 
\end{theorem}

We will prove this theorem at the end of this section. Next, we present several topological consequences of Theorem \ref{collapsetheorem}.
\begin{corollary}
\label{homologybound}
Assume the hypotheses of Theorem \ref{collapsetheorem}. Then $$H_{>d}(\widetilde M_{\leq\epsilon})=0.$$
\end{corollary}

\begin{proof}
Let $\varphi:F\ra\partial\widetilde M_{\leq\varepsilon}$ be a homology cycle of dimension $k>d$. By the previous theorem, it can be homotoped in $\widetilde M_{\leq\epsilon+2\delta}$ to a map $\hat\varphi$ with $d$-dimensional image. Since $M$ is tame, we can push the homotopy a little along the product direction of $\widetilde M_{\leq\epsilon+2\delta}=\partial\widetilde M_{\leq\epsilon+2\delta}\times[0,\infty)$ so that it stays in the $\epsilon$-thin part $\widetilde M_{\leq\epsilon}$.  By a standard argument (recalled in Appendix D) we can further homotope the map $\hat\varphi$ in $\widetilde M_{\leq\epsilon}$ to a map $\overline{\varphi}$ whose image lands in the $d$-skeleton $\widetilde M_{\leq\epsilon}^{(d)}$ of a triangulation of $\widetilde M_{\leq\epsilon}$. Since $d<k$, we conclude that $\varphi$ is zero on $k$-dimensional homology. 
\end{proof}

\begin{corollary}
Assume the hypotheses of Theorem \ref{collapsetheorem}. If $d\leq 1$ then each component of $\widetilde M_{\leq\epsilon}$ is aspherical. 
\end{corollary}

\begin{proof}
If $d\leq 1$ then, arguing as in the previous corollary, any map $\varphi:S^k\ra\widetilde M_{\leq\epsilon}$ can be homotoped in $\widetilde M_{\leq\epsilon}$ to factor through a graph. 
Therefore each component of $\widetilde M_{\leq\epsilon}$ is aspherical. 
\end{proof}
In order for these results to mean anything, we need some control over the dimension $d$. The first observation is that $d$ is also equal to the dimension of the image of $\rho=\beta\circ\mu$ {\it in the Tits metric.}
\begin{proposition}
\label{titsdimension}
Assume the hypotheses of Theorem \ref{collapsetheorem}. Then $$d=\dim(\im\rho,\angle_x)=\dim(\im\rho,\Td)\leq\dim(\partial_{\infty},\Td).$$ 
\end{proposition}
\begin{proof}
First, note that each simplex $\rho(\sigma)$ is compact in the Tits metric. Therefore, the identity map $(\rho(\sigma),\Td)\ra(\rho(\sigma),\angle_x)$ is a homeomorphism (since it is a continuous bijection from a compact space to a Hausdorff space) and thus it preserves topological dimensions, i.e. $\dim(\rho(\sigma),\angle_x)=\dim(\rho(\sigma),\Td)$. Now, since $\im(\rho)$ is a {\it countable}\footnote{The proposition is not true for uncountable unions. For example, the entire boundary at infinity of hyperbolic space $\partial_{\infty}\mathbb H^n$ (a union of an uncountable number of points) is discrete in the Tits topology but $(n-1)$-dimensional in the sphere topology.} union of the images of simplices $\rho(\sigma)$ its dimension is equal to the supremum of the dimensions of the simplices\footnote{This is the countable sum theorem in dimension theory. It says that for a normal space the dimension of a countable union of closed subsets is the supremum of the dimensions of the subsets (see \cite{hurewiczdim}).} in either the $\angle_x$ or the $\Td$-metric. We conclude that 
$$
\dim(\im\rho,\angle_x)=\sup_{\sigma}\dim(\rho(\sigma),\angle_x)=\sup_{\sigma}\dim(\rho(\sigma),\Td)=\dim(\im\rho,\Td)
$$ 
which proves the proposition.
\end{proof}

\begin{remark}
We initially defined $d$ via the sphere topology on $\partial_{\infty}$ because this is the topology for which the cone map $c_t$ is a homeomorphism onto its image. This is unsatisfying because the sphere topology does not reflect in any way the geometry of the universal cover. After all, the metric space $(\partial_{\infty},\angle_x)$ is just a round $(n-1)$-sphere. The present proposition is useful because the topological dimension of the Tits boundary is a geometrically meaningful quantity that can be (and often is\footnote{For symmetric spaces it is one less than the dimension of a maximal flat.}) much smaller than $n-1$.
\end{remark}

We end this section by giving the proof of Theorem \ref{collapsetheorem}.
\begin{proof}[Proof of Theorem \ref{collapsetheorem}]
The proof consists of several steps. First we prove the two properties of the homotopy $(\beta\circ\mu)_t$ mentioned in the bullets and then we explain how to use these properties to collapse $K$ to a $d$-dimensional subset.  

\subsection*{Claim} The homotopy $(\beta\circ\mu)_t$ is in the $(\epsilon+2\delta)$-thin part $\widetilde M_{\leq \epsilon+2\delta}$ for all $t$.

Recall from \ref{lateruse} that for every point $x\in\partial \widetilde M_{\leq\epsilon}$ there is a non-trivial element $\gamma\in\Gamma$ that is $(\epsilon+2\delta)$-small at $x$ and fixes $\mu(x)$.  
Because $\beta$ is $\Gamma$-equivariant $\gamma$ also fixes $\beta\circ\mu(x)$, so $\gamma$ is $(\epsilon+2\delta)$-small on the entire geodesic ray $[x,\beta\circ\mu(x))$. Since $(\beta\circ\mu)_t$ is defined by flowing along these geodesic rays for a time $t$, its image is in $\widetilde M_{\leq\epsilon+2\delta}$.
 
\subsection*{Claim} The homotopy $(\beta\circ\mu)_t$ diverges.

Let $F$ be a fundamental domain for the $\Gamma$-action on $\widetilde M_{\leq\epsilon}$ and note that it is compact. Since $\mu(F)$ is contained in a finite union of simplices of $\Delta_{\lfloor pAb\rfloor}$ and $\beta$ diverges on simplices, we conclude that $c_t\circ\beta\circ\mu\mid_F$ diverges. It is asymptotic to $(\beta\circ\mu)_t\mid_F$ so Lemma \ref{comparison divergence} that $(\beta\circ\mu)_t\mid_F$ diverges. But since $(\beta\circ\mu)_t$ is $\Gamma$-equivariant and $F$ is a fundamental domain, this actually implies that the entire $(\beta\circ\mu)_t$ diverges.

\subsection*{Claim} Collapsing $K$ to dimension $d$ in the thin part. 

Since $(\beta\circ\mu)_t$ diverges and is asymptotic to $c_t\circ\beta\circ\mu$, for any compact subset $K$ and large enough $t$ the straight-line homotopy between $(\beta\circ\mu)_t\mid_K$ and $c_t\circ\beta\circ\mu\mid_K$ is {\it inside} $\widetilde M_{\leq\epsilon}$ by Lemma \ref{comparison divergence}. Thus for a compact $K$ we can go along $(\beta\circ\mu)_t$ for a sufficiently large time and then take the straight line homotopy to $c_t\circ\beta\circ\mu$, and during this process the image of the set $K$ will stay inside the $(\epsilon+2\delta)$-thin part $\widetilde M_{\leq\epsilon+2\delta}$. Since $c_t$ is a diffeomorphism, the topological dimension of the image of $c_t\circ\beta\circ\mu$ is equal to $d$. This finishes the proof of the theorem. 
\end{proof}

\section{The importance of being Lipschitz\label{importanceofbeinglipschitz}}
Having discussed in more words than necessary the necessity of being divergent, we might seem to be displaying signs of triviality. On the contrary, dear readers, we have now realized for the first time in our lives the vital importance of being Lipschitz. 

As pointed out in the previous section, we need some control over the dimension $d =\dim(\im(\beta\circ\mu),\angle_x)$. The inconvenient truth that continuous maps can be space-filling means that if $\beta$ is only continuous, then $d$ can be as high as $(n-1)$ and all information on the topology of $\partial\widetilde{M}_{\leq\epsilon}$ will be lost. Therefore, we need $\beta$ to be Lipschitz because Lipschitz maps do not raise dimensions, so that we will have 
\begin{equation}
\label{lipschitzcontrol}
d\leq\dim(\Delta_{\lfloor pAb\rfloor})\leq\rank_{Ab}(\pi_1M)-1. 
\end{equation}

To get further constraints on the dimension $d$, it turns out to be important to understand non-degenerate simplices. 
A simplex $\lambda:\Delta^k\ra X$ is {\it non-degenerate} if $\lambda(\Delta^k)\not=
\lambda(\partial\Delta^k)$. Since $\im(\beta)$ is the union of all the non-degenerate simplices $\beta(\sigma)$, Lipschitzness of $\beta$ will imply that 
$$
d\leq\dim(\im(\beta),\angle_x)\leq\max\{k\mid\mbox{there is a non-degenerate }k\mbox{-simplex }\beta(\sigma)\},
$$ 
so understanding non-degenerate simplices may tell us something about $d$. 

Some simplices are better adapted for this than others. We will discuss three possibilities in the next section. Of course, the third one is always the one to be chosen in the end. It will be named, however, after Busemann.  

\section{Intermission and flyers on various types of simplices }\label{intermission} 
\subsection*{Summary of previous sections} We have found a systematic way (i.e. via $\beta$), for a given fine enough triangulation of $\partial\widetilde{M}_{\leq\epsilon}$,  of sending vertices of $\partial\widetilde{M}_{\leq\epsilon}$ to $\partial_\infty$. We need to fill in $\beta$ with simplices in $\partial_\infty$ that are Lipschitz and satisfy the criterion for divergence explained above. 
\newline
 
So let us now turn to the problem of actually constructing divergent simplices for the map $\beta$. There are at least three different ways to do it. The easiest method is to use geodesic simplices so we will mention it first.

\subsection{Geodesic simplices} 
To reassure the reader that the present discussion is not devoid of content, we note that one way to build $\beta$ is using geodesic simplices. Recall that a geodesic simplex $\sigma_k$ with (ordered set of) vertices $v_0,\dots,v_k$ mutually $\leq\pi/2$ apart is defined inductively as the iterated geodesic join $\sigma_k=\sigma_{k-1}*v_k$. So, by definition, a geodesic simplex is contained in the convex hull of its vertices. If the set of vertices has diameter $<\pi/2-\alpha$ then the geodesic simplex $\sigma_k$ is inside the ball $B_{\pi/2-\alpha}(v_i)$ centered at any vertex. Therefore, by (\ref{diam1}), if we form $\beta$ using geodesic simplices, then the resulting simplices with diverge. It is also easy to see from the definition that that the resulting map $\beta$ will be Lipschitz and $\Gamma$-equivariant. So, this $\beta$ will have all the properties listed in Subsection \ref{goodbeta} and all the results of Subsection \ref{perks} and Section \ref{importanceofbeinglipschitz} apply to it.
In particular, \underline{geodesic simplices are sufficient to establish the rank$_{Ab}(\pi_1M)$} \underline{and $\dim(\partial_{\infty},\Td)$ versions of Theorem \ref{other}}. However it is difficult to say anything about non-degenerate geodesic simplices and we do not know how to get the half-dimensional bound of Theorems \ref{analog} and \ref{factor} using geodesic simplices. 

\subsection{Barycentric simplices} These simplices were introduced in \cite{kleiner}. Suppose the diameter of the set $\{v_0,\dots,v_k\}$ is $<\pi/2-\alpha$. For each $t\in\Delta^k$, let $\lambda(t)$ be the unique minimum of the function 
$$
f_t(\cdot):=\sum_it_i\Td(\cdot,v_i)^2.
$$ 
This defines a map $\lambda:\Delta^k\ra\partial_{\infty}$ that is called the {\it barycentric simplex with vertices $v_0\dots,v_k$}. Points $x$ with $\Td(x,v_i)\geq\pi/2-\alpha$ for all $i$, are \underline{not} on the barycentric simplex $\lambda$. This is because any function of the form $f_t$ has
$$
f_t(x)\geq(\pi/2-\alpha)^2>f_t(v_i)
$$ 
so it does not have a minimum at $x$. Therefore $\lambda$ is contained in the $(\pi/2-\alpha)$-neighborhood of its vertex set $N_{\pi/2-
\alpha}(\lambda^{(0)})$. Barycentric simplices are Lipschitz and defined in an equivariant way, so we can use them to construct a ``divergent simplex'' map $\beta$. These simplicies are well adapted to understanding the Tits boundary with the Tits metric. Their key feature is that non-degenerate barycentric $k$-simplices must have $k\leq\dim(\partial_{\infty},\Td)$. In particular, we get from this that $d\leq\dim(\partial_{\infty},\Td)$, but we already knew that.

\subsection{\label{buseintro}Busemann simplices}
But, the main focus of the rest of this paper is to introduce a new way of constructing simplices at infinity which we call {\it Busemann simplices} and which can also be used to build a ``divergent simplex'' map $\beta$. Briefly, for a set of Busemann functions $h_0,\dots,h_k$ centered at vertices $v_0,\dots,v_k\in\partial_{\infty}$ with $\Td(v_i,v_j)<\pi/2-\alpha$, and a basepoint $x\in\widetilde M$ we let $\sigma_R(t)$ be the unique minimum of the function 
$$
f_t(\cdot)=\sum_it_ih_i(\cdot)
$$
on the sphere $S_R(x)$. This defines for each radius $R$ a map $\sigma_R:\Delta^k\ra S_R(x)$. Doing this for all the simplices in $\Delta_{\lfloor pAb\rfloor}$ gives a map 
$$
\beta_R:\Delta_{\lfloor pAb\rfloor}\ra S_R(x).
$$ 
As we will see in Section \ref{gradientsection}, this map is {\it Lipschitz} in the $\angle_x$-metric with a Lipschitz constant that {\it does not depend on $R$ or $x$.} This allows us to find a convergent subsequence $\beta_{R_i}\ra \beta:\Delta^k\ra\partial_{\infty}$ converging to a Lipschitz map $\beta$. Such a limit map $\beta$ is called ``the''\footnote{It depends on a sequence of scales $\{R_i\ra\infty\}$ that we choose once and for all.} {\it Busemann map} and its restriction to each simplex $\sigma$ is called a {\it Busemann simplex}. The limit map $\beta$ {\it does not depend on the choice of basepoint $x$} and it follows from this that $\beta$ is $\Gamma$-equivariant and Lipschitz in the Tits metric. The basepoint independence also leads, via (\ref{diam1}), to divergence for Busemann simplices. 
Therefore, this $\beta$ constructed out of Busemann simplices serves as a ``divergent simplex'' map and the results of Subsection \ref{perks} and Section \ref{importanceofbeinglipschitz} apply to it. 

Busemann simplices are particularly well adapted to studying the topology of the end. Let us say a bit about why this is the case. Let $G$ be a discrete group which preserves horospheres at the vertices of the simplex $\sigma$. Busemann simplices are constructed so that the group $G$ preserves horospheres {\it on the entire simplex}. One also has a good understanding of non-degenerate Busemann simplices in terms of the finite approximations $\beta_{R_i}(\sigma)$. These are the key features that lead to a bound
$$
hdim(G)+k+1\leq n
$$
for a non-degenerate Busemann $k$-simplex $\beta(\sigma)$. This dimension bound is our main technical result.
If $\sigma$ is a $k$-simplex in $\Delta_{\lfloor pAb\rfloor}$ then the group preserving horospheres at the vertices is {\it at least} $\mathbb Z^{k+1}$ so if $\beta(\sigma)$ is non-degenerate we get from the dimension bound that $2(k+1)\leq n$ and therefore 
$$
k\leq\lfloor n/2\rfloor-1.
$$
Since $\beta$ is Lipschitz and its image $\im(\beta)$ is the union of the non-degenerate $\beta(\sigma)$, we get the {\it half-dimensional collapse} phenomenon
$$
d\leq\lfloor n/2\rfloor -1.
$$ 

\section{Busemann simplices (Mostly metric properties)}\label{gradientsection}
We now construct the Busemann simplices introduced at the end of the last section. These are limits of singular simplices on spheres centered at a fixed point $x$ of finite radius $R_i$, i.e.

$$
\lim_{R_i\ra\infty}\sigma_{R_i,x}:\Delta^k\ra(\partial_{\infty},\angle_x).
$$ 
Since Busemann simplices are defined as limits, we need to make sure that such limits exist and must also be Lipschitz and divergent. Therefore, we will first give some preliminary estimates on the finite radius approximations. 
\subsection{Preliminary estimates}
For each $t = (t_0,\dots, t_k) \in \Delta^k$, take the convex combination
\[f_t \colon =t_0h_0+\dots+t_kh_k.\]

\subsubsection{\textbf{Infinitesmal Lipschitz estimate}}\label{inflipschitz}
If the $\nabla h_0,\dots,\nabla h_k$ are mutually at an angle $\leq\pi/2$ then the norm of the gradient is controlled by 
\begin{equation*}
{1\over\sqrt{k+1}} \leq |\nabla f_t|   \leq 1.
\end{equation*}
Therefore the radial projection from the convex hull of the $\nabla h_i$ to the unit sphere is at most $\sqrt{k+1}$-Lipschitz, so we get
$$
\angle(\nabla f_t,\nabla f_{t'})\leq \sqrt{k+1}|t-t'|_2.
$$ 
\subsubsection{\textbf{Radius-$R$ Lipschitz estimate}\label{lipschitzestimate}}
Fix a basepoint $x\in\widetilde M$. Since $f_t$ never attains its infimum, for each $R>0$ there is a unique point $\sigma_R(t)$ at which $f_t$ is minimal on the sphere $S_x(R)$. This defines a map
$$
\sigma_R:\Delta^k\ra S_x(R).
$$

\begin{lemma}\label{Lipschitz map}
For each $R >0$, the map \[\sigma_R \colon \Delta^k \rightarrow (S_x(R), \measuredangle_x)\] 
is  $2\sqrt{k+1}$-Lipschitz in the $L^2$-metric on $\Delta^k$. 

\end{lemma}
\begin{proof}
The main idea of this proof is in Figure 5, which the reader is encouraged to look at if they try to follow what is written next. 
\newline

Fix $R>0$ and $\delta$. Let $p_1 = \sigma_R (t)$ and $p_2 = \sigma_R(t+\delta)$. Let $\alpha = \measuredangle_x(p_1,p_2)$. Our goal is to bound $\alpha$ in terms of $\delta$. For each $i = 1,2$, 
\begin{itemize}
\item Let $\beta_i$ be the angle at $p_i$ between $-\nabla f_t$ and $-\nabla f_{t+\delta}$,
\item Let $\alpha_i = \measuredangle_{p_i} (x, p_{i+1})$, where addition in $i$ is taken mod 2,
\item Let $\mu_1$ be the angle at $p_1$ between the $-\nabla f_{t+\delta}$ and the tangent to the geodesic from $p_1$ to $p_2$, and
\item Let $\mu_2$ be the angle at $p_2$ between the $-\nabla f_{t}$ and the tangent to the geodesic from $p_2$ to $p_1$. 
\end{itemize}

\begin{figure}
\label{gradientfigure}
\centering
\includegraphics[scale=0.11]{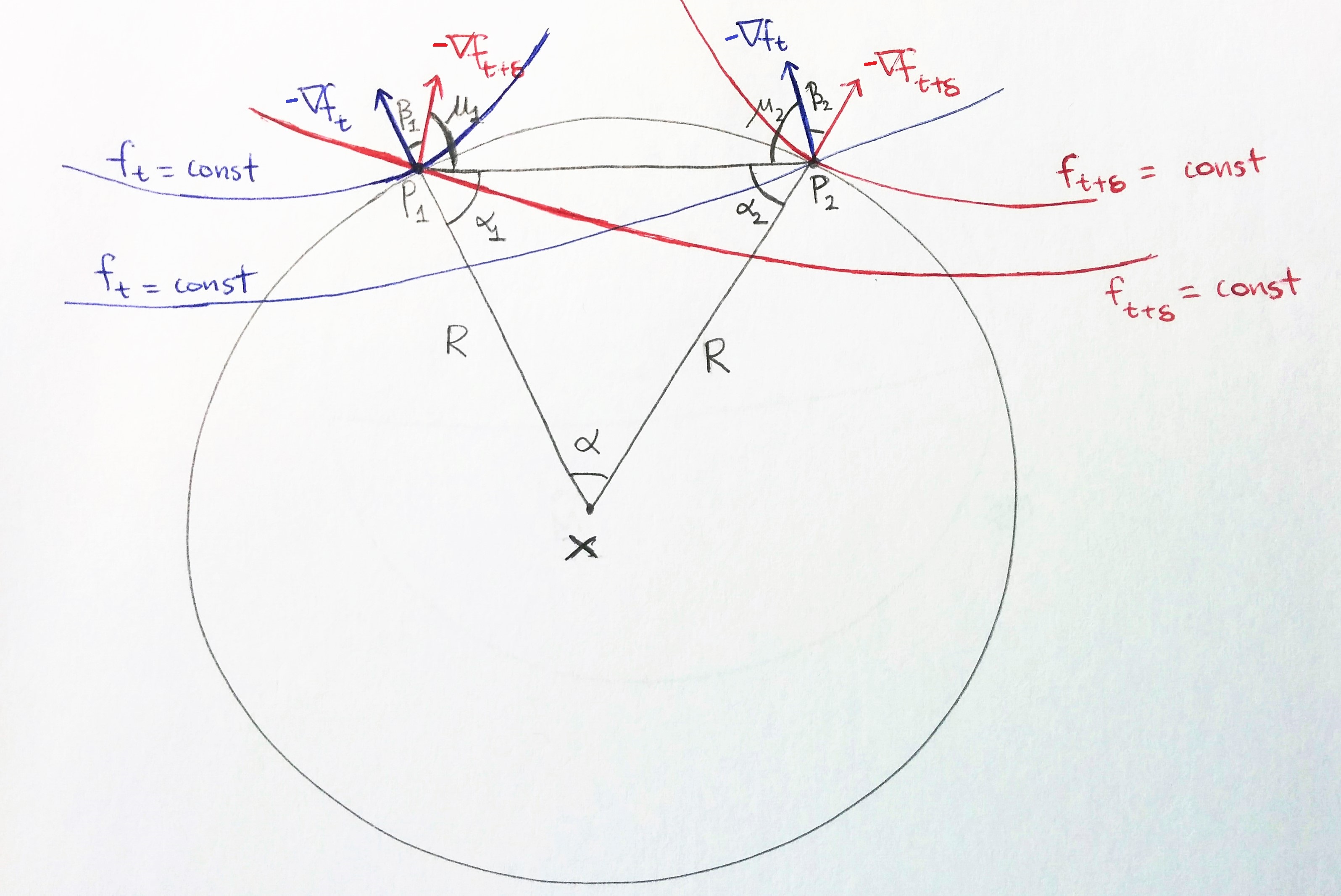}
\caption{}
\end{figure}

Then 
\[ \alpha_1 + \mu_1 + \beta_1 \geq \pi \]
since at $p_1$ the vector $-\nabla f_t$ is parallel to the tangent vector to the geodesic from $x$ to $p_1$ because they are both orthogonal to the level set $f_t = f_t(p_1)$. Similarly,
\[ \alpha_2 + \mu_2 + \beta_2 \geq \pi. \]
Therefore, 
\[\beta_1 +\beta_2 \geq (\pi - \alpha_1 -\alpha_2) + (\pi - \mu_1 - \mu_2).\]
Now, $\alpha$, $\alpha_1$ and $\alpha_2$ are the three angles of the triangle $\triangle xp_1p_2$. Thus,
\[ \alpha +\alpha_1 +\alpha_2 \leq \pi, \quad  \text{so}\quad \alpha \leq \pi - \alpha_1 -\alpha_2.\]
Hence, 
\[ \beta_1 +\beta_2 \geq \alpha + (\pi - \mu_1 - \mu_2).\]
Next, we show that $\mu_1 \leq \pi/2$ and $\mu_2 \leq \pi/2$, so that $\beta_1 +\beta_2 \geq \alpha$, which we can use to bound $\alpha$ in terms of $\delta$. 

To see that $\mu_1 \leq \pi/2$, observe that $p_1$ does not belong to the sublevel set $f_{t+\delta} \leq f_{t+\delta}(p_2)$ since the sphere $S_x(R)$ lies on the other side of the level set $f_{t+\delta} = f_{t+\delta}(p_2)$. Therefore, $f_{t+\delta}(p_1) > f_{t+\delta}(p_2)$, so  $p_2$ is contained in the sublevel set $f_{t+\delta} \leq f_{t+\delta}(p_1)$. Since at $p_1$ the vector $-\nabla f_{t+\delta}$ is orthogonal to the level set $f_{t+\delta} = f_{t+\delta}(p_1)$ it follows that $\mu_1 \leq \pi/2$. Similarly, we see that $\mu_2 \leq \pi/2$ and we obtain that 
\[\alpha \leq \beta_1 +\beta_2.\]
Since each $\beta_i\leq\sqrt{k+1}|\delta|_2$ by \ref{inflipschitz}, it follows that $\alpha\leq 2\sqrt{k+1}|\delta|_2$.
\end{proof}

\subsection{Definition of Busemann simplices}
The approximations $\sigma_R:\Delta^k\ra\ S_{x}(R)$ depend on the choice of basepoint $x$. For the moment, let us emphasize this dependence and denote them by $\sigma_{R,x}$. The Lipschitz estimate \ref{lipschitzestimate} implies there is a sequence of radii $R_i\ra\infty$ for which the maps $\sigma_{R_i,x}$ converge in $\widetilde M\cup\partial_{\infty}$ to a map 
$$
\sigma:=\lim_{R_i\ra\infty}\sigma_{R_i,x}:\Delta^k\ra(\partial_{\infty},\angle_x).
$$ 
We call any such map $\sigma$ a {\it Busemann simplex}. 

\subsection{Properties of Busemann simplices}
\subsubsection{\textbf{Independence of basepoin}t\label{bpindep}}
Next we will show that $\sigma$ does not depend on the choice of basepoint $x$ but only on the sequence of radii $R_i$. The (easy) estimate we need for this is
\begin{equation}
\label{nobp}
d_{\widetilde M}(\sigma_{R,x}(t),\sigma_{R,y}(t))\leq D+\sqrt{2DR+D^2}
\end{equation}
where $D:=d(x,y)$. The keys point is that $R$ appears with a square root sign in the estimate (\ref{nobp}). It follows from this that $\{\sigma_{R_i,x}\}$ converges if and only if $\{\sigma_{R_i,y}\}$ converges, and that both converge to the same $\sigma$.
\begin{remark}
Here is what we are using: If $x_i\ra\xi$ and ${d(y_i,x_i)\over d(x_0,x_i)}\ra 0$ then $y_i\ra\xi$.
\end{remark}   
\begin{proof}[Proof of estimate (\ref{nobp})]
We will use the following notation. Denote by $x_i$ and $y_i$ the closest point projections of $x$ and $y$ onto the sublevel set $\{f_t\leq c_i\}$, respectively. Suppose that $d(x,x_1)=R$ and $d(y,y_2)=R$. In other words, $x_1=\sigma_{R,x}(t)$ and $y_2=\sigma_{R,y}(t)$.  
Without loss of generality $c_2\leq c_1$. 
\begin{figure}
\centering
\includegraphics[scale=0.6]{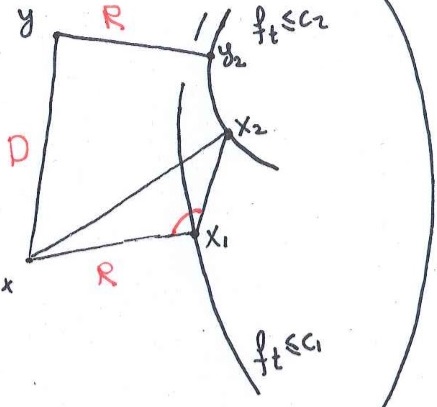}
\end{figure}
We need to bound $d(x_1,y_2)$ in terms of $D$ and $R$. 
First, note that 
\begin{eqnarray*}
d(x_1,y_2)&\leq& d(x_1,x_2)+d(x_2,y_2)\\
&\leq& d(x_1,x_2)+d(x,y)\\
&=&d(x_1,x_2)+D,
\end{eqnarray*}
because closest point projection to $\{f_t\leq c_2\}$ is a contraction. 
It remains to bound $d(x_1,x_2)$. To do this, note that $c_2\leq c_1$ implies $\angle_{x_1}(x,x_2)\geq\pi/2$, so
\begin{eqnarray*}
d(x_1,x_2)^2+R^2&=&d(x_1,x_2)^2+d(x,x_1)^2\\
&\leq& d(x,x_2)^2\\
&\leq& d(x,y_2)^2\\
&\leq&(D+R)^2,
\end{eqnarray*}
where the first inequality is by triangle comparison with an obtuse triangle in Euclidean space, the second inequality is because $x_2$ is the closest point to $x$ on $\{f_t\leq c_2\}$ while $y_2$ is just some point on this sublevel set, and the third is the triangle inequality. Simplifying, we get $d(x_1,x_2)\leq\sqrt{2DR+D^2}$. 
\end{proof}

\subsubsection{\textbf{Busemann simplices are Lipschitz in the Tits metric}}
A consequence of the basepoint independence is that $\sigma:\Delta^k\ra\partial_{\infty}$ is $2\sqrt{k+1}$-Lipschitz in the $\angle_y$ metric for {\it any} point $y\in\widetilde M$, since we can use finite approximations $\{\sigma_{R_i,y}\}$ based at $y$ to get the limit simplex $\sigma$, and for these Lemma \ref{Lipschitz map} gives the $2\sqrt{k+1}$-Lipschitz estimate. Therefore, Busemann simplices are $2\sqrt{k+1}$-Lipschitz in the $\angle$-metric $\angle=\sup_{y\in\widetilde M}\angle_y$. Since for distances $<\pi$ the Tits metric agrees with the $\angle$-metric (in the sense that $\angle=\min(\Td,\pi)$, see Appendix \"A)  Busemann simplices are also $2\sqrt{k+1}$-Lipschitz in the Tits metric.  
\subsubsection{\textbf{Diameter bound}}
Busemann simplices are small in the following sense.\begin{lemma}
\label{diambound}
If the set of vertices $\sigma^{(0)}$ has $\Td$-diameter $<\pi/2-\alpha$ then the entire Busemann simplex $\sigma$ is contained in a ball of radius $\pi/2-\alpha$ centered at any one of the vertices, i.e. 
$$
\Td(\overline h_i,\overline h_j)<\pi/2-\alpha\mbox{ for all }i,j\implies\Td(\overline h_i,\sigma(t))\leq\pi/2-\alpha.
$$ 
\end{lemma}
\begin{figure}
\centering
\includegraphics[scale=0.6]{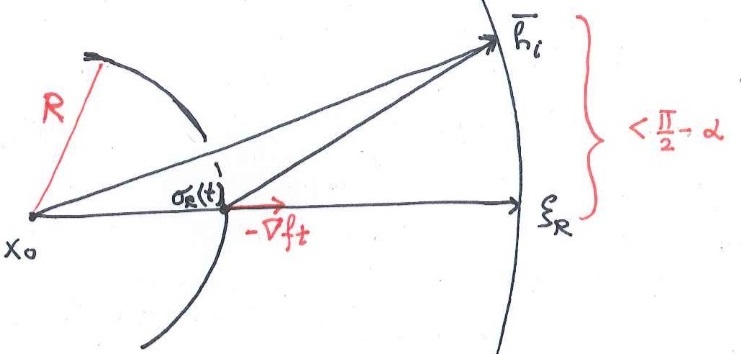}
\end{figure}
\begin{proof}
Fix a basepoint $x_0$ and let $\sigma_R(t):=\sigma_{R,x_0}(t)$ be the finite approximation based at $x_0$. Extend the geodesic segment $[x_0,\sigma_R(t)]$ to a geodesic ray $[x_0,\xi_R)$ with endpoint $\xi_R\in\partial_{\infty}$. Then, by construction, we have
\begin{equation}
\label{compare}
\angle_{x_0}(\sigma_R(t),\overline h_i)\leq\angle_{\sigma_R(t)}(\xi_R,\overline h_i).
\end{equation}
Now, let $f_t=t_0h_0+\dots+t_kh_k$. At the point $\sigma_R(t)$ the negative of the gradient $-\nabla f_t$ points at $\xi_R$. Since the gradient of $f_t$ is a convex combination of gradients of the Busemann functions $h_i$, we see that in the $\angle_{\sigma_R(t)}$-metric the point $\xi_R$ is in the convex hull of the set $\{\overline h_0,\dots,\overline h_k\}$. Since this set has diameter $<\pi/2-\alpha$ in the $\angle_{\sigma_R(t)}$-metric, its convex hull is contained in an $\angle_{\sigma_R(t)}$-metric $(\pi/2-\alpha)$-ball around each vertex $\overline h_i$, and therefore 
$$
\angle_{\sigma_R(t)}(\xi_R,\overline h_i)<\pi/2-\alpha.
$$ 
Using the earlier inequality (\ref{compare}) and taking the limit as $R_i\ra\infty$ we get 
$$
\angle_{x_0}(\sigma(t),\overline{h_i})\leq\pi/2-\alpha.
$$
Since this holds for every basepoint $x_0$, we get the same bound in the Tits metric. 
\end{proof}
\subsubsection{\textbf{Invariance of horospheres on a Busemann simplex}\label{invarianthorospheres}} 
Let 
$$f=t_0h_0+\dots+t_kh_k
$$ 
be a convex combination of Busemann functions with $\Td(h_i,h_j)\leq \pi/2$. Then $f$ is a convex function that does not attain its infimum. If we denote 
$$
s_i:=f(\sigma_{R_i}(t)),
$$ 
then, by definition, $\sigma_{R_i}(t)$ is the closest point projection of the basepoint $x_0$ to the sublevel set $\{f\leq s_i\}$.
It follows from $3.9$ of \cite{ballmangromovschroeder} that the limit
\begin{equation}
\label{buselimit}
\hat f(x):=\lim_{i\ra\infty}d(x,\{f\leq s_i\})-d(x_0,\{f\leq s_i\})
\end{equation}
exists and is equal to a Busemann function centered at $\sigma(t)$. 
If the Busemann functions $h_i$ are $G$-invariant then $f$ is also (obviously) and inspecting the formula (\ref{buselimit}) we see that $\hat f$ is, as well. In summary, we get 
\begin{lemma}
\label{preservedhoro}
If $G$ preserves horospheres at all the vertices $\overline h_i$ of a Busemann simplex $\sigma:\Delta\ra\partial_{\infty}$, then $G$ preserves horospheres at every point $\sigma(t)$ of $\sigma$. 
\end{lemma}
\begin{remark}
We do not know whether the same is true for geodesic or barycentric simplices. For those simplices, we only know that points on the simplex are fixed by $G$ but we do not know that horospheres centered at those points are $G$-invariant. 
\end{remark}

\section{The Busemann map\label{busemannmapsection}}
In this section, we will construct a map $\beta$ that is Lipschitz and satisfying the properties listed in Subsection \ref{goodbeta} whose restriction to each simplex $\sigma$ of $\Delta_{\lfloor pAb\rfloor}$ is a Busemann simplex.
\subsection{Construction}
Define $\beta$ on the vertices of $\Delta_{\lfloor pAb\rfloor}$ by 
$$
\beta([A]):=\xi_{[A]}.
$$
Fix a basepoint $x$. Doing the construction from \ref{lipschitzestimate} of the finite approximations $\sigma_R$ on each simplex of $\Delta_{\lfloor pAb\rfloor}$ gives a map 
$$
\beta_R:\Delta_{\lfloor pAb\rfloor}\ra S_x(R).
$$ 
This map is $3k$-Lipschitz\footnote{Since $2\sqrt{k+1}\leq 3k$, we will write $3k$-Lipschitz from now on.} in the $\angle_x$-metric: It is enough to check this on paths, where it follows from the fact (Lemma \ref{Lipschitz map}) that it is $3k$-Lipschitz on each simplex. By Arzela-Ascoli, we can take a sequence of radii $R_i\ra\infty$ for which $\beta_{R_i}$ converge to a $3k$-Lipschitz map 
\begin{equation}
\beta:\Delta_{\lfloor pAb\rfloor}\ra(\partial_{\infty},\angle_x).
\label{busemannmap}
\end{equation} 
Since the Lipschitz constant does not depend on $x$, this map is also Lipschitz in the Tits metric. We call it ``the'' {\it Busemann map}.

\subsection{Equivariance}
For any $\gamma\in\Gamma$ and simplex $\tau$ in $\Delta_{\lfloor pAb\rfloor}$, the Busemann simplices $\beta:\gamma\tau\ra\partial_{\infty}$ and $\gamma\beta:\tau\ra\partial_\infty$ have the same ordered set of vertices, so they are equal, i.e. $\gamma\beta=\beta\gamma$. Therefore $\beta$ is $\Gamma$-equivariant.

\subsection{Divergence\label{divsimplices}} If $\sigma$ is a simplex in $\Delta_{\lfloor pAb\rfloor}$ with vertices $[A_0]<\dots<[A_k]$, then the Busemann simplex $\beta(\sigma)$ is fixed pointwise by $A_k$ because $\beta$ is $\Gamma$-equivariant, and contained in a $(\pi/2-\alpha)$-neighborhood of a vertex $\beta([A_k])$ by Lemma \ref{diambound}. The group $C_{A_k}$ preserves horospheres at this vertex, so we get by Corollary \ref{divergencecriterion} that the Busemann simplex $\beta(\sigma)$ diverges in $M$. 

\begin{remark}
Here is a slightly different way to see that the Busemann simplices diverge: Since $\beta$ is continuous, the image $\beta(\sigma)$ is compact in the Tits metric. Cover it with finitely many $(\pi/2-\alpha)$-balls. Since $C_{A_k}$ preserves horospheres on the {\it entire} Busemann simplex (Lemma \ref{preservedhoro}) we can apply Proposition \ref{divergentsectorsearly} to the centers of each of these balls and conclude that the Busemann simplex diverges.
\end{remark}

\section{(Non)-degeneracy and consequences\label{nondegeneracysection}}
We saw earlier how it is important to have controls over the dimension $d =\dim(\im(\beta\circ\mu),\angle_x)$. Since Busemann simplices are Lipschitz, $d$ is bounded by the maximum of the dimensions of \emph{non-degenerate} Busemann simplices in $\beta$.  The dimension of a non-degenerate simplex is in turn  bounded by the number of its vertices (minus one). 
Thus, it would be good to know, especially if one wants to be efficient, when a Busemann simplex is degenerate and what we can do with non-degenerate ones. The goal of this section is to address these and to set things up for the next sections, where we will bound the number of vertices of non-degenerate simplices of $\beta$ by $\lfloor n/2\rfloor$. 

\subsection{Non-degenerate simplices and linearly independent vectors}
A simplex $\lambda:\Delta\ra X$ is {\it non-degenerate} if the image $\lambda(\Delta)$ is not contained in the image of the boundary $\lambda(\partial\Delta)$. Any point $x\in\lambda(\Delta)\setminus\lambda(\partial\Delta)$ is called a {\it non-degenerate point}. The meaning of non-degeneracy for the finite approximations $\sigma_R$ is, partly, explained by the following simple lemma. 
\begin{lemma}
Suppose $\sigma$ is a Busemann simplex and let $\sigma_R$ be a finite approximation of it. If $x\in\sigma_R(\Delta)$ is non-degenerate, then the gradient vectors $\{\nabla h_0,\dots,\nabla h_k\}$ are linearly independent at $x$. \end{lemma}
\begin{proof}
Note that $x=\sigma_R(t)$ for some $t=(t_0,\dots,t_k)\in\Delta$, and look at the convex combination $f_t=t_0h_0+\dots+t_kh_k$. At the point $x$, the gradient 
$$
\nabla f_t=t_0\nabla h_0+\dots +t_k\nabla h_k
$$ 
is perpendicular to the sphere $S_{x_0}(R)$. 
Suppose the $\{\nabla h_i\}$ are linearly dependent at $x$, and write down a linear dependence relation as
$$
\rho:=s_0\nabla h_0+\dots+s_k\nabla h_k=0.
$$
In this relation, at least one of the coefficients $s_j$ is positive. So, there is a smallest $\epsilon\geq 0$ such that $t_i-\epsilon s_i=0$ for some $i$. After reordering the indices, we may assume this happens for $i=0$. Then $t_0-\epsilon s_0=0$ and $t_i-\epsilon s_i\geq 0$ for all $i$, so at $x$ we have
\begin{eqnarray*}
\nabla f_t&=&\nabla f_t-\epsilon\rho\\
&=&(t_1-\epsilon s_1)\nabla h_1+\dots+(t_k-\epsilon s_k)\nabla h_k\\
&=&a_1\nabla h_1+\dots+a_k\nabla h_k
\end{eqnarray*}
for some non-negative constants $\{a_i\}_{i=1}^k$. Setting 
$$
t':=\left(0,{a_1\over |a|_1},\dots,{a_k\over |a|_1}\right)\in\partial\Delta,
$$ 
the above equation can be rewritten as 
$$
\nabla f_t=|a|_1\nabla f_{t'}
$$ 
at the point $x$, which implies that $x=\sigma_R(t')$. So, $x$ is a degenerate point. 
\end{proof}

\begin{remark}
The proof of this lemma is less delicate than it may appear at first glance. All we are doing is finding an intersection point $a$ of the line $t+\mathbb Rs$ with the boundary of the positive ``octant'' $\partial((\mathbb R_+)^{k+1})$
and observing that the boundary point $a/|a|_1\in\partial\Delta$ defined by this is mapped to $\sigma_R(t)$. 
\end{remark}
The map $\sigma_R:\Delta\ra\widetilde M$ may be very far from an embedding. However, we will see next that the situation is better if we restrict to the {\it preimages of non-degenerate points}. Putting these together for all $R$ forms the open\footnote{The set $(\mathbb R^+\times\Delta)^{nd}_{\sigma}$ is open because being non-degenerate is an open condition.} set 
\begin{equation}
\label{prenondeg}
(\mathbb R^+\times\Delta)^{nd}_{\sigma}:=\{(R,t)\in\mathbb R^+\times\Delta\mid \sigma_R(t)  \mbox{ is non-degenerate}\}.
\end{equation}
\begin{corollary}
\label{injective}
The map 
\begin{eqnarray}
(\mathbb R^+\times\Delta)^{nd}_{\sigma}&\ra&\widetilde M,\\
(R,t)\hspace{0.7cm}&\mapsto&\sigma_R(t)
\end{eqnarray} 
is injective.
\end{corollary}
\begin{proof}
Note that $\sigma_R(t)=x=\sigma_R(t')$ gives the linear relation
$$
t_0\nabla h_0+\dots+t_k\nabla h_k=c(t'_0\nabla h_0+\dots+t'_k\nabla h_k)
$$ 
at the point $x$. If $x$ is non-degenerate, the previous lemma implies this relation is trivial. So, we must have $t=t'$. This proves the corollary.
\end{proof}
\subsection{Busemann cone} In section \ref{dimensionboundsection} it will often be useful to put all the (images of) Busemann simplices $\sigma_R$ based at a single point $x_0$ together.
The {\it Busemann cone} of $\sigma$ (based at $x_0$) is the set  
$$
\sigma_{>0}:=\bigcup_{R>0}\sigma_R(\Delta)
$$
of all points in $\widetilde M$ 
that lie on $\sigma_R=\sigma_{R,x_0}$ for some radius $R$. Sometimes it is convenient to also include the basepoint $x_0$. The result is then a closed set
$$
\sigma_{\geq 0}:=\sigma_{> 0}\cup\{x_0\}.
$$ 
As a matter of convention, we declare that $x_0$ is a degenerate point.
\subsection{A sequential criterion for degeneracy\label{sequential}}
Suppose $\sigma_{R_i}\ra\sigma$ is a Busemann simplex. Since the sequence of radii $\{R_i\}$ is chosen somewhat non-canonically, sometimes (in life) we cannot avoid dealing with points on the Busemann cone that are not on $\sigma_{R_i}$. Therefore, it will be useful later to have the following lemma which gives a criterion for degeneracy of a point $\sigma(t)$ in terms of a sequence of points on the Busemann cone $\{q_i\}$ converging to $\sigma(t)$. Note that the $q_i$'s need not belong to $\sigma_{R_i}$.
\begin{lemma}
\label{sublinear}
Let $\sigma:\Delta\ra\partial_{\infty}$ be a Busemann simplex. Then, a point $\sigma(t)$ is degenerate if and only if there is a sequence $\{q_i\}_{i=1}^{\infty}$ of degenerate points on the Busemann cone $q_i\in\sigma_{\geq 0}$ with 
\begin{equation}
\label{sublin}
\lim_{i\ra\infty}{d(q_i,\sigma_{R_i}(t))\over R_i}=0.
\end{equation}
\end{lemma}
\begin{proof}
The sublinearity (\ref{sublin}) implies that $q_i\ra\sigma(t)$, but this by itself does not yet mean that $\sigma(t)$ is degenerate. To prove degeneracy, we need to find a point $t'\in\partial\Delta$ satisfying $\sigma(t)=\sigma(t')$. We will now do this.
\begin{figure}
\label{closestpoint}
\centering
\includegraphics[scale=0.6]{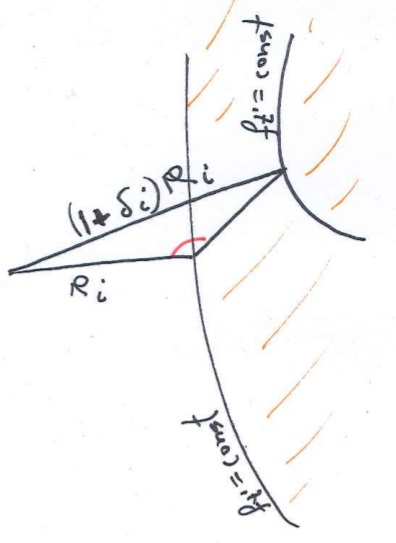}
\end{figure}

Since $q_i$ are degenerate points on the Busemann cone, we have 
$$
q_i=\sigma_{R_i'}(t'_i)
$$ 
for $t'_i\in\partial\Delta$ and radii $R_i'\ra\infty$. Passing to a subsequence, we may assume 
$$
t'_i\ra t'\in\partial\Delta.
$$ 
Now, the sublinearity (\ref{sublin}) implies that ${|R_i'-R_i|\over R_i}\ra 0$. Therefore we can write 
$$
R_i'=(1+\delta_i)R_i\mbox{ for a sequence }\delta_i\ra 0.
$$ 
Since both $\sigma_{R_i}(t')$ and $\sigma_{R'_i}(t')$ are obtained as closest point projections of $x_0$ to (different) sublevel sets of the {\it same} convex function $f_{t'}:=t_0'h_0+\dots+t_k'h_k$, comparison with an obtuse Euclidean triangle (see \ref{projes}) gives
$$
{d(\sigma_{R_i'}(t'),\sigma_{R_i}(t'))\over R_i}\leq \sqrt{2\delta_i+\delta_i^2}\ra 0.
$$
Therefore the sequence $\{q_i\}=\{\sigma_{R'_i}(t')\}$ converges to the same point as the sequence $\{\sigma_{R_i}(t')\}$, and that point is
$$
\sigma(t)=\lim q_i=\sigma(t')\in\partial\Delta.
$$
\end{proof}

\section{A dimension bound\label{dimensionboundsection}}
Busemann simplices provide a way to connect points at infinity whose horospheres are preserved. Our goal in this section is to relate the dimension of a group $G$ acting on $\widetilde{M}$ preserving some horospheres and the dimension of the Busemann simplex with vertices the centers of those horospheres. One expects these two dimensions to be complementary. Theorem \ref{dimensionbound} gives the expected bounds on these dimensions. This is responsible for the $n/2$ bound in the main theorems and is the climax of this paper. We will end this section by giving the proof of Theorem \ref{gdimensionboundintro} in the introduction.
\begin{theorem}
\label{dimensionbound}
If $\Fix^0(\mathbb Z^r)$ has a non-degenerate Busemann $k$-simplex $\sigma$ then 
\begin{equation}
\dim\widetilde M\geq k+1+r.
\end{equation} 
\end{theorem}
\begin{remark}
The method of proof of Theorem \ref{dimensionbound} applies to any subgroup $G<\Gamma$ that preserves horospheres on a non-degenerate Busemann $k$-simplex (see \ref{othergroups}). But, for now it is helpful to focus on the case $G=\mathbb Z^r$. This is all we need for half-dimensional collapse.  
\end{remark}

This is not quite a natural argument, so before giving the technical proof, we will give a description of the proof of Theorem \ref{dimensionbound} in Subsection \ref{secondproblemsandsolutions}.

\subsection{Problems and solutions}\label{secondproblemsandsolutions}
A natural approach to obtaining the bound $\dim\widetilde M\geq r+k+1$ is to show that each intersection of horospheres has dimension $\geq r$ and that there is a $(k+1)$-dimensional family of such intersections of horospheres. Consider the following \emph{parameter space for intersections of horospheres}. 
\newline

Pick representatives $h_i$ for the vertices of the non degenerate Busemann simplex $\sigma:\Delta^k\ra \Fix^0(\mathbb Z^r)$, such that $h_i(x_0)=0$, and look at the ``horospherical coordinates'' map
\begin{eqnarray}
\vec{h}\colon \widetilde M&\ra&\mathbb R^{k+1},\\
x&\mapsto&(h_0(x),\dots,h_k(x)).
\end{eqnarray} 
The map $\vec{h}$ is (obviously) $\mathbb Z^r$-invariant, and its image $\vec{h}(\widetilde M)$ in $\mathbb R^{k+1}$ is the parameter space for all possible intersections of horospheres $\cap_{i=0}^k\{h_i=b_i\}$.
\newline

Let $T_b$ be the intersection of horospheres $\vec{h}^{-1}(b)$. One can hope to show that each intersection of horospheres $T_b$ has dimension $\geq r$ by letting $\mathbb{Z}^r$ act on it. Ideally, $T_b$ is an $(n-k-1)$-dimensional submanifold that is contractible. It will then follow that $r\leq (n-k-1)$, so we obtain $\dim \widetilde{M} = n \geq r+1+k$. 
\newline

\noindent
\textbf{The main problem.} However, in order for this intersection of horospheres to be an $(n-k-1)$-dimensional submanifold, traditionally the gradients of the corresponding Busemann functions need to be linearly independent everywhere on the submanifold. If the gradients are linearly independent everywhere on the whole manifold $\widetilde{M}$, then the submanifold will be homotopy equivalent to $\widetilde{M}$ and, thus, is contractible. This leads us to the \underline{main} p\underline{roblem}: these gradients need not be linearly independent (even on just the intersection of the horospheres). So we cannot show that $T_b$ has dimension $(n-k-1)$. Neither can we bound its dimension from below by $r$. The map $\vec{h}$ is $\mathbb Z^r$-invariant, so the group $\mathbb Z^r$ acts on each $T_b$, but this by itself is not enough to bound the dimension of $T_b$ from below by $r$. Pessimistically speaking, $T_b$ could be discrete. So \emph{we will not try to prove anything about $T_b$}. Instead, we will try to bound the dimension of
\[(\widetilde{M}-T_o)/\mathbb{Z}^r,\]
for some fixed $b =o\in \mathbb{R}^{k+1}$. Note that $T_o$ is a closed subset of $\widetilde{M}$, so $ (\widetilde{M}-T_o)$ is a  manifold of dimension $n$. 
\newline

This problem has a topological \underline{solution} if we can show $ (\widetilde{M}-T_o)/\mathbb Z^r$ has a nontrivial homology class in dimension $(k+r)$ because it will imply that $n\geq(k+r+1)$ since $ (\widetilde{M}-T_o)/\mathbb Z^r$ is noncompact. 
\newline

\noindent
\textbf{Finding a $(k+r)$-homology class in $(\widetilde{M}-T_o)/\mathbb{Z}^r$.} Note that in the ideal case mentioned above, such a nontrivial class exists and it can be represented by a map $\mathbb{S}^k\times\mathbb{T}^r \rightarrow (\widetilde{M}-T_o)/\mathbb Z^r$. We will try to find such a map $\mathbb{S}^k\times\mathbb{T}^r \rightarrow (\widetilde{M}-T_o)/\mathbb Z^r$ in general and show that it is nontrivial in homology by mimicking barely the ideal case. Therefore, it is worth describing what happens ideally first.
\newline

\noindent
\textit{The ideal case.}  Ideally, the gradients $\nabla h_0, \nabla h_1, ..., \nabla h_k$ are linearly independent everywhere on $\widetilde M$, so $V := \vec{h}(\widetilde{M})$ is an open set of $\mathbb{R}^{k+1}$ and $\widetilde{M}$ has a \emph{product} structure $\widetilde{M} \cong V \times T_o$, with $\vec{h}$ being the projection onto the factor $V$. Clear, $(\widetilde M-T_o)$  is also a product.
\[(\widetilde M-T_o) = (V-\{o\})\times T_o.\] 
To get a map $\mathbb{S}^k\times\mathbb{T}^r \rightarrow(\widetilde M-T_o)/\mathbb{Z}^r$, we find a map from $\mathbb{S}^k$ and a map from $\mathbb{T}^r$ and take the product of them. To get a map from $\mathbb{S}^k$, we take a small $k$-sphere centered at $o$ in $V$ and naturally obtain a map 
\[ \mathbb{S}^k \rightarrow (\widetilde M-T_o)/\mathbb{Z}^r\]  
by first including  $\mathbb{S}^k$  into the factor $(V-\{o\})$ and then taking the quotient by the action of $\mathbb{Z}^r$. Next, to get a map $\mathbb{T}^{r} \rightarrow (\widetilde M-T_o)/\mathbb{Z}^r$, we first take a $\mathbb{Z}^r$-equivariant map $\mathbb{R}^r\rightarrow \widetilde{M}$ and then compose it with the projection, using the product structure $\widetilde{M} \cong V \times T_o$, to the fiber $T_{o}$. This gives a $\mathbb{Z}^r$- equivariant map $\mathbb{R}^r \rightarrow T_o$. So after taking the quotient by $\mathbb{Z}^r$, we get a map $\mathbb{T}^{r} \rightarrow T_o/\mathbb{Z}^r$. The product of these two maps gives a nontrivial\footnote{Nontriviality of this homology cycle is not hard but we will not explain here because nothing in ideal worlds require an explanation.} homology cycle $\mathbb{S}^k\times\mathbb{T}^r \rightarrow(\widetilde M-T_o)/\mathbb{Z}^r$. All of this works because of the product structure on $(\widetilde{M} - T_o)$ created by $\nabla h_i$'s.
\newline

\noindent
\textit{In real life}, the gradients $\nabla h_0, \nabla h_1, ..., \nabla h_k$ need not be linearly independent everywhere so they do not give $\widetilde{M}$ a product structure. One can attempt to define a map $\mathbb{S}^k\times \mathbb{T}^r \rightarrow (\widetilde{M} - T_o)/\mathbb{Z}^r$ as follows. First, take a $\mathbb{Z}^r$-equivariant map $f \colon \mathbb{R}^r\rightarrow \widetilde{M}$ that takes $0 \mapsto x_0$, where $x_0$ is a fixed basepoint at which we will take the Busemann cone later. Then for each $T_b\ne \emptyset$, compose $f$ with the closest point projection to $T_b$. If we can do this for $b$ taking values in a $k$-sphere in $(\vec{h}(\widetilde{M})-o)$, then we obtain a map $\mathbb{S}^k\times \mathbb{T}^r \rightarrow (\widetilde{M} - T_o)/\mathbb{Z}^r$. Specifically, if $\vec{h}(\widetilde{M})$ has nonempty interior (as a subset of $\mathbb{R}^{k+1}$), then we can take $o$ to be an interior point and let $b$ take value in a $k$-sphere surrounding $o$. A p\underline{roblem} with this approach is that $T_b$ needs not be convex so closest point projection is not well-defined. Nevertheless, intersections of horo\emph{balls} are convex, so the \underline{solution} is to project $f(\R^r)$ onto 
\[\widehat{T}_b:= \{x \in \widetilde{M}\; |\; h_i(x) \leq b_i, \; i =0,1,..., k\}\]  
instead. However, the price we pay for this is that the projection of $f(\R^r)$ needs not land in $T_b$. Neither should it even be close to $T_b$. So this is a p\underline{roblem}. 
\newline

But there is a \underline{solution}, which is to use the magic\footnote{We are not able to come up with an explanation.} of the \emph{Busemann cone} $\sigma_{>0}$. We will show later this section\footnote{Lemma \ref{horocoordinates}} that if $b\in\vec{h}(\sigma_{>0})$, then the closest-point projection $p(b,x_0)$ of $x_0 = f(0)$ onto the intersection $\widehat{T}_b$ of horoballs  is actually contained in the intersection of horospheres $T_b$, or in other words, $\vec{h}(p(b,x_0)) = b$. It follows that the projection of $f(\mathbb{R}^r)$ is contained in an $L$-neighborhood of $T_b$. This is because the action of $\mathbb{Z}^r$ is cocompact on $f(\R^r)$ and closest-point projections are distance non-increasing. A result of this is that $L$ is independent of $b$ once we fix $f$, which suggests doing the following: if $\vec{h}(\sigma_{>0})$ contains a ball of radius larger than $L$, then one can define a  map $\mathbb{S}^k\times\mathbb{T}^r \rightarrow (\widetilde{M}-T_o)/\mathbb{Z}^r$. So it is not enough to show that $\vec{h}(\widetilde{M})$ has nonempty interior; we need to show that the subset $\vec{h}(\sigma_{>0})$ \emph{has arbitrarily large balls}. Before we explain this, we should point that we also need to check that the homology class obtained is nontrivial.
\newline

\noindent
\textbf{Showing the $(k+r)$-homology class is nontrivial.} Note that having such a map $\mathbb{S}^k\times\mathbb{T}^r \rightarrow (\widetilde{M}-T_o)/\mathbb{Z}^r$ is not enough, we also need this map to be nontrivial in homology, which holds if there is a map
\[ (\widetilde{M} -T_o)/\mathbb{Z}^r \rightarrow \mathbb{S}^k\times\mathbb{T}^r \] 
such that the composition
\[\mathbb{S}^k\times\mathbb{T}^r \rightarrow (\widetilde{M} -T_o)/\mathbb{Z}^r \rightarrow \mathbb{S}^k\times\mathbb{T}^r\]
has non-zero degree. There is a natural candidate for $(\widetilde{M} -T_o)/\mathbb{Z}^r \rightarrow \mathbb{T}^r$, which is the composition
\[ (\widetilde{M} -T_o)/\mathbb{Z}^r \hookrightarrow \widetilde{M}/\mathbb{Z}^r \simeq\mathbb{T}^r.\]
In fact, this is a good candidate because the restriction $\mathbb{T}^r\rightarrow\mathbb{T}^r$ is a homotopy equivalence. A natural candidate for $(\widetilde{M} -T_o)/\mathbb{Z}^r \rightarrow \mathbb{S}^k$ is to take $\vec{h} \colon (\widetilde{M} -T_o) \rightarrow (\R^{k+1}-o) \simeq \mathbb{S}^k$ and then quotient out by the action of $\mathbb{Z}^r$ using the fact that $\vec{h}$ is $\mathbb{Z}^r$-invariant. This is also a good candidate, but we will not comment on this. 
\newline

In summary, what we are left to explain is how to show that the set of horospherical coordinates $\vec{h}(\sigma_{>0})$ of the Busemann cone $\sigma_{>0}$ has arbitrarily large balls. This is where we need $\sigma$ to be non-degenerate.

\bigskip
\noindent
\textbf{How to show that $\vec{h}(\sigma_{>0})$ has arbitrarily large balls.} First, let us comment on why $\vec{h}(\sigma_{>0})$ has nonempty interior as a subset of $\mathbb{R}^{k+1}$. The reason is because $\vec{h}$ maps non-degenerate points in $\sigma_{>0}$ to interior points of the image $\vec{h}(\sigma_{>0})$. This is an Invariance-of-Domain argument. Recall that one obtains the Busemann cone $\sigma_{>0}$ by mapping in 
\[\sigma_{>0} \colon \mathbb R^+\times\Delta \rightarrow\widetilde{M}\]
 as explained in Section \ref{nondegeneracysection}. Since $\sigma$ is non-degenerate, the approximation $\sigma_{R_i}$ has non-degenerate points if $i$ is large enough, and since being non-degenerate is an open condition, this implies that the set $(\mathbb R^+\times\Delta)^{nd}$ of non-degenerate $(R,t)$-coordinates of $\sigma_{>0}$ is open in $\mathbb R^+\times\Delta$ and therefore is a $(k+1)$-dimensional manifold. Note that $\vec{h}\circ\sigma_{>0}$ maps $(\mathbb{R}^+\times\Delta)^{nd}$ into $\mathbb{R}^{k+1}$, so if one can show that it is injective on $(\mathbb{R}^+\times\Delta)^{nd}$, then one can use Invariance of Domain to show it is an open map and obtain that  $\vec{h}(\sigma_{>0})$ has nonempty interior. To see that $\vec{h}\circ\sigma_{>0}$ is injective on $(\mathbb{R}^+\times\Delta)^{nd}$, we need the maps $\vec{h}$ and $\sigma_{>0}$ to be injective (when restricted to the relevant domains). The map $\sigma_{>0}$ restricted to $(\mathbb{R}^+\times\Delta)^{nd}$ is injective by Corollary \ref{injective}. The restriction of $\vec{h}$ to the Busemann cone $\sigma_{>0}$  has an inverse $p( \cdot , x_0 )$ since\footnote{See Corollary \ref{inversecor}.}  $\vec{h}(p(x_0,b)) = b$ for all $b \in \sigma_{>0}$, and therefore is injective.  
\newline

To see that $\vec{h}(\sigma_{>0})$ has arbitrarily large balls, we suppose for contradiction that it does not, which implies that $\vec{h}(\sigma_{>0})$ has an $L$-net $Q$ of boundary points\footnote{in the sense of point set topology} of $\vec{h}(\sigma_{>0})$. Since $\vec{h}$ maps non-degenerate points in $\sigma_{>0}$ to interior points of the image $\vec{h}(\sigma_{>0})$, the points in $Q$ either belong to $\vec{h}(\sigma_{>0}(\mathbb{R}^+\times \partial \Delta))$ or are not in the image $\vec{h}(\sigma_{>0})$. The latter is ridiculous since $\vec{h}$ restricted to $\sigma_{>0}$ diverges as $R\ra\infty$; the skeptical readers should keep reading because this will be clear after the next paragraph.  
\newline

Take $p( Q, x_0)$ to get a subset of $\sigma_{>0}$ and if this gives an $L$-net in $\sigma_{>0}$, then this will imply that all points of $\sigma$ are degenerate, which contradicts the assumption that $\sigma$ is non-degenerate. However,  $p( Q, x_0)$ needs not give a net in $\sigma_{>0}$ because $p(\cdot, x_0)$ might stretch distance between points in a nonuniform way. Nevertheless, a quantitative estimate on how $p(\cdot, x_0)$ distorts distance, as in the proof of Lemma \ref{root}, implies that $p(Q,x_0)$ is a $L\sqrt{R}$-net, by which we mean any point in $\sigma_{>0}$ that is a distance $R$ from $x_0$ is $L\sqrt{R}$-close to a point in $Q$. Since $\sqrt{R}$ is sublinear, this implies all points in $\sigma$ are degenerate. 
\newline

There is, however, one subtle point we need to be careful with. The Busemann simplex $\sigma$ is constructed as a limit of $\sigma_{R_i}$ for a particular sequence $\{R_i\}$. So if $\sigma (t)$ is a non-degenerate point, then $\sigma_{R_i} (t)$ is a non-degenerate point for $i$ large enough. As we saw above, $\sigma_{R_i}(t)$ is $L\sqrt{R_i}$-close to a point $q_i\in Q$. The concern is that $q_i$ might not be on $\sigma_{R_i}$ but equal to $\sigma_{R'_i}(t')$ for some other $R'_i$ and some $t'\in \partial\Delta$. But this is not a problem by the sequential criterion for degeneracy given in Lemma \ref{sublinear}. 

\subsection{The setup} Let us now begin setting up the proof.
In the course of the argument, we will need to project various points to various (intersections of) horoballs. To keep track of all this, let $p(b,x)$ be {the {\it closest point projection of the point $x$ to the intersection of horoballs} $\{\vec h\leq b\}$. This defines a map 
$$
p:\mathbb R^{k+1}\times\widetilde M\ra\widetilde M.
$$ 
For a fixed $b$ the map $p(b,\cdot)$ is a contraction, since it is the closest point projection to a convex set. It is $\mathbb Z^r$-equivariant since $\mathbb Z^r$ preserves the Busemann functions $h_i$ and thus also preserves the intersection of horoballs $\{\vec h\leq b\}$. Putting the Busemann functions together gives the {\it horospherical coordinates} map $\vec h=(h_0,\dots,h_k):\widetilde M\ra\mathbb R^{k+1}$. It is $\mathbb Z^r$-invariant because its coordinates are. It is a contraction in the sup norm $|\cdot|_{\infty}$ on $\mathbb R^{k+1}$ because $|\nabla h_i|=1$. 

\noindent
\textbf{The ``error".} A central role in the proof is played by the difference 
$$
\vec h(p(b,\cdot))-b,
$$ 
which measures the extent to which the closest point projection to the intersection of horoballs $\{\vec{h}\leq b\}$ fails to land in the intersection of horospheres $\{\vec{h}=b\}$. We will call it the ``error''. 
\begin{lemma}
$\vec{h}(p(b,\cdot))-b$ is a $\mathbb Z^r$-invariant contraction in the sup norm.
\end{lemma} 
\begin{proof}
This follows directly from what we have said about $p(b,\cdot)$ and $\vec h$. 
\end{proof}
\subsection{Some key properties of the Busemann cone}
Now let us turn to the Busemann cone $\sigma_{\geq 0}$. 
The key property of the Busemann cone is that the horospherical coordinates embed it in $\mathbb R^{k+1}$.
\begin{lemma}
\label{horocoordinates}
For a Busemann simplex $\sigma$, the restriction
$$
\vec{h}\mid_{\sigma_{\geq 0}}:\sigma_{\geq 0}\ra\mathbb R^{k+1}
$$
is a homeomorphism onto its image, with inverse $p(\cdot,x_0)$.
\end{lemma}
\begin{proof}
Let $b:=\vec{h}(\sigma_R(t))$ be the image of a point on the Busemann cone. We need to show that 
$$
p(b,x_0)=\sigma_R(t).
$$

\begin{figure}
\centering
\includegraphics[scale=0.6]{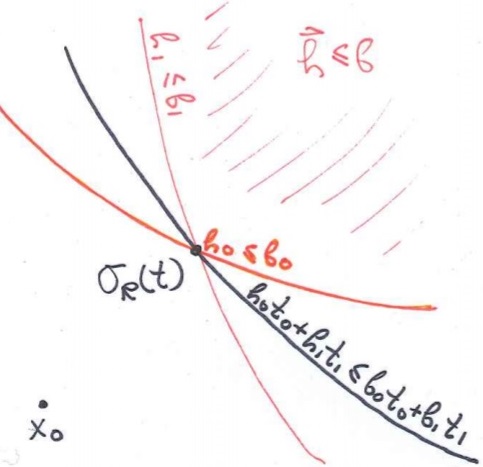}
\end{figure}

The intersection of horoballs $\cap_{i}\{h_i\leq b_i\}$ contains $\sigma_R(t)$ and is contained in the sublevel set $\{h_0t_0+\dots+h_kt_k\leq b_0t_0+\dots+b_kt_k\}$. Since $\sigma_R(t)$ is the unique closest point to $x_0$ in this sublevel set, it is also the unique closest point to $x_0$ on the intersection of horoballs.
\end{proof}
Denote the image of the Busemann cone in horospherical coordinates by 
$$
W:=\vec{h}(\sigma_{\geq 0}).
$$ 
Lemma \ref{horocoordinates} has several important consequences. The first is an estimate for the error in terms of the distance to the orbit of the basepoint $\mathbb Z^rx_0$. To get it, note first that the lemma immediately implies the error vanishes at $x_0$.
\begin{corollary}\label{inversecor}
For points $b\in W$ we have
$
\vec{h}(p(b,x_0))=b.
$
\end{corollary}
It follows that the error is bounded by the distance to the orbit $\mathbb Z^rx_0$. 
\begin{proposition}
\label{errorlemma}
For any $b\in W$ we have $|\vec{h}(p(b,x))-b|_{\infty}\leq d(x,\mathbb Z^rx_0)$.
\end{proposition}
\begin{proof}
We've shown that the error is a $\mathbb Z^r$-invariant contraction that vanishes at $x_0$, so this is clear.
\end{proof}
The second is a topological regular value theorem for Busemann simplices.
\begin{lemma}
\label{interior}
If $x\in\sigma^{nd}_{>0}$ is a non-degenerate point on the Busemann cone, then $\vec{h}(x)$ is an interior point of $W$. 
\end{lemma}
\begin{proof}
By Corollary \ref{injective} and Lemma \ref{horocoordinates} the composition 
\begin{eqnarray}
(\mathbb R^+\times\Delta)^{nd}_{\sigma}&\hookrightarrow\:\sigma_{>0}^{nd}\:\:\hookrightarrow&W\hspace{0.5cm}\subset\;\mathbb R^{k+1},\\
(R,t)&\mapsto\sigma_R(t)\mapsto&\vec{h}(\sigma_R(t))
\end{eqnarray}
is injective. Since $(\mathbb R^+\times\Delta)^{nd}_{\sigma}$ is an open subset of $\mathbb R^{k+1}$, Invariance of Domain implies that the image of this composition is also an open subset of $\mathbb R^{k+1}$ and thus is contained in the interior of $W$. 
\end{proof}
The third says that boundary points of $W$ come from degenerate points.
\begin{corollary}
\label{boundarypoints}
For every $b\in\partial W$ there is degenerate $q\in\sigma_{\geq 0}$ with $\vec h(q)=b$. 
\end{corollary}
\begin{proof}
Since $\sigma_{\geq 0}$ is closed, Lemma \ref{horocoordinates} implies that $W$ is closed\footnote{If $b^i\in W$ converges to $b$ then $p(b^i,x_0)\in\sigma_{\geq 0}$ converges to $q\in\sigma_{\geq 0}$, so $b^i\ra\vec{h}(q)\in W$.} and therefore $b\in W$. Thus, there is $q\in\sigma_{\geq 0}$ with $\vec{h}(q)=b$. Since $b$ is a boundary point, Lemma \ref{interior} implies that $q$ is a degenerate point. 
\end{proof}

\subsection{Finding large balls in $W$\label{largeballssubsection}}
In Lemma \ref{interior} we showed that if $\sigma$ is non-degenerate then $W$ contains open balls. In this subsection we will show that when $\sigma$ is non-degenerate then $W$ contains {\it arbitrarily large} balls. To do this, we will need to control how much the map $p(\cdot,x)$ distorts things. 
\begin{lemma}
\label{root}
Suppose that $a,b\in\mathbb R^{k+1}$ with $b\leq a$,\footnote{The notation $b\leq a$ means $b_i\leq a_i$ for all $0\leq i\leq k$.} and $x\in\widetilde M$. Then 
$$
d(p(a,x),p(b,x))\leq\sqrt{2d(x,p(a,x))|a-b|_1+|a-b|_1^2}.
$$
\end{lemma}
\begin{proof}
We use the short hand notation $x_a:=p(a,x)$. In other words, $x_a$ is the closest point to $x$ on the intersection of horoballs $\{\vec{h}\leq a\}$. In this notation, what we need to prove is
\begin{equation}
\label{squareroot}
d(x_a,x_b)\leq\sqrt{2d(x,x_a)|a-b|_1+|a-b|_1^2}.
\end{equation}
First, note that since $b\leq a$ the point $x_b$ lies in the intersection of horoballs $\{\vec{h}\leq a\}$ and therefore we have
$$
\angle_{x_a}(x,x_b)\geq\pi/2.
$$
Therefore 
\begin{eqnarray*}
d(x,x_a)^2+d(x_a,x_b)^2&\leq& d(x,x_b)^2\\
&\leq& d(x,x_{ab})^2\\
&\leq&(d(x,x_a)+d(x_a,x_{ab}))^2,
\end{eqnarray*}
where the first inequality follows by triangle comparison with an appropriate obtuse Euclidean triangle, the second inequality is because $x_b$ is the closest point to $x$ in the intersection of horoballs $\{\vec{h}\leq b\}$ while the projection $x_{ab}$ of $x_a$ to $\{\vec{h}\leq b\}$ is also in this intersection, and the third is the triangle inequality.

\begin{figure}
\centering
\includegraphics[scale=0.7]{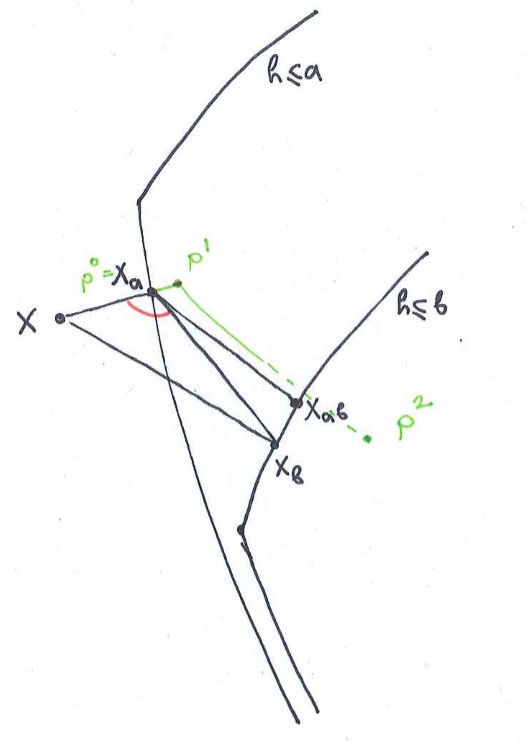}
\end{figure}

In order to get (\ref{squareroot}) from this, it is enough to prove 
\begin{equation}
\label{monotone}
d(x_a,x_{ab})\leq|a-b|_1.
\end{equation}
\underline{{\it Proof of (\ref{monotone}):}} Starting at $x_a=:p^0$ go for a time $|a_0-b_0|$ along $-\nabla h_0$ to arrive at a point which we call $p^1$, then go for time $|a_1-b_1|$ along $-\nabla h_1$ to arrive at a point denoted $p^2$, and continue this way, going from $p^i$ for time $|a_i-b_i|$ along $-\nabla h_i$ to arrive at $p^{i+1}$. We obtain a piecewise geodesic path 
$$
[p^0,p^1]\cup\dots\cup[p^k,p^{k+1}]
$$ 
of length $|a-b|_1$ starting at $x_a$. Along this path all the Busemann functions $h_i$ are monotone decreasing because $\angle(-\nabla h_i,-\nabla h_j)\leq\pi/2$. Moverover, each $h_i$ decreases by $|a_i-b_i|$ on the segment $[p^i,p^{i+1}]$ and therefore at the endpoint $p^{k+1}$ we get $h_i(p^{k+1})\leq b_i$ for all $i$. Thus, the endpoint $p^{k+1}$ is contained in the intersection of horoballs $\{h_0\leq b_0,\dots, h_k\leq b_k\}$. Since $x_{ab}$ is the closest point to $x_a$ in this intersection of horoballs, we get
\begin{eqnarray*}
d(x_a,x_{ab})&\leq& d(x_a,p^{k+1})\\
&\leq&|a-b|_1,
\end{eqnarray*}
which is what we wanted to show.
\end{proof}
Now we are ready to prove the main result of this subsection.
\begin{proposition}
\label{largeballs}
If $\sigma:\Delta\ra\partial_{\infty}$ is a non-degenerate Busemann simplex, then $W$ contains arbitrarily large $(k+1)$-balls.
\end{proposition}
\begin{proof}
Fix $L>0$. Let
$$_a\lowerlefttriangle:= \{b\in\mathbb R^{k+1}\mid b\leq a,|b-a|_1\leq L\}.
$$
We will show that $W$ contains a set of the form $_a\lowerlefttriangle$, which will imply that $W$ contains a ball of radius $L/(2k+2)$ because  $_a\lowerlefttriangle$ is a right angled triangle/simplex in $\mathbb{R}^{k+1}$ with side length $L$. Since $L$ is arbitrary, this will imply $W$ contains arbitrarily large balls.

Suppose that $W$ does not contain any set of the form $_a\lowerlefttriangle$. 
From the definition of Busemann simplices, we have a sequence $R_{i}\ra\infty$ so that the maps $\sigma_{R_i}$ converge to $\sigma$. In particular $\sigma_{R_i}(t)\ra\sigma(t)$. Let $a^i:=\vec{h}(\sigma_{R_i}(t))$.\footnote{We are indexing the sequences $\{a^i\}$ and $\{b^i\}$ with superscripts instead of subscripts to avoid possible confusion with the coordinates $a_i$ and $b_i$ of $a$ and $b$.} Since $_{a^i}\lowerlefttriangle$ is not contained in $W$, there is a point $b^i\in\partial W\cap$ $_{a^i}\lowerlefttriangle$. 
By Corollary \ref{boundarypoints}, we get a sequence of degenerate points $q_i\in\sigma_{\geq 0}$ on the Busemann cone for which $\vec{h}(q_i)=b^i$.
We claim the distance $d(q_i,\sigma_{R_i}(t))$ grows sublinearly as a function of $R_i$. In fact, since $q_i=p(b^i,x_0)$ and $\sigma_{R_i}(t)=p(a^i,x_0)$, by Lemma \ref{root}\footnote{By letting $a=a^i,b=b^i,x=x_0$ in Lemma \ref{root} and noting that $|a^i-b^i|_1\leq L$ and $d(x_0,p(a^i,x_0))=R_i$.}, we get
$$
d(q_i,\sigma_{R_i}(t))\leq\sqrt{2R_iL+L^2}.
$$
Thus, by Lemma \ref{sublinear}, $\sigma(t)$ is a degenerate point. Since this is true for every $t\in\Delta$, we see that $\sigma$ is a degenerate simplex. This contradicts the assumption that $\sigma$ is non-degenerate. Thus $W$ contains a set of the form $_a\lowerlefttriangle$. 
\end{proof}

\subsection{\label{productcycle}Proof of Theorem \ref{dimensionbound}} We are now ready for the proof of the dimension bound. Given a non-degenerate Busemann $k$-simplex $\sigma$, we will use the projection family $p(\cdot,\cdot)$ to build a $(k+r)$-dimensional homology cycle $S^k\times\mathbb T^r\ra(\widetilde M-\vec{h}^{-1}(s))/\mathbb Z^r$ in the subset obtained by cutting out an appropriately chosen intersection of horospheres, and then show that defines a non-zero homology class in the manifold $(\widetilde M-\vec{h}^{-1}(s))/\mathbb Z^r$. This will imply that the dimension of the manifold is greater than $k+r$ and will finish the proof of Theorem \ref{dimensionbound}. The proof uses Propositions \ref{errorlemma} and \ref{largeballs} that we established in the previous subsections.    
\subsubsection{Building the homology cycle\label{construction}}
Since $\widetilde M$ is contractible, the quotient $\widetilde M/\mathbb Z^r$ is homotopy equivalent to an $r$-torus. Such a homotopy equivalence $g:\mathbb T^r\ra\widetilde M/\mathbb Z^r$ is covered by a $\mathbb Z^r$-equivariant map $\tilde g:\mathbb R^r\ra\widetilde M$. Form the Busemann cone $\sigma_{\geq 0}$ based at $\tilde g(0)$ and let $W=\vec h(\sigma_{\geq 0})$. Proposition \ref{errorlemma} tells us that for all points $b\in W$ we have the uniform error bound
\begin{eqnarray}
\label{errorbound}
|\vec h(p(b,\widetilde g(\cdot)))-b|_{\infty}
&\leq&\mbox{ diameter of } \widetilde g([0,1]^r).
\end{eqnarray}
Since $\sigma$ is a non-degenerate Busemann $k$-simplex, the image of its Busemann cone in horospherical coordinates $W$ contains arbitrarily large balls by Proposition \ref{largeballs}. So, we can find a closed ball $\overline B_s(L)$ in $W$ whose radius $L$ (in the sup metric) is greater than the diameter of $\widetilde g([0,1]^r)$. Let $S^k_s:=S_s(L)$ be its boundary sphere. Then the map
\begin{eqnarray}
\label{homologycycle}
S^k_s\times\mathbb R^r&\ra&\widetilde M,\\
(b,z)&\mapsto&p(b,\tilde g(z)).
\end{eqnarray}
will miss the intersection of horospheres $\{\vec h=s\}$ because for every $b\in S^k_s$ the distance $|s-b|_{\infty}=L$ while the distance $|\vec{h}(p(b,\widetilde g(\cdot)))-b|_{\infty}$ is strictly less than $L$. Thus, quotienting out by the $\mathbb Z^r$-action we get a homology cycle
\begin{equation}
\label{homologycyclemap}
S_s^k\times\mathbb T^r\ra(\widetilde M-\vec h^{-1}(s))/\mathbb Z^r.
\end{equation}

\subsubsection{\label{nonzeross}The homology cycle is non-trivial in the homology of $(\widetilde M\setminus \vec{h}^{-1}(s))/\mathbb Z^r$}
Look at the composition 
\begin{equation}
\begin{array}{ccccc}
\label{compo}
S_s^k\times\mathbb R^r&\ra&\widetilde M-\vec h^{-1}(s)&\ra&(\mathbb R^{k+1}\setminus\{s\})\times\widetilde M,\\
(b,z)&\mapsto& p(b,\widetilde g(z))&\mapsto&(\vec h(p(b,\widetilde g(z))),p(b,\widetilde g(z))).
\end{array}
\end{equation}
We can homotope the second factor of this composition via the straight line homotopy in $\widetilde M$ to the map $(b,z)\mapsto\widetilde g(z)$. We can also homotope the first factor of this composition via the (Euclidean) straight line homotopy to the map $(b,z)\mapsto b$. As we travel along this straight line, the sup distance to $b$ is monotone decreasing, so the inequality (\ref{errorbound}) is preserved, and therefore the straight line avoids $s\in\mathbb R^{k+1}$. Both of these homotopies are $\mathbb Z^r$-equivariant. In summary, we can $\mathbb Z^r$-equivariantly homotope the map (\ref{compo}) to $(b,z)\mapsto(b,\widetilde g(z))$. This covers the homotopy equivalence $id\times g$ of the $\mathbb Z^r$-quotients $S_s^k\times\mathbb T^r\ra(\mathbb R^{k+1}\setminus\{s\})\times\widetilde M/\mathbb Z^r$ , so the original map (\ref{compo}) also defines a homotopy equivalence of the $\mathbb Z^r$-quotients. But, this original map factors through the map (\ref{homologycyclemap}) representing the homology cycle, so that cycle is non-zero in homology.

\subsubsection{Finishing the proof of Theorem \ref{dimensionbound}} 
Since $(\widetilde M-\vec h^{-1}(s))/\mathbb Z^r$ is a non-compact $n$-manifold, it deformation retracts to an $(n-1)$-complex. We showed that this $(n-1)$-complex has non-zero $(r+k)$-homology which implies that $r+k\leq n-1$. This finishes the proof of Theorem \ref{dimensionbound}.

\begin{remark}
The logic of the argument in \ref{nonzeross} is that once we form the composition (\ref{compo}) then we can ``straighten it out'' by homotoping the two factors independently. Of course, the resulting map no longer factors through the manifold $(\widetilde M-\vec h^{-1}(s))/\mathbb Z^r$, but this is not important because all we are trying to verify is that the composition gives a homotopy equivalence. Having done this, we conclude that the original map which {\it did} factor through $(\widetilde M-\vec h^{-1}(s))/\mathbb Z^r$ is also a homotopy equivance, and this gives the non-zero $(k+r)$-dimensional homology class. 
\end{remark}

\subsection{Proof of Theorem \ref{gdimensionboundintro} (replacing $\mathbb Z^r$ by a general group $G$)\label{othergroups}}
In the proof of the dimension bound (Theorem \ref{dimensionbound}), the abelian nature of the group $\mathbb Z^r$ only played a role in subsection \ref{productcycle}, an even there the specifics of the group $\mathbb Z^r$ played only a minor role. In this subsection we will phrase the argument given in \ref{productcycle} in a way that applies to any subgroup $G<\Gamma$. The topological conclusion is that we get the same result as before for any group $G$ if we replace the rank of the abelian group $\mathbb Z^r$ by the homological dimension of the discrete group $G$. This is expressed in Corollary \ref{gdimbound} below, which is a reformulation of Theorem \ref{gdimensionboundintro} from the introduction.\footnote{We stated it here in a slightly different, but equivalent, form.} The main technical step is the following proposition which combines \ref{construction} and \ref{nonzeross}.
\begin{proposition}
\label{factorprop}
Suppose that $\sigma:\Delta^k\ra\Fix^0(G)$ is non-degenerate Busemann $k$-simplex. Then the inclusion of any compact subset of $S^k\times\widetilde M/G$ factors up to homotopy through $(\widetilde M-\vec{h}^{-1}(s))/G$ for an appropriately chosen intersection of horospheres $\vec{h}^{-1}(s)$. 
\end{proposition} 
\begin{proof}
It is clearly enough to prove the proposition for subsets of the form $S^k\times K$, where $K$ is compact. 
Let $K\hookrightarrow\widetilde M/G$ be the inclusion of a subset of diameter $<L$. It is covered by a $G$-equivariant inclusion $\widetilde K\hookrightarrow\widetilde M$. Pick a point $x_0\in\widetilde K$ and form the Busemann cone $\sigma_{\geq 0}$ based at this point. Since $\sigma$ is a non-degenerate $k$-simplex, the image of the Busemann cone in horospherical coordinates contains a closed ball $\overline B_s(L)\subset W$ of radius $L$ (in the sup metric) centered at some point $s\in W$. Let $S_s^k:=S_s(L)$ be its boundary sphere and denote by 
$$
i:S_s^k\times\widetilde K\hookrightarrow \mathbb R^{k+1}\times\widetilde M
$$ 
the standard inclusion. Let
$$
f:S_s^k\times\widetilde K\ra \mathbb R^{k+1}\times\widetilde M
$$ 
be the map given by  
$$
f(b,z):=(\vec{h}(p(b,z)),p(b,z)),
$$
and denote by $f_t$ the straight line homotopy\footnote{Along the Euclidean straight lines $(1-t)b+t\vec{h}(p(b,z))$ in the first factor and geodesic lines from $z$ to $p(b,z)$ in the second factor.} starting at $i$ and ending at $f$. 
Since points $b\in S_s^k\subset W$ are in the image of the Busemann cone, we have 
$$
\vec{h}(p(b,x_0))=b \mbox{ for all } b\in S_s^k.
$$ 
$K$ has diameter $<L$ so every point $z$ in $\widetilde K$ is a distance $<L$ from a point of the orbit $Gx_0$. Since $\vec{h}(p(b,\cdot))$ is a $G$-invariant contraction we get
\begin{eqnarray*}
|\vec{h}(p(b,z))-b|_{\infty}&=&|\vec{h}(p(b,z))-\vec{h}(p(b,Gx_0))|_{\infty}\\
&\leq&d(z,Gx_0)\\
&<&L
\end{eqnarray*} 
for all points $(b,z)\in S_s(L)\times\widetilde K$. Therefore the straight line homotopy $f_t$ between $f$ and $i$ avoids $s\times\widetilde M$, i.e. it stays in $(\mathbb R^{k-1}\setminus s)\times\widetilde M$. Finally, note that the map $f=f_1$ factors as
$$
S_s^k\times\widetilde K\stackrel{p(\cdot,\cdot)}\longrightarrow\widetilde M-\vec{h}^{-1}(s)\stackrel{\vec{h}\times id}\longrightarrow(\mathbb R^{k+1}\setminus s)\times\widetilde M.
$$
Quotienting out by the $G$-action, we see that $S_s^k\times K\hookrightarrow (\mathbb R^{k+1}\setminus s)\times\widetilde M/G$ can be homotoped to factor through the proper open subset $(\widetilde M-\vec{h}^{-1}(s))/G$ of $\widetilde M/G$. 
\end{proof}
\begin{corollary}[Theorem \ref{gdimensionboundintro}]
\label{gdimbound}
Suppose $\sigma:\Delta^k\ra\Fix^0(G)$ is a non-degenerate Busemann $k$-simplex. Then
$$hdim(G)+k\leq n-1.$$
\end{corollary}
\begin{proof}
Let $\Sigma$ be a non-zero $r$-dimensional homology class (with any system of coefficients) in $\widetilde M/G$. Then $S^k\times\Sigma$ is a non-trivial $(r+k)$-dimensional homology class in $S^k\times\widetilde M/G$. It is supported on a compact subset $S^k\times K$, so the proposition implies that it defines a non-trivial homology class on the open manifold $(\widetilde M\setminus\vec{h}^{-1}(s))/G$. Since this manifold is homotopy equivalent to an $(n-1)$-complex we conclude that $r+k\leq n-1$.
\end{proof}

\section{Finishing the proofs of Theorems \ref{analog}, \ref{factor}, and \ref{other}}
Our main application of Theorem \ref{dimensionbound} is the following corollary. It is the last ingredient we need in order to finish the proofs of Theorems \ref{analog},\ref{factor}, and \ref{other}.
\begin{corollary} 
\label{busemannmapbound}
Let $\beta:\Delta_{\lfloor pAb\rfloor}\ra(\partial_{\infty},\Td)$ be the Busemann map. Then
$$
\dim(\im\beta)\leq\lfloor n/2\rfloor-1.
$$
\end{corollary}

\begin{proof}
Let $\sigma=([A_0]<\dots<[A_k])$ be a $k$-simplex in $\Delta_{\lfloor pAb\rfloor}$ and note that, because we are using virtual equivalence classes, 
$$
\rank(A_k)\geq k+1.
$$ 
Note that the top abelian group $A_k$ preserves horospheres on the entire Busemann simplex $\beta(\sigma)$. So, if $\beta(\sigma)$ is a non-degenerate, Theorem \ref{dimensionbound} gives
$$
n\geq k+1+\rank(A_k).
$$
The right hand side is $\geq 2(k+1)$ so we conclude
$$
k\leq\lfloor n/2\rfloor -1. 
$$
Since $\beta$ is Lipschitz, Lipschitz maps do not raise Hausdorff dimension, and topological dimension is less than or equal to Hausdorff dimension, we get for any non-degenerate Busemann simplex $\beta(\sigma)$ that
$$
\dim(\beta(\sigma))\leq\mbox{Hausdorff-}\dim(\beta(\sigma))\leq\dim(\sigma)\leq\lfloor n/2\rfloor -1.
$$
The image $\im(\beta)$ is a countable union of non-degenerate Busemann simplices, so its topological dimension is also bounded by this.
\end{proof}
Now, let $\beta$ be the Busemann map. 
The above corollary implies that
$$
d=\dim(\im(\beta\circ\mu))\leq\dim(\im\beta)\leq\lfloor n/2\rfloor-1. 
$$
So, Theorem \ref{collapsetheorem} applied to the Busemann map $\beta$ proves Theorem \ref{factor}. We saw in section \ref{importanceofbeinglipschitz} and subsection \ref{goodbeta} that $d\leq\rank_{Ab}(\pi_1M)-1$ and $d\leq\dim(\partial_{\infty},\Td)$ so it also proves Theorem \ref{other}. Inspecting Theorem \ref{collapsetheorem} we see that Theorem \ref{analog} a) follows from the two bullets in Theorem \ref{collapsetheorem} (in \ref{tinydelta} we can pick $\delta<\epsilon/2$) and that Theorem \ref{analog} b) follows from $d\leq\lfloor n/2\rfloor -1$.

\section{Low dimensional collapse: $\dim(\im\rho)\leq 1$}
In this section we will explain the more specific consequences we get when our analogue of the rational Tits building has one dimensional image, i.e. when $d:=\dim(\im\rho)\leq 1$. This happens if one of the following is true.  
\begin{itemize}
\item
The fundamental group $\pi_1M$ does not contain $\mathbb Z^3$, or
\item
$\dim(\partial_{\infty},\Td)\leq 1$, or
\item
$\dim M\leq 5$.
\end{itemize}  
We have seen that if $d\leq 1$ then each component of $M_{\leq\epsilon}$ is aspherical. Let $\partial$ be a component of $\partial M_{\leq\epsilon}$. Next we ask ``how bad can the fundamental group of $\partial$ be?'' Notice it is an extension of a subgroup of $\pi_1M$ by the fundamental group of a component $\hat\partial$ of $\partial\widetilde M_{\leq\epsilon}$, i.e. it fits into an exact sequence
\begin{equation}
1\ra\pi_1(\hat\partial)\ra\pi_1\partial\ra\pi_1M.
\end{equation}
The main theorem of this section is that if $d\leq 1$ then {\it every finitely generated subgroup of $\pi_1(\hat\partial)$ is free}. Such a group is called a {\it locally free} group.  
\begin{theorem}
\label{locallyfree}
If $d\leq 1$ then $\pi_1(\hat\partial)$ is locally free.
\end{theorem} 
The moral of the theorem is that the fundamental group of the aspherical manifold $\partial$ is built out of two ingredients that we understand, namely locally free groups and fundamental groups of nonpositively curved manifolds, so it is ``not too bad''. We will give a concrete application of this philosophy at the end of this section.
\subsection{Proof of Theorem \ref{locallyfree} (Theorem \ref{locallyfreeintro} of the introduction)}
Let $G<\pi_1\widetilde M_{\leq \epsilon}$ be a finitely generated subgroup. Since $\widetilde M_{\leq\epsilon}$ is aspherical, the inclusion of groups is induced by a map 
$$
BG^{(2)}\stackrel{\varphi}\ra\widetilde M_{\leq\epsilon}.
$$ 
First, we will explain the proof in the simpler case when $G$ is {\it finitely presented}. In this case, the complex $BG^{(2)}$ is finite, so $\varphi$ can be homotoped in $\widetilde M_{\leq\epsilon}$ to a map $\overline{\varphi}$ whose image $\overline{\varphi}(BG^{(2)})$ is contained in a graph. Thus the inclusion of $G<\pi_1(\hat\partial)$ factors through a free group, implying $G$ is free.

In general, we only know that $G$ is finitely generated, i.e. that $BG^{(1)}$ is finite. So, if we try to homotope the map $\varphi$ to factor through a graph then the finitely many generators $BG^{(1)}$ do not cause any problems but the infinitely many relations $BG^{(2)}$ might. They may ``stick out'' into the thick part in the course of the homotopy. The key idea is that this does not cause a problem because the relations are one-dimensional at the end of the homotopy and we can simply attach them to the thin part as ``strands'' sticking out into the thick part without breaking the argument. 

Let us now make precise sense of this idea. Think of $\varphi$ as a map of pairs 
\begin{equation}
\label{mapofpairs}
(BG^{(2)},BG^{(1)})\stackrel{\varphi}\ra(\widetilde M,\widetilde M_{\leq\epsilon}).
\end{equation}
{\bf Claim:}
We claim that the map of pairs $\varphi$ can be homotoped inside $(\widetilde M,\widetilde M_{\leq\epsilon})$ to a map of pairs $\overline\varphi$ whose image is in the $1$-skeleton of a triangulation of $(\widetilde M,\widetilde M_{\leq\epsilon})$, i.e.
$$
(BG^{(2)},BG^{(1)})\stackrel{\overline{\varphi}}\ra(\widetilde M^{(1)},\widetilde M_{\leq\epsilon}^{(1)}).
$$ 

This is the main technical step in the proof of the theorem. It will follow from a relative version of the collapse argument given in the proof of Theorem \ref{collapsetheorem}. We will state it as a separate proposition and, since for this collapse argument it is not important that $d\leq 1$, we state it for general $d$. 
\begin{proposition}
Suppose $K$ is a compact subset of $\partial\widetilde M_{\leq\epsilon}$. Then the standard inclusion $i:(\partial\widetilde M_{\leq\epsilon},K)\hookrightarrow(\widetilde M,\widetilde M_{\leq\epsilon})$ is homotopic {\it as a map of pairs} in $(\widetilde M,\widetilde M_{\leq\epsilon})$ to a map 
$$
\overline f:(\partial\widetilde M_{\leq\epsilon},K)\ra(\widetilde M^{(d)},\widetilde M^{(d)}_{\leq\epsilon}
)
$$
whose image is contained in the $d$-skeleton of some triangulation of $(\widetilde M,\widetilde M_{\leq\epsilon})$.
\end{proposition}
\begin{proof}
In the proof of Theorem \ref{collapsetheorem} we homotoped the inclusion of a compact subset $K\hookrightarrow\partial\widetilde M_{\leq\epsilon}$ inside $\widetilde M_{\leq\epsilon}$ to a map with $d$-dimensional image. But, the proof really gave us a bit more. We constructed a homotopy between the inclusion $\partial\widetilde M_{\leq\epsilon}\hookrightarrow \widetilde M_{\leq\epsilon}$ and $c_t\circ\beta\circ\mu$ whose restriction to a compact subset $K$ had, for large enough $t$, the additional property that the set $K$ stayed inside $\widetilde M_{\leq\epsilon}$.
Since the image of $c_t\circ\beta\circ\mu$ is $d$-dimensional, one can deform it (not just its restriction to $K$, but the entire map $c_t\circ\beta\circ\mu$) by a $\delta$-small deformation into a $d$-skeleton of a sufficiently fine triangulation of $\widetilde M$.
This is done in two steps. First, one deforms the map by a $\delta$-small deformation to a map $f'$ whose image has Hausdorff dimension $\leq d$ (not just topological dimension $\leq d$). Second one deforms $f'$ into a $d$-skeleton of a sufficiently fine triangulation. The triangulation can be taken to be a triangulation of pairs $(\widetilde M,\widetilde M_{\leq\epsilon})$ and this second deformation can be done in a way that preserves simplices of the triangulation, i.e. a point inside a closed simplex stays inside that simplex (see Appendix D for more details on these two steps). 
\end{proof}
We apply this to our situation $d\leq 1$ with $K=\varphi(BG^{(1)})$ to get the desired homotopy from $\varphi=i\circ\varphi$ to $\overline\varphi:=\overline f\circ\varphi$. This proves the claim. 
Next, we observe that the maps 
\begin{equation}
\label{graph}
\overline{\varphi}:BG^{(2)}\ra\widetilde M^{(1)}\cup\widetilde M_{\leq\epsilon}
\end{equation} 
and 
\begin{equation}
\label{original}
\varphi:BG^{(2)}\ra\widetilde M^{(1)}\cup\widetilde M_{\leq\epsilon}
\end{equation} 
are homotopic. To see this, note first that the maps $\varphi:BG^{(1)}\ra\widetilde M_{\leq\epsilon}$ and $\overline\varphi:BG^{(1)}\ra\widetilde M_{\leq\epsilon}$ are already homotopic by construction. Since $\widetilde M^{(1)}\cup \widetilde M_{\leq\epsilon}$ is aspherical, we can extend this to a homotopy between (\ref{graph}) and (\ref{original}). 

Finally, since the map $\varphi$ in (\ref{original}) is $\pi_1$-injective\footnote{The map $\varphi:BG^{(2)}\ra\widetilde M_{\leq\epsilon}$ is $\pi_1$-injective by construction, and attaching the one dimensional ``strands'' $\widetilde M^{(1)}$ to the thin part $\widetilde M_{\leq\epsilon}$ does not change this.} while the image of $\overline{\varphi}$ is contained in a graph, we conclude that $G$ is a free group. This proves that $\pi_1(\hat\partial)$ is locally free.

\subsection{No uniformly positive scalar curvature} Theorem \ref{locallyfree} gives enough information about the fundamental group $\pi_1\partial$ to conclude, via the method of \cite{blockweinberger}, that the manifold $M$ does not admit a complete Riemannian metric of uniformly positive scalar curvature. Briefly, the paper \cite{blockweinberger} shows that to prove there is no such metric, it is enough to show that $\partial$ is aspherical and that its fundamental group $\pi_1\partial$ satisfies the strong Novikov conjecture. It also shows that to verify the last condition it is enough to know that
\begin{itemize}
\item $\pi_1M$ satisfies the strong Novikov conjecture ``with coefficients'', and
\item
$K:=\ker(\pi_1\partial\ra\pi_1M)$ satisfies the Baum-Connes conjecture.
\end{itemize}  
In our situation $M$ is a complete, Riemannian manifold of nonpositive curvature, so the first bullet is true by \cite{kasparov}.\footnote{``Coefficients'' are not explicitly mentioned in \cite{kasparov}, but \cite{baumconneshigson} say (p. 44) that \cite{kasparov} also proves the coefficients version. The version with coefficients is also stated on p. 8 of \cite{puschnigg}.} The second bullet is more delicate. But, in our situation we have shown that $K$ is a locally free group, so it satisfies the Baum-Connes conjecture by \cite{block}. Therefore, we get the following.
\begin{theorem}[Corollary \ref{nopsccor}]
If $d\leq 1$ then the manifold $M$ does not have a complete Riemannian metric that has uniformly positive scalar curvature.
\end{theorem}
\begin{remark}
Previously, the results of this section were known for locally symmetric manifolds of $\mathbb Q$-rank $\leq 2$  (see \cite{blockweinberger}) and in the case when $(\partial_{\infty},\Td)$ is discrete (this follows from \cite{gromovlawson} together with \cite{eberleinvisibility}). Moreover, the locally symmetric manifolds of $\mathbb Q$-rank $\geq 3$, e.g. the product of three punctured tori $\dot{\mathbb  T^2}\times \dot{\mathbb T^2}\times \dot{\mathbb  T^2}$, do admit complete metrics of uniformly positive scalar curvature (see \cite{blockweinberger}) so our results are sharp.   
\end{remark}

\begin{remark}
It would be interesting to find a proof of ``no upsc'' that uses more directly the metric properties of the homotopy $\rho_t$ and bypasses index theory and the Novikov conjecture for the end. 
\end{remark}

\section{\label{bgs} Appendix A: A useful comparison}
In this paper, we occasionally make estimates. Some of these (\ref{bpindep}, \ref{sequential}, \ref{largeballssubsection}, \ref{finseg}) depend on the following standard comparison. The only estimate that is not based on this is the Lipschitz estimate in \ref{lipschitzestimate}. It uses symmetry. 

\subsection{\label{closestpointprojection}Comparison with obtuse Euclidean triangle}
Let $C\subset\widetilde M$ be a convex set and $p_C:\widetilde M\ra C$ the closest point projection. Let $x_0\not\in C,y\in C$ and let $x=p_C(x_0)$ be the closest point to $x_0$ in $C$. Since $C$ is convex and $x$ is the closest point to $x_0$ in $C$, we observe (see Figure above) that 
\begin{equation}
\label{obtuse}
\measuredangle_x(x_0,y)\geq\pi/2.
\end{equation} 
Because of this, the Euclidean triangle with sides $[x_0,x]$ and $[x,y]$ meeting at an angle $\angle_x(x_0,y)$ has $d(x_0,y)^2\geq d(x_0,x)^2+d(x,y)^2$. Triangle comparison implies that the same is true for our triangle in $\widetilde M$. Consequently, 
\begin{equation}
\label{triangle}
d(x,y)\leq\sqrt{d(x_0,y)^2-d(x_0,x)^2}.
\end{equation}
The obtuseness (\ref{obtuse}) also implies that 
\begin{equation}
\label{acute}
\angle_{x_0}(x,y)<\pi/2.
\end{equation} 
Next is a typical example of how this triangle comparison is used.
\begin{figure}
\centering
\includegraphics[scale=0.6]{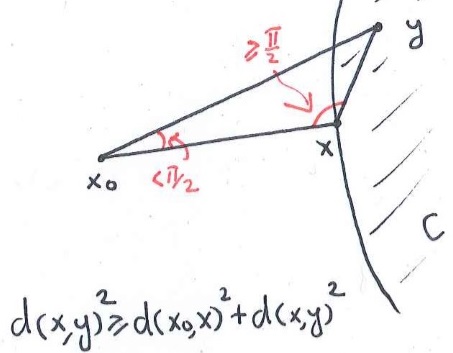}
\end{figure}
\subsection{Distance dependence of projections to sublevel sets\label{projes}}
Suppose $f:\widetilde M\ra\mathbb R$ is a convex function that does not attain its infimum. Pick a basepoint $x_0$. Then for every $R\geq 0$ there is a unique constant $t(R)$ so that the sublevel set $\{f\leq t(R)\}$ meets the sphere $S_{x_0}(R)$ at a single point $\lambda_R$. Comparing with Euclidean triangles, we see how $\lambda_R$ can vary with $R$. 
\begin{lemma}
\label{cpholder}
$d(\lambda_{R(1\pm\delta)},\lambda_R)\leq\sqrt{2\delta+\delta^2}R$.
\end{lemma}
\begin{proof}
The sublevel set $\{f\leq t(R)\}$ is convex, and $\lambda_{R}$ is the point on it that is closest to $x_0$.
Since $t(R+\varepsilon R)\leq t(R)$, the point $\lambda_{R(1+\delta)}$ is also in the sublevel set $\left\{f\leq t(R)\right\}$. So, we can apply inequality (\ref{triangle}) with $C=\{f\leq t(R)\}$, $x=\lambda_{R}$ and $y=\lambda_{R(1+\delta)}$ to get $d(\lambda_{R},\lambda_{R(1+\delta)})\leq\sqrt{2\delta+\delta^2}R$. A similar argument with $C=\{f\leq t(R-\delta R)\}, x=\lambda_{R(1-\delta)}$ and $y=\lambda_R$ shows that $d(\lambda_{R(1-\delta)},\lambda_{R})\leq\sqrt{2\delta-\delta^2}R$.
\end{proof}


\section{Appendix \"A: The metrics $\angle_x,\angle$ and $\Td$ on $\partial_{\infty}$}
In this appendix we collect some results about metrics on $\partial_{\infty}$. Everything except for \ref{convexsubsection} can be found in Section 4 and Appendix 3 of \cite{ballmangromovschroeder}. See also 3.5 of \cite{eberlein}.
\subsection{Angle metric}
For every point $x\in\widetilde M$ we have the round metric $(\partial_{\infty},\angle_x)$ which identifies the boundary at infinity with the unit sphere $T^1_x\widetilde M$ in the tangent space at $x$. Taking the supremum over all $x\in\widetilde M$ we get the angle metric
$$
\angle(\xi,\eta):=\sup_{x\in\widetilde M}\angle_x(\xi,\eta). 
$$ 
\subsection{Description via distance in $\widetilde M$}This metric has the following alternative description. Fix a basepoint $x_0$ in $\widetilde M$ and denote by $r_{\xi}=[x_0,\xi)$ the unit speed geodesic ray starting at $x_0$ and going to $\xi\in\partial_{\infty}$. 
Then
$$
\sin\left({\angle(\xi,\eta)\over 2}\right)=\lim_{t\ra\infty}{d(r_{\xi}(t),r_{\eta}(t))\over 2t}. 
$$
\subsection{Tits metric}
By construction, the space $(\partial_{\infty},\angle)$ has diameter $\pi$. So, in order to study large scale features of $\partial_{\infty}$ one takes the induced path metric. This is called the Tits metric and denoted $(\partial_{\infty},\Td)$. Then 
$$
\angle=\min(\Td,\pi),
$$ 
so for small scale purposes there is no difference between $\Td$ and $\angle$. 
\subsection{Relation between sphere and $\Td$ topology}The topology obtained from the metric $(\partial_{\infty},\Td)$ is generally very different from the usual sphere topology $(\partial_{\infty},\angle_x)$. But, we have the following relation, which is usually called {\it lower semicontinuity} of $\Td$ in the sphere topology: If $x_i\ra x$ in the sphere topology then 
$$
\Td(x,y)\leq \liminf_{i}\Td(x_i,y).
$$
Here is a typical application. Let $C$ be a subset of $\partial_{\infty}$ and $\overline C$ its closure in the sphere topology. Then 
$$
\Td\mbox{-diameter}(C)=\Td\mbox{-diameter}(\overline C).
$$
\subsection{Curvature bounded above}
A key feature of $\Td$ is that any points $x,y$ with $\Td(x,y)<\pi$ are connected by a unique $\Td$-geodesic. Moreover, $\Td$ is CAT($1$), which means that for three points $x,y,z$ mutually a distance $\leq\pi/2$ the geodesic triangle $\Delta_{xyz}$ is thinner than the corresponding comparison triangle with the same side lengths in the round sphere. 
\subsection{Sets of $\Td$-diameter $\leq\pi/2$ have canonical Centers}
If $K$ is a set in $\partial_{\infty}$ that has $\Td$-diameter $\leq\pi/2$ and is closed in the sphere topology, then there is a {\it unique} point $\xi\in\partial_{\infty}$ at which the function
$$
\rho(\cdot):=\sup_{\eta\in K}\Td(\eta,\cdot)
$$
attains its infimum. This point $\xi=\xi_K$ is called {\it the Center} of $K$.
Let us recall the proof of this fact.
\begin{itemize}
\item (Finding an infimum)
Take sequence $\xi_i$ with $\rho(\xi_i)\ra\inf\rho$. After taking a subsequence, we may assume this sequence converges in the sphere topology to a point $\xi_i\ra\xi'$. Lower semicontinuity of $\Td$ implies that $\rho(\xi')\leq\lim\rho(\xi_i)=\inf\rho$. So, the only issue is uniqueness. 
\item (Uniqueness if $\inf\rho<\pi/2$)
Since the set $K$ has diameter $\leq\pi/2$ we obviously have $\inf\rho\leq\pi/2$. A standard CAT($1$) comparison argument gives uniqueness if we have the {\it strict} inequality $\rho<\pi/2$: If $\xi$ and $\xi'$ are two different points with $\rho(\xi)=\rho(\xi')=\inf\rho<\pi/2$ then there is a unique geodesic $[\xi,\xi']$. The midpoint $\eta$ of this geodesic has\footnote{Explicitly, one gets via triangle comparison with the round sphere that 
$$
\cos(\rho(\eta))\geq\cos(\inf\rho)\cdot\cos\left({\Td(\xi,\xi')\over 2}\right)
$$ and therefore $\rho(\eta)<\inf\rho$ whenever $\xi\not=\xi'$.} $\rho(\eta)<\inf\rho$, which is a contradiction. \end{itemize}
So, one just needs to show that
$$
\inf\rho<\pi/2.
$$
This is done in two steps.
\begin{itemize}
\item (Centers on the round spheres $T_x^1\widetilde M$)
First, one shows that for every $x\in\widetilde M$ there is a unique point $\xi_x\in T^1_x\widetilde M=(\partial_{\infty},\angle_x)$ at which the function $\rho_x(\cdot):=\sup_{\eta\in K}\angle_x(\eta,\cdot)$ attains its infimum. Moreover, there is a positive constant $\alpha:=\alpha_n>0$ only depending on the dimension $n$ so that $\rho_x(\eta_x)\leq\pi/2-\alpha$.
\item (Flowing centers to infinity)
The points $\xi_x$ are a continuous vector field of unit vectors on $\widetilde M$. One takes an integral curve $g:[0,\infty)\ra\widetilde M$ of this vector field and checks that any accumulation point $\eta\in\overline{g([0,\infty))}\cap\partial_{\infty}$ of this integral curve on the boundary at infinity satisfies $\rho(\eta)\leq\pi/2-\alpha$.
\end{itemize}
\subsection{\label{convexsubsection}Convex sets contain their Centers}
In general, the center $\xi_K$ may not be in the set $K$. However, if $K$ is {\it convex} in the $\Td$-metric then $\xi_K\in K$. To see this, note that for any point $\xi'$ with $\rho(\xi')\leq\pi/2$ we can take a sequence of points $x_i\in K$ with 
\begin{eqnarray*}
\Td(x_i,\xi')&\ra&\inf_{y\in K}\Td(y,\xi')
\end{eqnarray*} 
Passing to a subsequence, we can assume $x_i$ converges in the sphere topology to some point $x\in K$. Then, by lower semicontinuity of $\Td$, the point $x$ is a closest point to $\xi'$ in the set $K$ and comparison with the round sphere gives $\rho(x)\leq\rho(\xi')$. Since the Center $\xi_K$ is the unique infimum of $\rho$, it must be contained in $K$. 
\begin{remark}
In Section \ref{goodpoints} we show that if $K=\overline C$ is the closure (in the sphere topology) of a $\Td$-convex set $C$, then we still get the conclusion $\xi_K\in K$. 
\end{remark}
\subsection{Centers of mass for fixed sets of parabolics}One application of this is to finding canonical centers of mass in fixed sets of parabolic elements. Suppose $\gamma\in\Gamma$ is a parabolic element, and let $\Fix(\gamma)$ in $\partial_\infty$ be its fix set. This set is non-empty, and one can find a canonical center of mass inside of it by the following process. 
\begin{itemize}
\item
First, find a point $\xi\in\Fix(\gamma)$ so that $\Td(\xi,\eta)\leq\pi/2$ for all $\eta\in\Fix(\gamma)$. Such a point can be obtained as follows: Since $\gamma$ is parabolic, the displacement function $d_{\gamma}$ does not attain its infimum and therefore $\nabla d_{\gamma}$ is non-zero everywhere. So, one has a continuous vector field $-\nabla d_{\gamma}/|\nabla d_{\gamma}|$ which tells one how to ``flow towards the infimum of $d_{\gamma}$''. One takes an integral curve $g:[0,\infty)\ra\widetilde M$ of this vector field and checks that any accumulation point $\xi\in\overline{g([0,\infty))}\cap\partial_{\infty}$ is fixed by $\gamma$ and satisfies $\Td(\xi,\eta)\leq\pi/2$ for all $\eta\in\Fix(\gamma)$.  
\item
Second take the set of {\it all such points}
$$
B_{\gamma}=\{\xi\in\Fix(\gamma)\mid \Td(\gamma,\eta)\leq\pi/2 \mbox{ for all }\eta\in\Fix(\gamma)\}.
$$
This set is non-empty, has diameter $\leq\pi/2$ and is closed in the sphere topology. Therefore it has a unique Center, which we denote $\xi_{\gamma}$. 
\item
Since $B_{\gamma}$ has diameter $\leq\pi/2$, any two points $\xi,\xi'\in B_{\gamma}$ are connected by a unique geodesic $[\xi,\xi']$. The endpoints of this geodesic are fixed by $\gamma$ so uniqueness implies that the entire geodesic $[\xi,\xi']$ is fixed pointwise by $\gamma$. Moreover, CAT(1) comparison shows that for any point $\xi''$ on this geodesic $[\xi,\xi']$ we have $\Td(\xi'',\eta)\leq\pi/2$ for all $\eta\in\Fix(\gamma)$. Therefore $B_{\gamma}$ is convex and hence
$$
\xi_{\gamma}\in B_{\gamma}.
$$
\end{itemize}
\subsection{Centers of mass for fixed sets of abelian groups}
Since $\xi_{\gamma}$ is {\it the} center of mass of $\Fix(\gamma)$, it is fixed by anything that preserves this fixed set. So, it is fixed by the anything that commutes with $\gamma$. In particular, for any abelian group $A$ containing a parabolic element $\gamma$, the point $\xi_{\gamma}$ is contained in the set 
$$
B_A:=\{\xi\in\Fix(A)\mid\Td(\xi,\eta)\leq\pi/2 \mbox{ for all } \eta\in\Fix(A)\}.
$$ 
This set is $\Td$-convex, has $\Td$-diameter $\leq\pi/2$ and is closed in the sphere topology for the same reasons as $B_{\gamma}$. So, it has a unique Center which we denote by $\xi_A\in B_A$. In this way we have constructed a center of mass $\xi_A\in\Fix(A)$ for any abelian group $A$ containing a parabolic element. 

\begin{remark}
In Section \ref{goodpoints} we take this construction one step further and obtain a unique center of mass for any {\it virtual equivalence class of abelian groups $[A]$ containing a parabolic} by replacing the fix set $\Fix(A)$ in this construction with the countable union of fixed sets 
$$
\bigcup_{n\in\mathbb N}\Fix(n!A)
$$ 
of finite index subgroups $n!A=\{\gamma^{n!}\mid\gamma\in A\}$. The added difficulties involved are that this union may not be fixed by any single element and also that it is, possibly, no longer closed in the sphere topology. These difficulties are dealt with in Section \ref{goodpoints}.
\end{remark}

\section{Appendix B: The Karlsson-Margulis lemma (isometries of positive infimum displacement)}
In this appendix we describe a special case (Proposition \ref{karlsmarg} below) of the main theorem of \cite{karlssonmargulis} which is the ``nonpositively curved geometry'' part of their paper (as opposed to the ``ergodic theory part''). There is nothing new here, but we found it comforting to know that the proof of this special case is elementary and does not resort to any ergodic theory.

\subsection{Geodesic rays sublinearly tracking $\gamma$-orbits}
Let $\gamma$ be an isometry, pick a basepoint $y$, and let $y_n:=\gamma^ny$. The infimum displacement of $\gamma$ can computed as the limit 
\begin{equation}
\label{displacement}
A:=\lim_{n\ra\infty}{d(y,y_n)\over n}.
\end{equation}
\begin{proposition}
\label{karlsmarg}
Suppose $A>0$. Then there is a geodesic ray $c$ for which 
\begin{equation}
\label{km}
\lim_{k\ra\infty} {d(y_k,c(Ak))\over k}=0.
\end{equation}
\end{proposition}
In other words, the Proposition says that if the infimum displacement of $\gamma$ is positive, then there is a geodesic ray sublinearly tracking the positive $\gamma$-orbit $y,\gamma y,\gamma^2y,\dots$. It follows that the positive $\gamma$-orbit converges to a unique limit point 
$$
\lim_{n\ra\infty}\gamma^ny=c(\infty),
$$ 
but the sublinear tracking is stronger that just this statement alone. Also note that there is another (different) geodesic ray $c'$ that sublinearly tracks the negative $\gamma$-orbit $y,\gamma^{-1}y,\gamma^{-2}y,\dots$. In general the rays $c$ and $c'$ do not form a bi-infinite geodesic.
\begin{remark}
When $\gamma$ is a {\it hyperbolic} element, then the proposition is easy. If $y\in\widetilde M$ is a point at which $\gamma$ has minimum displacement then the positive $\gamma$-orbit of $y$ spans a geodesic ray $c$ and its negative $\gamma$-orbit spans a geodesic ray $c'$. The union of these rays is a bi-infinite geodesic which is an axis of $\gamma$. The point of the proposition is that some aspects of this nice situation are still true for parabolic $\gamma$ as long as the infimum displacement of $\gamma$ is positive.  
\end{remark}

\subsection{Finding good orbit points} 
Fix $\epsilon>0$. By the displacement formula (\ref{displacement}), there is $K=K_{\epsilon}$ such that for all $k\geq K$ 
\begin{equation}
\label{pinch}
(A-\epsilon)k\leq d(y,y_k)\leq(A+\epsilon) k.
\end{equation}
Since the sequence $\{d(y,y_n)-(A-\epsilon)n\}_n$ is {\it unbounded above}, there are \underline{arbitrarily large $n$} such that the $n$-th term of this sequence is larger than all its predecessors, i.e. for $0\leq k\leq n$
$$
d(y,y_{n-k})-(A-\epsilon)(n-k)\leq d(y,y_n)-(A-\epsilon)n.
$$
Since $\gamma$ is an isometry $d(y,y_{n-k})=d(y_k,y_n)$, so this can be rewritten as
\begin{equation}
\label{unbd}
(A-\epsilon)k\leq d(y,y_n)-d(y_k,y_n).
\end{equation}
The right hand side is $\leq d(y,y_k)$ because of the triangle inequality. Thus, for $K\leq k\leq n$ we get 
\begin{equation}
\label{unbounded}
(A-\epsilon)k\leq d(y,y_n)-d(y_k,y_n)\leq(A+\epsilon)k
\end{equation}
if $n$ satisfies (\ref{unbd}). Call $y_n$ an {\it $\epsilon$-good orbit point} if $n$ satisfies (\ref{unbd}). We emphasize that \underline{there are infinitely many $\epsilon$-good orbit points}, so we can always find a sequence $\{y_{n_i}\}$ of such points converging to a point at infinity. 
\subsection{Finite segments\label{finseg}} 
The heart of the proof is the following lemma. It says that for $\epsilon$-good orbit points $y_n$ the geodesic segments $[y,y_n]$ sublinearly track a large segment of the $\gamma$-orbit $y_K,y_{K+1},\dots,y_n$ up to an error determined by $\epsilon$. 
\begin{lemma}
For $\epsilon>0$ pick $K_{\epsilon}$ satisfying (\ref{pinch}) and $n=n_{\epsilon}$ satisfying (\ref{unbd}).
Let $c_n=[y,y_n]$ be the geodesic segment from $y$ to $y_n$ and $c_n(Ak)$ the point obtained by going for a distance $Ak$ along $c_n$. Then for any $K_{\epsilon}\leq k\leq n$
\begin{equation}
\label{bound}
{d(y_k,c_n(Ak))\over k}\leq 2\sqrt{A\epsilon}+\epsilon.
\end{equation}
\end{lemma}
\begin{figure}
\centering
\includegraphics[scale=0.55]{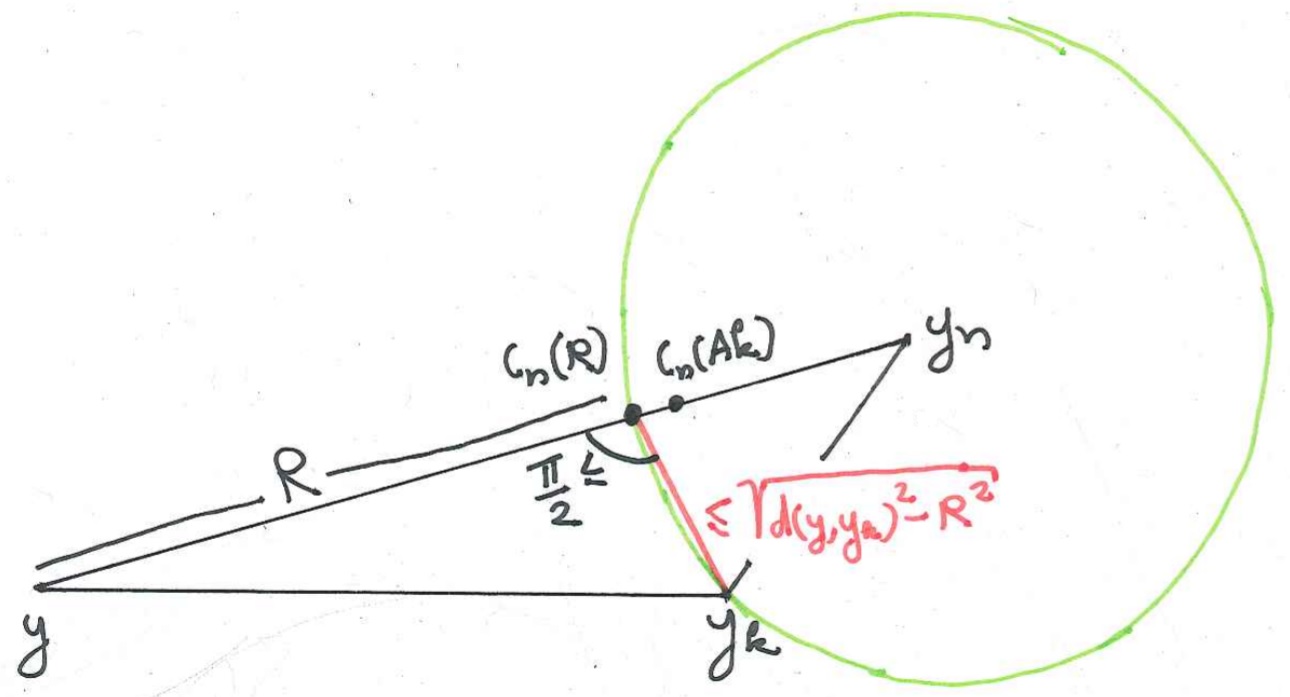}
\end{figure}
\begin{proof}
Let $R:=d(y,y_n)-d(y_k,y_n)$. Notice that 
$$
\angle_{c_n(R)}(y,y_k)\geq\pi/2
$$ 
because the geodesic segment $[c_n(R),y_k]$ is contained in the convex ball $\{d(\cdot,y_n)\leq d(y_k,y_n)\}$ and the segment $[c_n(R),y]$ is perpendicular to this ball. So, comparing with the corresponding obtuse Euclidean triangle gives 
\begin{eqnarray*}
d(y_k,c_n(R))&\leq&\sqrt{d(y,y_k)^2-R^2}\\
&\leq&\sqrt{(A+\epsilon)^2k^2-(A-\epsilon)^2k^2}\\
&=&2\sqrt{A\epsilon}k.
\end{eqnarray*}
Inequality (\ref{unbounded}) implies $|R-Ak|\leq\epsilon k$ so we conclude that 
\begin{eqnarray*}
d(y_k,c_n(Ak))&\leq& d(y_k,c_n(R))+d(c_n(R),c_n(Ak))\\
&\leq& (2\sqrt{A\epsilon}+\epsilon)k.
\end{eqnarray*}
\end{proof}
\subsection{Limits}
What remains is to take limits as $n\ra\infty$ and as $\epsilon\ra 0$ to get the desired geodesic ray $c$. We need to take some care in how we do this. 

Fix $\epsilon>0$ and let $\{y_{n_i}\}$ be a convergent sequence of $\epsilon$-good orbit points. Then the segments $c_{n_i}$ converge to a geodesic ray 
$$
c^{\epsilon}:=\lim_{i\ra\infty}c_{n_i}.
$$ 
The Lemma implies that inequality (\ref{bound}) holds for $n=n_i$ and all $K_{\epsilon}\leq k\leq n_i$. Therefore,  taking the limit of (\ref{bound})$_{n=n_i}$ as $i\ra\infty$, we get for all $k\geq K_{\epsilon}$ that 
\begin{equation}
\label{epsilon}
{d(y_k,c^{\epsilon}(Ak))\over k}\leq 2\sqrt{A\epsilon}+\epsilon.
\end{equation}
Now we vary $\epsilon$. By the triangle inequality we get from (\ref{epsilon}) that
\begin{equation}
\label{cauchy}
{d(c^{\epsilon}(Ak),c^{\epsilon'}(Ak))\over k}\leq (2\sqrt{A\epsilon}+\epsilon)+(2\sqrt{A\epsilon'}+\epsilon')
\end{equation}
when $k\geq\max\{K_{\epsilon},K_{\epsilon'}\}$. Therefore, any sequence $\{c^{\epsilon}(\infty)\}_{\epsilon\ra 0}$ is Cauchy in the $\angle_y$-metric, which implies that $c^{\epsilon}$ converges\footnote{Alternatively, we could have just picked {\it some} sequence $\epsilon_i\ra 0$ for which $c^{\epsilon_i}$ converges and let $c=\lim_{i\ra\infty} c^{\epsilon_i}$. Inequality (\ref{cauchy}) isn't really important for this step. It is crucial in the next step to get inequality (\ref{epsilonlimit}).} to a geodesic ray 
$$
c:=\lim_{\epsilon\ra 0}c^{\epsilon}.
$$
Next, notice that because the distance function $d(\cdot,\cdot)$ is convex, knowing (\ref{cauchy}) for arbitrarily large $k$ implies it for all $k$. So
$$
(\ref{cauchy}) \mbox{ holds for \underline{all} } k>0.
$$ 
Taking the limit as $\epsilon'\ra 0$ of inequality (\ref{cauchy}),\footnote{For this it is important that (\ref{cauchy}) is valid for all $k$ independent of $\epsilon'$.} we get that for all $k>0$
\begin{equation}
\label{epsilonlimit}
{d(c^{\epsilon}(Ak),c(Ak))\over k}\leq 2\sqrt{A\epsilon}+\epsilon.
\end{equation}
Finally, putting this together with (\ref{epsilon}) gives
$$
{d(y_k,c(Ak))\over k}\leq 2(2\sqrt{A\epsilon}+\epsilon).
$$
for all $k\geq K_{\epsilon}$. This proves the proposition.

\section{Appendix C: Invariant horospheres and convex combinations of displacement functions}
In this appendix we will describe another way of finding ``centers at infinity'' for an abelian group $A<\Gamma$ containing parabolic elements that is different from the method described in Section \ref{goodpoints}. The idea is to start with (infinite) convex combinations of displacement functions and deform them to Busemann functions via the limit process described in Subsection \ref{invarianthorospheres}. The ``output'' Busemann functions produced in this way will be just as invariant as the ``input'' displacement functions. It is, of course, not a priori clear why {\it infinite} convex combinations of displacement functions are any better than finite ones, but it turns out that using infinite combinations makes it easy to compare the ``outputs'' for adjacent abelian groups $B<A$. 

In order for an infinite sum of displacement functions to converge, we need the coefficients to decay sufficiently fast. To specify how fast, fix a finite generating set $\Gamma=\left<\gamma_1,\dots,\gamma_r\right>$ and denote by $||\cdot ||$ the word length for this generating set. Let $c$ be a constant $>\log(2r)$ and set
\begin{equation}
\omega(\gamma):=e^{-c||\gamma||}.
\end{equation}
We will show below that the infinite series of displacement functions
$$
f:=\sum_{\gamma\in\Gamma}\omega(\gamma)d_{\gamma}
$$
converges pointwise and that for any {\it abelian} subgroup $A<\Gamma$, the subseries
$$
f_A:=\sum_{\gamma\in A}\omega(\gamma)d_{\gamma}
$$
has infimum given by
$$
\inf f_A=\sum_{\gamma\in A}\omega(\gamma)|\gamma|.
$$
$f_A$ is convex and invariant under all isometries commuting with $A$. Writing 
\begin{equation}
\label{positive}
f_A-\inf f_A=\sum_{\gamma\in A}\omega(\gamma)(d_{\gamma}-|\gamma|),
\end{equation}
as an infinite sum of non-negative functions $\omega(\gamma)(d_{\gamma}-|\gamma|)$ shows that {\it if at least one of the displacement functions $d_{\gamma}$ does not attain its infimum $|\gamma|$ then $f_A$ also does not attain its infimum.} In other words,
$$
\mbox{ if } A \mbox{ has a parabolic then } f_A \mbox{ doesn't attain its infimum.}
$$

{\bf Recall how we found invariant horospheres in Subsection \ref{invarianthorospheres}:} Fix a basepoint $z\in \widetilde M$. For a convex function $h$ that does not attain its infimum, let $x_s$ be the closest point to $z$ on the sublevel set $\{h\leq s\}$. Let 
$$
\partial h:=\overline{\{x_s\}_{s\in\mathbb R}}\cap\partial_{\infty}
$$
be the set of accumulation points of $\{x_s\}$. It is not hard to see that $\partial h$ does not depend on the choice of basepoint $z$. If $x_{s_i}\ra \xi\in\partial h$, then the limit
$$
\hat h(x):=\lim_{i\ra\infty}d(x,\{h\leq s_i\})-d(z,\{h\leq s_i\})
$$ 
exists, and is equal to a Busemann function centered at $\xi$. (See 3.9 of \cite{ballmangromovschroeder}). This function has the same invariance properties as the ``input'' $h$. 

{\bf We apply this to the functions $f_A$.} Since $f_A$ is $C_A$-invariant, any Busemann function $\hat f_A$ obtained by this process is also $C_A$-invariant. Thus, 
$$
\mbox{horospheres centered at points of } \partial f_A \mbox{ are } C_A\mbox{-invariant}.
$$ 
Moreover, if $B<A$ is a subgroup of $A$ then because of (\ref{positive}) we have
$$
f_B-\inf f_B\leq f_A-\inf f_A,
$$
which shows that {\it every sublevel set of $f_B$ contains a sublevel set of $f_A$.} It is not hard to see from this that the diameter of the set $\partial f_B\cup\partial f_A$ in the $\angle_z$-metric is $\leq\pi/2$. Since these sets do not depend on the basepoint $z$, the same is true for the diameter in the Tits metric. More generally, for a chain 
$$
\mbox{parabolic }\gamma\in
A_0<\dots<A_k
$$ 
of abelian subgroups containing a parabolic element, 
\begin{equation}
\label{diameterbound}
\mbox{diam}(\partial f_{A_0}\cup\dots\cup\partial f_{A_k})\leq\pi/2.
\end{equation}
So, for some purposes the points in $\partial f_A$ are good enough and can be used in place of the points $\xi_A$ described in Section \ref{goodpoints}. Their construction is in many ways more elementary and direct than the one in Section \ref{goodpoints}. It does not use the two-step center-of-mass procedure and one sees that the resulting horospheres are $C_A$-invariant directly, without using the Karlsson-Margulis lemma. But, the main drawback of this construction is that it is not invariant enough (it depends on the choice of a generating set of $\Gamma$) so it does not give a {\it $\Gamma$-equivariant} map $\beta$ to the boundary at infinity. Additionally, for a point $\xi\in\partial f_A$ one only knows that $\Td(\xi,\eta)\leq\pi/2$ for {\it some}\footnote{For instance, if $\eta\in\partial f_{A'}$ for $A'<A$ or $A<A'$ then $\Td(\xi,\eta)\leq\pi/2$ by (\ref{diameterbound}).} points $\eta$ fixed by $A$, which makes dealing with finite index issues more awkward.

In the rest of this section we prove convergence and the infimum formula.   
\subsection{Fast decay} 
The number of elements of word length $=R$ is at most $(2r)^R=e^{R\log(2r)}$ since there are $2r$ choices (the generators and their inverses) for each letter in a word of length $R$. The decay estimate we need is
\begin{eqnarray}
\label{decay}
\sum_{\gamma\in\Gamma}\omega(\gamma)||\gamma||&=&\sum_{R=0}^{\infty}\sum_{||\gamma||=R}e^{-cR}R\\
&\leq&\sum_{R=0}^{\infty}e^{\log(2r)\cdot R}e^{-cR}\cdot R\\
&<&\infty.
\end{eqnarray} 
\subsection{Convergence of the series} We show that the infinite series
$$
f(x)=\sum_{\gamma\in\Gamma}\omega(\gamma)d_{\gamma}(x)
$$
converges for every point $x\in\widetilde M$. 
Express an element $\gamma$ in the generating set $\gamma_1,\dots,\gamma_r$ as $\gamma=\gamma^{\pm}_{i_1}\cdots\gamma^{\pm}_{i_k}$ in the shortest possible way (so $||\gamma||=k$). 
At a point $x\in\widetilde M$, the triangle inequality implies
$$
d_{\gamma}(x)=d_{\gamma^{\pm}_{i_1}\cdots\gamma^{\pm}_{i_k}}(x)\leq d_{\gamma_{i_1}}(x)+\dots+d_{\gamma_{i_k}}(x)\leq||\gamma||\max_{i=1}^rd_{\gamma_i}(x).
$$ 
So, the $R$-tail of the infinite series $f(x)$ is bounded by
\begin{equation}
\label{tailbound}
T_R(x):=\sum_{||\gamma||>R}\omega(\gamma)d_{\gamma}(x)\leq\left(\sum_{||\gamma||>R}\omega(\gamma)||\gamma||\right)\max_{i=1}^r d_{\gamma_i}(x).
\end{equation}
So, decay estimate (\ref{decay}) ensures the series $f(x)$ converges for all $x\in\widetilde M$. 
\begin{remark}
Note that it follows from (\ref{tailbound}) and (\ref{decay}) that 
$$
\lim_{R\ra\infty}T_R(x)=0,
$$ 
and therefore also 
\begin{equation}
\label{rtail}
\lim_{R\ra\infty}\inf T_R=0.
\end{equation} 
We will use this remark at the end of the proof of the infimum formula.
\end{remark}
\subsection{The infimum of $f_A$ for an abelian group $A$} We will show that
\begin{equation}
\label{infiniteinf}
\inf f_{A}=\sum_{\gamma\in A}\omega(\gamma)|\gamma|.
\end{equation}
Here is the main observation we exploit: Suppose $F=\{\gamma_1,\dots,\gamma_n\}$ is a {\it finite} set of {\it commuting} isometries, and $T$ is an $F$-invariant convex function. Then we can ``simultaneously infimize'' all the displacement functions $\{d_{\gamma}\}_{\gamma\in F}$ {\it and} the function $T$.
Here is the reason: Let $C_0=\{T\leq c_0\}$ and $C_i=\{d_{\gamma_i}\leq c_i\}$ be non-empty sublevel sets of $T$ and $d_{\gamma_i}$, and let $p_i:\widetilde M\ra C_i$ be the closest point projections to these sublevel sets. They are all $F$-equivariant contractions. Notice that the point $z=p_0\circ p_1\circ\dots\circ p_n(x)$
\begin{itemize}
\item
is moved $\leq c_i$ by the element $\gamma_i$, because it is the image of an element\footnote{Namely, the element $p_i\circ\dots\circ p_n(x)\in C_i$.} in $C_i$ by the $\gamma_i$-equivariant contraction $p_0\circ p_1\circ\dots\circ p_{i-1}$, and
\item
is contained in $C_0$, since it is in the image of $p_0$.
\end{itemize}
Therefore the intersection of sublevel sets $C_0\cap C_1\cap\dots\cap C_n$ is nonempty (it contains $z$). Since we can do this for $c_i$ as close to the infima as we want, we conclude that 
$$
\inf \left(\sum_{i=1}^n\omega(\gamma_i)d_{\gamma_i}+T\right)=\sum_{i=1}^n\omega(\gamma_i)\inf d_{\gamma_i}+\inf T.
$$

Apply this when $F$ is the set of $R$-small elements $\{\gamma\in A\mid ||\gamma||\leq R\}$ and $T$ is the $R$-tail of the series $f_{A}$, i.e. 
$$
T=T_{R,A}:=\sum_{\gamma\in A,||\gamma||>R}\omega(\gamma)d_{\gamma}
$$ 
to conclude 
\begin{eqnarray*}
\inf\left(\sum_{\gamma\in A}\omega(\gamma)d_{\gamma}\right)&=&
\inf\left(\sum_{\gamma\in A,||\gamma||\leq R}\omega(\gamma)d_{\gamma}+T_{R,A}\right)\\
&=&\sum_{\gamma\in A,||\gamma||\leq R}\omega(\gamma)\inf d_{\gamma}+\inf T_{R,A}.
\end{eqnarray*}
It follows from (\ref{rtail}) that $\lim_{R\ra\infty}\inf T_{R,A}=0$, so taking the limit as $R\ra\infty$ proves the infimum formula. 
\begin{remark}
Let us emphasize that we are {\it always} using the word length in $\Gamma$ and never the word length in $A$ to measure distances in the proof of convergence and the inf formula above. 
\end{remark}
\section{Appendix D: $\delta$-deforming subsets to subcomplexes}
Let $K\subset\widetilde M$ be a compact subset of Hausdorff dimension $d$. For any $\delta>0$ we can $\delta$-deform it to a $d$-dimensional subcomplex. In this appendix we will recall this argument and also explain why it applies when instead of a compact set we have a countable union of compact sets $\cup_i K_i$. Then, we will explain how to get the same conclusion only assuming that $K$ has topological dimension $d$ (instead of Hausdorff dimension $d$). 

The argument is by induction on skeleta of a fine triangulation of $\widetilde M$. In the course of the induction the inclusion $i:K\hookrightarrow\widetilde M$ will be $\delta$-deformed to maps $f:K\ra\widetilde M$ that are not necessarily inclusions. So, it is better to state the result like this.
\begin{proposition}
\label{deformlipschitz}
Let $K$ be a compact set and suppose that $f:K\ra\widetilde M$ is a map whose image $f(K)$ has Hausdorff dimension $d$. 
Then for every $\delta>0$ the map $f$ can be $\delta$-deformed to a map $\hat f$ whose image is contained in a $d$-dimensional subcomplex of $\widetilde M$. 

More generally, the same is true if $K=\cup_iK_i$ is a countable union of compact sets.
\end{proposition}
Let $P$ be a $\delta$-fine triangulation of $\widetilde M$. The main step is to show, for every $k>d$, that if $f(K)$ is contained in the $k$-skeleton $P^{(k)}$ then we can deform via a family $f_t:K\ra P^{(k)}$ to a map $f_1:K\ra P^{(k-1)}$ to the $(k-1)$-skeleton in such a way that
during the deformation points of each closed simplex $\sigma$ of $P^{(k)}$ remain in that simplex, i.e. 
$$
f_t(\sigma)\subset\sigma,
$$
and
$$
\dim(f_1(K))\leq\dim(f(K))=d.
$$
Once one proves this main step, iterating it proves the Proposition. 

Now we prove the main step. Note that since $k$ is greater than the Hausdorff dimension of the image $f(K)$, for every $k$-simplex $\sigma$ in $P^{(k)}$ there is a point $x\in\mbox{Interior}(\sigma)-f(K)$ of the interior of $\sigma$ that is not contained in the image $f(K)$. 
First let's deal with the case when \underline{$K$ is compact}. In this case there is a (maybe very small) $\epsilon_x$-neighborhood of $x$ that is disjoint from $f(K)$. Therefore, by radially projecting away from $x$ on the simplex $\sigma$ (and not doing anything outside $\sigma$) we get a map 
$$
r^{\sigma}_1:(P^{(k)}-\{x\})\ra (P^{(k)}-\mbox{Interior}(\sigma))
$$ 
that is Lipschitz on $f(K)$, so it does not increase Hausdorff dimension, i.e. 
$$
\dim(r_1^{\sigma}\circ f(K))\leq\dim(f(K))=d.
$$
This map is part of an obvious radial projection homotopy $r^{\sigma}_t$ that is the identity outside the interior of $\sigma$. Doing such a deformation on each $k$-simplex $\sigma$ of $P^{(k)}$ gives the desired deformation of $f_t$ of $f$ into the $(k-1)$-skeleton $P^{(k-1)}$. This deformation preserves simplices and does not increase Hausdorff dimension of $f(K)$ because this is true for all the individual deformations $r^{\sigma}_t$. This finishes the proof of the main step when $K$ is compact. 

Next, we explain why the same argument applies in general when $K$ is a countable union of compact sets $\cup_i K_i$. As before, we conclude that the radial projection $r^{\sigma}_1$ is Lipschitz on each compact set $f(K_i)$, but the Lipschitz constant might depend on $i$. However, this doesn't matter because we still get for each individual compact set that   
$$
\dim(r_1^{\sigma}\circ f(K_i))\leq\dim(f(K_i))\leq d
$$
and therefore we get the same bound
$$
\dim(r_1^{\sigma}\circ f(\cup_iK_i))\leq d
$$
for the countable union. Therefore, the main step works for countable unions of compact sets $\cup_iK_i$.
\subsection*{A topological version}
There is also a topological version of this, which says that if the topological dimension of $K$ is $\leq d$ then we can $\delta$-deform $K$ to a $d$-dimensional subcomplex of $\widetilde M$. The simplest way to arrive at this version is to first $\delta$-deform $K$ to a map $f'$ whose image has Hausdorff dimension $\leq d$ and then apply the proposition to $f'$. So, all we need is the following lemma. 
\begin{lemma}
Suppose that $f:S\hookrightarrow \widetilde M$ is the inclusion of a subset of {\it topological} dimension $\leq d$. For any $\delta>0$, we can $\delta$-deform it to a map $f':S\ra\widetilde M$ whose image has Hausdorff dimension $\leq d$.
\end{lemma}
\begin{proof}
Start with a cover of $S$ by $\delta$-balls in $\widetilde M$. Since $S$ has topological dimension $\leq d$, this cover has a refinement whose nerve $N$ has dimension $\leq d$. Build a map $S\ra N$ to this nerve using a partition of unity and a map from the nerve $N\ra\widetilde M$ using geodesic simplices in $\widetilde M$. The composition $S\ra N\ra\widetilde M$ is then $\delta$-close to $f$ and its image has Hausdorff dimension $\leq d$ because $N$ is $d$-dimensional and the map $N\ra\widetilde M$ built using geodesic simplices is Lipschitz.  
\end{proof}
\begin{remark}
One can try to prove the topological version directly by mimicking the proof of Proposition \ref{deformlipschitz} without first deforming to a map with Hausdorff dimension $\leq d$. The difficultly with doing this direct argument is that it is not clear that the map $r^{\sigma}_1$ does not raise topological dimensions of subsets of $P^{(k)}-\{x\}$.  
\end{remark}

\bibliographystyle{amsplain}
\bibliography{moreends}

\end{document}